\documentclass[11pt,a4paper]{article}

\usepackage[utf8]{inputenc}

\usepackage{fullpage}
\usepackage{titlesec}

\usepackage{authblk}

\usepackage[usenames,dvipsnames]{xcolor}
\usepackage{graphicx}
\usepackage{subcaption}
\usepackage{tikz}
\usetikzlibrary{calc,positioning}
\usepackage{pgfplots}
\pgfplotsset{compat=1.18}
\pgfdeclarelayer{bg}
\pgfsetlayers{bg,main}

\usepackage{amssymb,dsfont}

\usepackage{amsmath}
\usepackage{amsthm}
\usepackage{mathtools}

\usepackage{multirow}
\usepackage{booktabs}
\usepackage{longtable}

\usepackage{enumerate}

\usepackage{hyperref}
\usepackage[capitalise]{cleveref}

\usepackage[linesnumbered,ruled]{algorithm2e}
\crefname{algocf}{Algorithm}{Algorithms}
\Crefname{algocf}{Algorithm}{Algorithms}

\usepackage{xspace}
\xspaceremoveexception{-}

\newcommand{\dotcup}{\mathbin{\dot\cup}}

\newcommand{\arr}{\mathrm{arr}}
\newcommand{\dep}{\mathrm{dep}}
\newcommand{\drive}{\mathrm{drive}}
\newcommand{\wait}{\mathrm{wait}}
\newcommand{\transfer}{\mathrm{trans}}
\newcommand{\energy}{\mathrm{energy}}
\newcommand{\energyact}{M}               
\newcommand{\possenergyact}{A_{\energy}} 

\newcommand{\ac}{\mathrm{ac}}
\newcommand{\br}{\mathrm{br}}
\newcommand{\opt}{\mathrm{opt}}
\newcommand{\length}{\operatorname{length}}
\newcommand{\pesp}{PESP\xspace}
\newcommand{\pesppassenger}{PESP-Passenger\xspace}
\newcommand{\pespenergy}{PESP-Energy\xspace}
\newcommand{\pesppassenergy}{PESP-Passenger-Energy\xspace}
\newcommand{\osn}{one-station network\xspace}
\newcommand{\w}{\rho} 

\colorlet{mygreen}{green!40!black}
\colorlet{myblue}{blue!70!black}
\colorlet{myred}{red!70!black}
\definecolor{brown}{rgb}{0.8,0.43,0.2}
\definecolor{DklCyan}{rgb}{0.12,0.47,0.58}

\newcommand{\R}{\mathbb R}
\newcommand{\Z}{\mathbb Z}

\newcommand{\cL}{\mathcal{L}}
\newcommand{\cA}{\mathcal{A}}
\newcommand{\cE}{\mathcal{E}}

\newcommand{\argmin}{\operatorname{arg\,min}}

\newtheorem{theorem}{Theorem}[section]
\newtheorem{proposition}[theorem]{Proposition}
\newtheorem{corollary}[theorem]{Corollary}
\newtheorem{lemma}[theorem]{Lemma}
\theoremstyle{definition}
\newtheorem{definition}[theorem]{Definition}
\theoremstyle{remark}
\newtheorem*{remark}{Remark}

\tikzset{
	>=stealth,
	station/.style={lightgray, draw=black},
	event/.style={circle, draw=black, very thick, minimum size = 1.5em, inner sep=0pt, font = \scriptsize},
	activity/.style={->, ultra thick},
	energy/.style={activity, red},
	brake/.style={activity, BurntOrange},
	accelerate/.style={activity, ForestGreen},
	Arr/.style={event, fill=white},
	Dep/.style={event, fill=blue!20!white}
}

\definecolor{S1}{HTML}{DB6AA2}
\definecolor{S2}{HTML}{007939}
\definecolor{S3}{HTML}{0068A9}
\definecolor{S41}{HTML}{A74329}
\definecolor{S42}{HTML}{C36324}
\definecolor{S45}{HTML}{C88B39}
\definecolor{S5}{HTML}{ED7508}
\definecolor{S7}{HTML}{816DA7}
\definecolor{S75}{HTML}{816DA7}
\definecolor{S8}{HTML}{65AA2F}
\definecolor{S9}{HTML}{9A2846}
\newcommand{\oktrainlabel}[2]{\raisebox{-.2em}{\includegraphics[width=.8em]{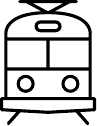}} #1 $\triangleright$ #2}
\newcommand{\oktrainx}[4]{\fill[fill=#4] ({#2},{#1-0.25}) rectangle (#3,{#1+0.25});}
\newcommand{\oktrainex}[4]{%
    \oktrainx{#1}{#2}{#3}{#4}
    \oktrainx{#1}{#2-0.5}{#2}{BurntOrange!50}
    \oktrainx{#1}{#3}{#3+0.4}{ForestGreen!50}}
\newcommand{\oktrain}[5]{%
    \node[anchor=east, font=\footnotesize] at (0,#1) {\oktrainlabel{#4}{#5}};
    \oktrainx{#1}{#2}{#3}{#4}}
\newcommand{\oktraine}[5]{%
    \oktrain{#1}{#2}{#3}{#4}{#5}
    \oktrainx{#1}{#2-0.5}{#2}{BurntOrange!50}
    \oktrainx{#1}{#3}{#3+0.4}{ForestGreen!50}
    }
\newcommand{\okoverlap}[4][0.4]{%
    \fill[black!30, fill opacity=0.5] (#4, #2) rectangle (#4+#1, #3);
    }

\makeatletter
\DeclareRobustCommand{\labelcrefrange}[2]{\@crefrangenostar{labelcref}{#1}{#2}}
\makeatother

\makeatletter
\renewcommand{\@fnsymbol}[1]{\@arabic{#1}}
\makeatother

\author[a]{Sarah Roth\thanks{s.roth@math.rptu.de}}
\author[a]{Sven Jäger\thanks{Present Address: Institute of Mathematics, TU Berlin, jaeger@math.tu-berlin.de}}
\author[b]{Niels Lindner\thanks{lindner@zib.de}}
\author[a,c]{Anita Schöbel\thanks{anita.schoebel@rptu.de}}
\affil[a]{Department of Mathematics, RPTU Kaiserslautern-Landau,\protect\\ Paul-Ehrlich-Str.\ 31, 67663 Kaiserslautern, Germany}
\affil[b]{Institute of Mathematics, Freie Universität Berlin, Germany,\protect\\ c/o Zuse Institute Berlin, Takustr.\ 7, 14195 Berlin, Germany}
\affil[c]{Fraunhofer Institute for Industrial Mathematics ITWM, \protect\\
Fraunhofer-Platz 1, 67663 Kaiserslautern, Germany}

\title{Optimizing Travel~Time and Regenerative~Energy for~Periodic~Timetables}

\begin{document}  
\maketitle

{\abstract{
Regenerating braking energy is one major pathway to make rail traffic energy-efficient. It is therefore desirable to design timetables that exploit this feature. However, timetables that allow to regenerate energy are often bad for the passengers. We hence formulate and analyze a bicriteria optimization problem (\pesppassenergy) to find periodic railway timetables that maximize the regenerated energy in terms of the brake-traction overlap time and minimize the travel time of the passengers. Our model extends the Periodic Event Scheduling Problem (PESP) and offers a rich combinatorial theory. We investigate its computational complexity on one-station networks, building on matchings and Hamiltonian paths. Besides showing its NP-hardness even for a single objective, we identify several polynomial-time solvable special cases. Finally, we provide two case studies, underlining the practicability of our model, and analyzing the Pareto front.
}}


\clearpage
\tableofcontents

\section{Introduction}
\label{sec-motivation}

In order to reach the climate goals, it is necessary to strengthen the role of public transportation in passenger transport. However, also public transport systems themselves consume a large amount of energy. In rail traffic, traction energy for electric vehicles accounts for up to 80\% of the total consumption \cite{douglas_assessment_2015}.
This motivates to design energy-efficient train trajectories. In this paper, we exploit that modern electric motors are able to regenerate energy while braking. For example, Deutsche Bahn reports that 18.7\% of the traction energy could be recovered in the year 2024 \cite{bahn}. For trains, especially in low voltage direct current (DC) systems, the most efficient way to use the regained energy is a transmission to a close-by accelerating train via the catenary~\cite{Scheepmaker2017}. Such systems are for example used in many metro and suburban rail systems, but also for long-distance railways in several countries such as the Netherlands, Russia, Italy, or Spain.

From the perspective of planning, it is therefore sensible to schedule train timetables in a way that synchronizes braking and acceleration processes of nearby trains. Such a schedule has two advantages concerning the energy usage: First, it enables a maximum usage of the regenerated energy and, hence, reduces the total amount of energy that needs to be purchased by the public transport operator. Second, it prevents power peaks that might surcharge the transportation system's power supply.

However, from a passenger perspective, this synchronization of braking and acceleration is undesirable, as it prevents a passenger transfer from the braking to the accelerating train. Narrowly missing a train leads to frustration of the passengers, and systematically long waiting times might cause them to prefer the car over public transport.

As an illustration of this trade-off, we consider an example motivated by the Swiss railways. The trains are scheduled by means of an integrated fixed-interval timetable \cite{KräuchiChristian2004MZfd}, which is designed to prevent narrowly missed connections. At each station, the timetable prescribes a fixed time such that shortly before this time, all trains stopping at the station arrive, and the trains depart shortly after that time, see \cref{fig:example-basel-a}. This enables short transfer times from and to all directions. On the other hand, a transfer of regenerative braking energy from one train to another is impossible. For this objective, an efficient timetable would schedule the trains one after another such that the braking and acceleration phases overlap pairwise, see \cref{fig:example-basel-b}.

\begin{figure}[ht]
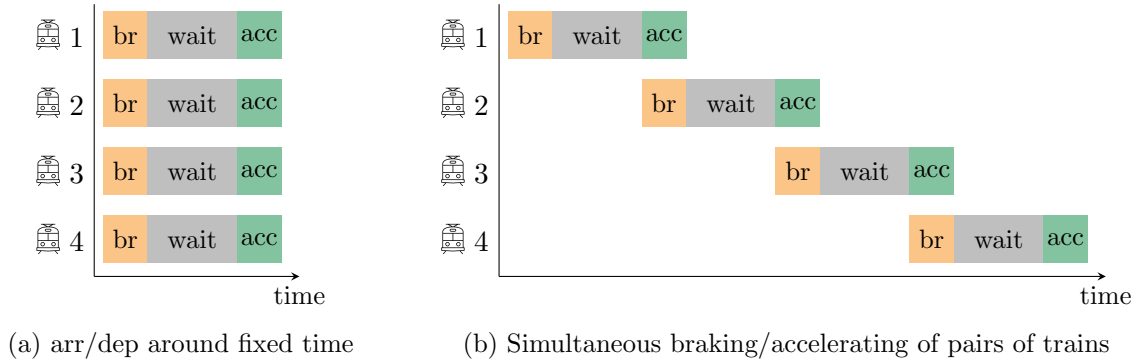

    \centering
	\begin{subfigure}{0.33\textwidth}
		\centering
		\begin{tikzpicture}[xscale=.59,yscale=.9]
            \tikzstyle{n} = [font=\small]
			\draw (-.2,-4.2) -- (-.2,-0.2);
			\foreach \i in {1,...,4} {
				\fill[fill=BurntOrange!50] (0,-\i) rectangle node[n] {br} (1,{.7-\i});
				\fill[fill=lightgray] (1,-\i) rectangle node[n] {wait} (3,{.7-\i});
				\fill[fill=ForestGreen!50] (3,-\i) rectangle node[n] {acc} (4,{.7-\i});
				\node[left] at (-.2,{.35-\i}) {\raisebox{-.1em}{\includegraphics[width=.8em]{train-icon.pdf}} \i};
 			}
			\draw[->] (-.2,-4.2) -- +(4.6,0) node[n, below] {time};
		\end{tikzpicture}
		
		\caption{arr/dep around fixed time}
		\label{fig:example-basel-a}
	\end{subfigure}
	\begin{subfigure}{0.66\textwidth}
		\centering
		\begin{tikzpicture}[xscale=.59,yscale=.9]
            \tikzstyle{n} = [font=\small]
			\draw (2.8,-4.2) -- (2.8,-0.2);
			\foreach \i in {1,...,4} {
				\fill[fill=BurntOrange!50] (3*\i,-\i) rectangle node[n] {br} +(1,.7);
				\fill[fill=lightgray] ({3*\i+1},-\i) rectangle node[n] {wait} +(2,.7);
				\fill[fill=ForestGreen!50] ({3*\i+3},-\i) rectangle node[n] {acc} +(1,.7);
				\node[left] at (2.8,{.35-\i}) {\raisebox{-.1em}{\includegraphics[width=.8em]{train-icon.pdf}} \i};
			}
			\draw[->] (2.8,-4.2) -- +(13.6,0) node[n, below] {time};
		\end{tikzpicture}
		\caption{Simultaneous braking/accelerating of pairs of trains}
		\label{fig:example-basel-b}
	\end{subfigure}
	
	\caption{Timetable patterns of braking, waiting and acceleration phases of four trains, see \cite{atmos}}
	\label{fig:example-basel}
\end{figure}

While it is beneficial to the environment to use as much of the regenerative braking energy as possible, it is also of utmost importance to provide attractive public transportation to the passengers. This was formulated for periodic timetables as a bicriteria optimization problem by Wang, Zhu, and Corman~\cite{Wang2022} and Wang, Bešinović, Goverde, and Corman~\cite{Wang2022a}.

\paragraph*{Our contribution.}
We summarize our key contributions as follows:

\begin{enumerate}
	\item We propose a mathematical framework and a mixed-integer programming formulation for the bicriteria discrete optimization problem of finding a feasible periodic timetable maximizing the brake-traction overlap and minimizing the passengers' travel time on an arbitrary event-activity network (\pesppassenergy), extending the Periodic Event Scheduling Problem (\pesp).
    \item For timetables minimizing travel time (\pesppassenger) or maximizing the brake-traction overlap (\pespenergy) individually, we investigate the complexity status of different problem variants on networks modeling a single station. We show that the (classic) PESP-Passenger is already NP-hard in the case of one station. Moreover, we characterize the structure of optimal solutions and link them to matchings and Hamiltonian paths, resulting in a polynomial-time algorithm with additive performance guarantee. We further identify several classes of polynomially solvable cases.
    \item On two case studies, one being a major hub of the Berlin S-Bahn network, we demonstrate that our bicriteria model is able to handle realistic scenarios, and we compute and analyze the Pareto front.
\end{enumerate}

This paper is based on a previous conference proceedings paper~\cite{atmos} for which it significantly extends the complexity analysis and algorithms. The Berlin S-Bahn case study is new.

\paragraph*{Structure of the paper.}
We begin with a literature review in \cref{sec-literature}. Our bicriteria periodic timetabling problem \pesppassenergy is described in \cref{sec-prob-descr}, along with our modeling of the brake-traction overlap, and our mixed-integer programming formulation. In \cref{sec-compl-overview}, we define one-station networks and provide an overview of our computational complexity results, depending on the choice of objective, the structure of possible transfers, and restrictions on train dwelling times. The following sections are devoted to understanding the structure of optimal solutions: 
We focus first on minimizing the total travel time of the passengers in a one-station network in \cref{sec:one-transfer}. For maximizing the brake-traction overlap, we relate our setting to bipartite matchings, Hamiltonian paths, and the Gilmore-Gomory problem of sequencing a one state-variable machine in \cref{sec:pesp-energy-one-station}. This allows us to clarify the complexity status of the problem for many situations. For the remaining case of unrestricted dwelling activities on a complete one-station network, we present polynomially solvable cases for special choices of acceleration and braking times in \cref{sec-Special-Cases}, while \cref{sec-Variants} deals with restrictions on the number of cycles. We present our two case studies in \cref{sec:numerics}, before concluding the paper in \Cref{sec:outlook}.

\section{Literature}
\label{sec-literature}

The task of designing efficient railway timetables by means of the Periodic Event Scheduling Problem (PESP) has been subject to study at least since the seminal work by Serafini and Ukovich in 1989~\cite{Serafini1989}. An overview can be found in Liebchen's dissertation~\cite{Liebchen2006}, and more recent approaches are illustrated in, e.g., \cite{borndorfer_concurrent_2020,grafe_et_al:OASIcs.ATMOS.2021.9,bortoletto_tropical_2022}. Traditionally, the literature on timetabling has focused on minimizing the passengers' travel time. More recently, however, the growing importance of energy efficiency has led to significant research attention on reducing energy consumption, particularly within the engineering sciences. A comprehensive review of these works goes far beyond the scope of this paper. Instead, we only mention the most relevant papers in our context and refer to the survey by Scheepmaker, Goverde, and Kroon~\cite{Scheepmaker2017} and the exemplary recent papers~\cite{Mo2020,Liao2021,
Huang2025}, which contain more extensive literature reviews.

There are two ways in which the timetable can affect the trains' energy consumption. On the one hand, there is the idea of saving energy by the implementation of energy-efficient driving strategies, see, e.g., Howlett, Milroy, and Pudney~\cite{Howlett1994}. These depend on the time scheduled for each driving section; typically longer travel times require less energy. Ghoseiri, Szidarovszky, and Asgharpour~\cite{Ghoseiri2004} considered a biobjective train scheduling model, combining the objectives of minimizing energy and minimizing travel time, and approximate the Pareto frontier using the $\varepsilon$-constraint method. 

On the other hand, the timetable can influence the usage of regenerative energy in train systems. This was first researched by Ramos, Pena, Fernández, and Cucala~\cite{Ramos2008}, who allow a modification of the dwell times to increase the brake-traction overlap. A more detailed modeling of the energy consumption that combines the driving strategies and the brake-traction overlaps has been studied by Yin, Yang, Tang, Gao, and Ran~\cite{Yin2017}, who devised a Lagrangian relaxation-based heuristic for this problem. In contrast, Gupta, Tobin, and Pavel~\cite{Gupta2016} considered a very simplified linear programming model to synchronize the start times of braking phases and the end times of acceleration phases. Bärmann, Gemander, Martin, and Merkert~\cite{barmann_energy-efficient_2021} report on implementing an energy-efficient timetable by optimizing both speed profiles and brake-traction overlaps for the Nuremberg subway.
Finally, Sun, Yao, Dong, and Clarke~\cite{Sun2023a} proposed a method to adjust speed profiles during operation in order to achieve brake-traction overlaps even in the case of delays.

The bicriteria problem of minimizing the passenger travel times and maximizing the brake-traction overlap has been investigated by Yang, Ning, Li, and Tang~\cite{Yang2014}, who developed a genetic algorithm. Moreover, Yang, Liao, Wu, Timmermans, Sun, and Gao~\cite{Yang2020} applied the NSGA-II algorithm for approximating the Pareto frontier. Gong, Luan, Yang, Qi, and Corman~\cite{gong2024} presented a non-linear integer programming formulation and an adaptive large-neighborhood search for the bicriteria problem integrating rolling stock circulations. %

All these works are based on a given aperiodic timetable that can be modified. Only recently, the study of the periodic version of this problem has been initiated by Wang, Zhu, and Corman~\cite{Wang2022a}. They assume a given nominal periodic timetable and develop a model that can be used for local adjustments. On the one hand, their aim is to maximize the brake-traction overlap to enable the usage of regenerative energy on a fixed set of synchronized arrival and departure events. On the other hand, they include passenger-related objectives such as the minimization of the generalized average travel time of all passengers and the minimization of the maximum individual increase in generalized  travel time. Wang et al.\ also provide a visualization of the Pareto frontier for these objectives on an instance of Dutch railways.

\section{Problem Description}
\label{sec-prob-descr}

We consider an extension of the Periodic Event Scheduling Problem (PESP), whose basic notions are recalled in \Cref{ss:prob-input}, including our construction of event-activity networks. We introduce our model of brake-traction overlaps in \Cref{ss:overlap}. This allows us to define the baseline problem of this paper, \emph{\pesppassenergy}, as a bicriteria optimization problem combining periodic timetabling and matching of braking and accelerating trains in \Cref{ss:pesp-passenergy}. Finally, we present a mixed-integer programming formulation in \Cref{ss:bicriteria-mip}.

\subsection{Preliminaries}
\label{ss:prob-input}

In the timetabling problem, we are given a public transport network and a set of lines $\ell \in \cL$, which are defined as sequences of served stations~$v \in \mathcal V$. Every line is operated periodically with a given period~$T$. Hence, we search for a periodic timetable that assigns an integer number (time) between $0$ and $T-1$ to each departure and each arrival of a line at a station. This timetable needs to respect the duration of activities such as travel times between stations, dwell times at stations, and transfer times for passengers. In contrast to the usual setting, e.g., in \cite{Liebchen2006}, we also consider \emph{energy} transfers between a braking and an accelerating train.

The standard model used for such periodic timetabling problems is the periodic event scheduling problem (\pesp) developed in \cite{Serafini1989}, which is based on an \emph{event-activity-network} (EAN) $\cE = (E, A)$, whose nodes are periodically repeated \emph{events} and whose arcs are \emph{activities} that take place between events. For timetabling, the network is derived from a set of directed lines~$\cL$ serving stations $v \in \mathcal{V}$, see for example~\cite{nac1996,odi1996}.  The node set is given by arrival and departure events $E = E_{\arr} \dotcup E_{\dep}$, constructed from the lines and stations as follows:
\[
E_{\arr} \coloneqq \bigl\{ (v, \ell, \arr) \bigm| \ell \in \cL\text{ arrives at }v \in \mathcal{V}\bigr\},\ 
E_{\dep} \coloneqq \bigl\{ (v, \ell, \dep) \bigm| \ell \in \cL\text{ departs at }v \in \mathcal{V}\bigr\}.
\]
The activities $A \coloneqq A_\drive \dotcup A_\wait \dotcup A_\transfer\dotcup \possenergyact$ connect the events as follows:
\begin{align*}
	A_\drive  &\coloneqq \bigl\{ ((v_1, \ell, \dep), (v_2, \ell , \arr)) \in E_\dep \times E_\arr \bigm| \ell \text{ serves }v_2 \text{ directly after } v_1\bigr\}, \\
	A_\wait  &\coloneqq \bigl\{ ((v, \ell, \arr), (v, \ell , \dep)) \in E_\arr \times E_\dep \bigr\}, \\
	A_\transfer &\subseteq \bigl\{ ((v, \ell_1, \arr), (v, \ell_2 , \dep)) \in E_\arr \times E_\dep \bigm| \ell_1 \neq \ell_2 \bigr\},\\
	\possenergyact &\subseteq \bigl\{ ((v, \ell_1, \dep), (v, \ell_2 , \arr)) \in E_{\dep} \times E_{\arr}\bigr\}.
\end{align*}

\begin{figure}
    \centering

\begin{tikzpicture}[->, shorten >=1pt, scale=0.28]
			\tikzset{labelnode/.style={sloped, above, font=\scriptsize}}
            \tikzset{labelnoderight/.style={right, font=\scriptsize}}
            \tikzset{labelnodeleft/.style={left, font=\scriptsize}}
			\filldraw[station] (-0.2,-1.2) rectangle ++ (14.4,8.5);
			\filldraw[station] (24.8,-1.2) rectangle ++ (14.4,8.5);
			
			\node[Arr] (1) at (2,6) {};
			\node[Dep] (2) at (12,6) {};
			
			\node[Dep] (3) at (2, 0){};
			\node[Arr] (4) at (12,0) {};
			
			\node[Arr] (5) at (27,6) {};
			\node[Dep] (6) at (37,6) {};
			
			\node[Dep] (7) at (27,0) {};
			\node[Arr] (8) at (37,0) {};
			
			\node[] (9) at (-5, 6) {};
            \node[] (10) at (-5, 0) {};

            \node[] (11) at (44, 6) {};
            \node[] (12) at (44, 0) {};

            \node[red,font=\scriptsize] (13) at (7, 3) {energy};
            \node[red,font=\scriptsize] (14) at (32, 3) {energy};

			\node[font=\scriptsize, anchor=east] at (1.3, 9){};
			\node[font=\scriptsize, anchor=west] at (12.7, 9){};
			
			\draw (1) edge node[labelnode] {wait} (2); 
			\draw (4) edge node[labelnode] {wait} (3); 
			
			\draw (5) edge node[labelnode] {wait} (6);
            \draw (8) edge node[labelnode] {wait} (7);
			

            \draw (4)  edge[ bend angle=45, bend right] node[labelnoderight] {\;transfer} (2);
            \draw (1)  edge[ bend angle=45, bend right] node[labelnodeleft] {transfer\;} (3);Arrival
            \draw (5)  edge[ bend angle=45, bend right] node[labelnodeleft] {transfer\;} (7);
            \draw (8)  edge[ bend angle=45, bend right] node[labelnoderight] {\;transfer} (6);
            

            \draw (2)  edge[red, bend angle=45, bend right]  (4);
            \draw (3)  edge[ red, bend angle=45, bend right] (1);
            \draw (7)  edge[red, bend angle=45, bend right]  (5);
            \draw (6)  edge[ red, bend angle=45, bend right] (8);

			\draw (2) edge node[labelnode] {drive} (5); 
			\draw (7) edge node[labelnode] {drive} (4); 

            \draw (9) edge node[labelnode] {drive} (1);		
            \draw (3) edge node[labelnode] {drive} (10);

            \draw (6) edge node[labelnode] {drive} (11);		
            \draw (12) edge node[labelnode] {drive} (8);
			
		\end{tikzpicture}
     \caption{Event-activity network with energy activities}
    \label{fig:EAN}
\end{figure}
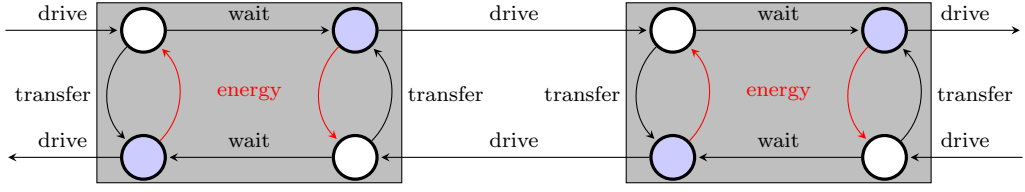

\cref{fig:EAN} depicts an EAN. Stations are represented by gray boxes. The white nodes correspond to the arrival events and the blue nodes correspond to the departure events at a station. The standard activities of the PESP are depicted by black arrows, whereas the newly introduced energy activities are highlighted by red arrows.

All activities impose precedence constraints, ensuring the correct (cyclic) order of the events. For $A_\drive$, $A_\wait$, and $A_\transfer$, this has been used in numerous works in the literature. The energy activities from $\possenergyact$ follow the same reasoning (see \cref{fig:energy_arc}): At a station~$v$, a braking train of line~$\ell_2$ can only transfer energy to an accelerating train of line~$\ell_1$ if the departure of $\ell_1$ takes place shortly before the arrival of $\ell_2$, i.e., the event $j = (v, \ell_1, \dep)$ precedes $i = (v,\ell_2,\arr)$. The energy activity $((v,\ell_1,\dep), (v,\ell_2,\arr))$ hence ensures that energy can be transferred from $\ell_2$ to $\ell_1$.

A \emph{periodic timetable} $\pi \colon E \to \{0,\dotsc,T-1\}$ assigns a time~$\pi_i$ to each event~$i \in E$, meaning that the event takes place at all times from $\pi_i + T\Z$. We impose bounds on the driving, transfer and waiting times: For each activity~$a \in A$ let $\Delta_a  = [l_a, u_a]$ be the interval of allowed durations with $l_a, u_a \in \Z$. Since we only determine the times modulo $T$, we can ignore multiples of $T$ in the activity durations and therefore assume that $0 \leq l_a \leq T-1$ and $0 \leq u_a - l_a \leq  T - 1$. Then a timetable is \emph{feasible} if the \emph{periodic tensions} $x_{a} \coloneqq (\pi_j - \pi_i - l_{a}) \bmod T + l_{a}$ lie within the provided bounds for all $a = ij \in A$. If $u_a - l_a = T-1$, then this is always the case and we call the activity \emph{free}.

For real-world applications it is desirable to find timetables that enable the usage of regenerative energy as well as short travel times for the passengers. Hence, it is necessary to consider a bicriteria problem maximizing the usage of regenerative energy, i.e., overlap of brake and traction phases, as well as minimizing the passengers' travel time so that we can study Pareto optimal solutions to find a good trade-off.
In order to properly express the objective functions under consideration, we need to be given the following parameters: First, we need the number of passengers $w_a$ traveling on each activity $a \in A_{\wait} \dotcup A_{\drive} \dotcup A_{\transfer}$  To simplify notation, we set $w_a=0$ for all $a \in \possenergyact$. 
Second, the time~$t^{\br}_i$ needed for braking at each arrival event~$i \in E_{\arr}$ and the time~$t^{\ac}_j$ needed for accelerating at each departure event~$j \in E_{\dep}$ have to be given (cf.~\cref{fig:energy_arc}). We assume that any acceleration and braking times together are smaller than the period time~$T$, i.e., $t^{\ac}_j + t^{\br}_i < T$ for any $ji \in \possenergyact$. 

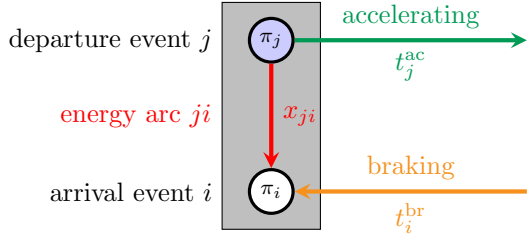
\begin{figure}
		\centering
		\begin{tikzpicture}[font={\small}, xscale=1.3]
			\fill[station] (-.5,-.5) rectangle +(1,3);
			\node[Dep, label={[label distance=2ex]left:departure event $j$}] (d) at (0,2) {$\pi_j$};
			\node[Arr, label={[label distance=2ex]left:arrival event $i$}] (a) at (0,0) {$\pi_i$};
			\draw[energy] (d) -- node[left, outer sep=4ex] {energy arc $ji$}  node[right] {$x_{ji}$} (a);
			\draw[accelerate] (d) -- node[above] {accelerating} node[below] {$t_j^\ac$} +(2.6,0);
			\draw[brake] (2.6,0) -- node[above] {braking} node[below] {$t_i^\br$} +(a);
			\node[label={below:\vphantom{$t_j^{\br}$}}] at (0,0) {};
		\end{tikzpicture}
		\caption{Energy arc with periodic timetable, tension, acceleration time, and braking time, see \cite{atmos}. Note that the energy arc goes in opposite direction to the energy flow.}
		\label{fig:energy_arc}
\end{figure}

 In this paper, we assume $u_a = l_a + T - 1$ for the bounds on the transfer arcs~$a$, meaning that these activities are free. The periodic tensions assigned to the transfer activities, however, count into the objective function of minimizing the passengers' travel time. Similarly, we assume the bounds on the energy arcs~$a \in \possenergyact$ to be given by $l_a = 0$ and $u_a = T - 1$. Hence, energy activities are free and do not impose any feasibility constraints, either. However, they are required to measure the overlap of the braking and the acceleration phases of two trains, as modeled in the subsequent section.

\subsection{Modeling the Brake-Traction Overlap}
\label{ss:overlap}

Let us consider one energy activity $ji \in A_{\energy }$. We denote the minimum of the acceleration and braking times associated with energy arc $ji$ by $t_{ji}^{\min} \coloneqq \min \{ t_{j}^{\ac}, t_{i}^{\br} \} $ and the maximum by $t_{ji}^{\max} \coloneqq \max \{ t_{j}^{\ac}, t_{i}^{\br} \} $.
As for the definition of a timetable above, we represent all periodically repeating times by their occurrence in the interval $[0,T)$. With this convention, and in contrast to \cite{Serafini1989}, we define a periodic interval as
\[
[a, b]_T \coloneqq
\begin{cases}
	[a \bmod T, b \bmod T] &\text{if } a \bmod T \leq b \bmod T,\\
	[0, b \bmod T] \cup[ a \bmod T, T) &\text{else.}
\end{cases}
\]
The length of a periodic interval is then $\length([a, b]_T) = (b-a) \bmod T$. 
The periodic interval of the acceleration phase after the departure event~$j$ is then $[\pi_j, \pi_j + t_j^{\ac}]_T$. Analogously, $[\pi_i - t_i^{\br}, \pi_i]_T$ describes the braking phase before the arrival event~$i$. Due to the assumption that $t^{\ac}_j + t^{\br}_i < T$, the intersection of both intervals is again a periodic interval, whose length equals the overlap of the braking and acceleration phases:

\begin{definition}[Brake-Traction Overlap]\label{def:overlap}
	For $ji \in \possenergyact$
	we define the \emph{(brake-traction) overlap} resulting from a periodic timetable~$\pi$ as 
	$o_{ji} \coloneqq \length\bigl([\pi_j, \pi_j + t_j^{\ac}]_T \cap [\pi_i - t_i^{\br}, \pi_i]_T\bigr)$.
\end{definition}

Clearly, the overlap does not depend on the exact times $\pi_j$ and $\pi_i$ but only on their difference, i.e., on the periodic tension $x_{ji} =(\pi_i - \pi_j) \bmod{T}$. 

\begin{lemma} \label{lem:brake-traction-overlap}
    Let $\pi$ be a periodic timetable. The overlap on an energy activity~$a \in \possenergyact$ with periodic tension~$x_a$ is $o_a = \max\bigl\{ \min \{ x_a,\, t^{\min}_a,\, t^{\max}_a + t^{\min}_a - x_a\},\ 0\bigr\}$.
\end{lemma}

\begin{proof}
	Let $ a= ji$. There are two cases in which there is an empty intersection $[\pi_j, \pi_j + t_j^{\ac}]_T \cap [\pi_i - t_i^{\br}, \pi_i]_T$. First, the intersection is empty if $\pi_j \leq \pi_i$ and $\pi_j + t_j^{\ac} < \pi_i - t_i^{\br}$. This is the case whenever $t_j^{\ac} + t_i^{\br}< \pi_i - \pi_j = (\pi_i - \pi_j) \bmod T = x_{ji}$.
	The second case in which the intersection is empty is if  $\pi_j > \pi_i$ and $\pi_j + t_j^{\ac} < \pi_i + T - t_i^{\br}$. This is true whenever $t_j^{\ac} + t_i^{\br} < \pi_i + T - \pi_j = (\pi_i - \pi_j) \bmod T = x_{ji}$. Hence, we have an empty intersection if and only if $t_j^{\ac} + t_i^{\br} - x_{ji} < 0$. In this case the overlap is $o_{ji} = 0 $.
	
	Provided that the intersection is non-empty, we receive the length of the overlap by the minimum of the lengths of the four periodic intervals $[\pi_j, \pi_j + t_j^{\ac}]_T$, $ [\pi_i - t_i^{\br}, \pi_i]_T$, $[\pi_j, \pi_i]_T$, $[\pi_i - t_i^{\br}, \pi_j + t_j^{\ac} ]_T$. These four cases are illustrated in \cref{fig:cases-overlap}.
    \begin{figure}
     \centering
     \begin{tikzpicture}
      \foreach \i in {0,2} {
       \draw[ultra thick] (0,\i) -- ++(6,0) ++(2,0) -- +(6,0);
       \draw[energy, -] (2,{\i+.2}) -- +(2,0) ++ (8,0) -- +(2,0);
       }
     
      \draw[brake, -] (.5,2.5) -- +(5,0);
      \draw (.5,2.1) -- +(0,-.2) node[below, brake] {$\pi_i-t_i^{\br}$};
      \draw (5.5,2.1) -- +(0,-.2) node[below, brake] {$\pi_i$};
      \draw[accelerate, -] (2,2.7) -- +(2,0);
      \draw (2,2.1) -- +(0,-.2) node[below, accelerate] {$\pi_j$};
      \draw (4,2.1) -- +(0,-.2) node[below, accelerate] {$\pi_j+t_j^{\ac}$};

      \begin{scope}[xshift=8cm]
       \draw[brake, -] (2,2.5) -- +(2,0);
       \draw (2,2.1) -- +(0,-.2) node[below, brake] {$\pi_i-t_i^{\br}$};
       \draw (4,2.1) -- +(0,-.2) node[below, brake] {$\pi_i$};
       \draw[accelerate, -] (.5,2.7) -- +(5,0);
       \draw (.5,2.1) -- +(0,-.2) node[below, accelerate] {$\pi_j$};
       \draw (5.5,2.1) -- +(0,-.2) node[below, accelerate] {$\pi_j+t_j^{\ac}$};
      \end{scope}

      \draw[brake, -] (.5,.5) -- +(3.5,0);
      \draw (.5,.1) -- +(0,-.2) node[below, brake] {$\pi_i-t_i^{\br}$};
      \draw (4,.1) -- +(0,-.2) node[below, brake] {$\pi_i$};
      \draw[accelerate, -] (2,.7) -- +(3.5,0);
      \draw (2,.1) -- +(0,-.2) node[below, accelerate] {$\pi_j$};
      \draw (5.5,.1) -- +(0,-.2) node[below, accelerate] {$\pi_j+t_j^{\ac}$};

      \begin{scope}[xshift=8cm]
       \draw[brake, -] (2,.5) -- +(3.5,0);
       \draw (2,.1) -- +(0,-.2) node[below, brake] {$\pi_i-t_i^{\br}$};
       \draw (5.5,.1) -- +(0,-.2) node[below, brake] {$\pi_i$};
       \draw[accelerate, -] (.5,.7) -- +(3.5,0);
       \draw (.5,.1) -- +(0,-.2) node[below, accelerate] {$\pi_j$};
       \draw (4,.1) -- +(0,-.2) node[below, accelerate] {$\pi_j+t_j^{\ac}$};
      \end{scope}
     \end{tikzpicture}
     \caption{The four cases from the proof of \cref{lem:brake-traction-overlap}. The overlap of the braking and acceleration intervals is shown in red.}
     \label{fig:cases-overlap}
    \end{figure}
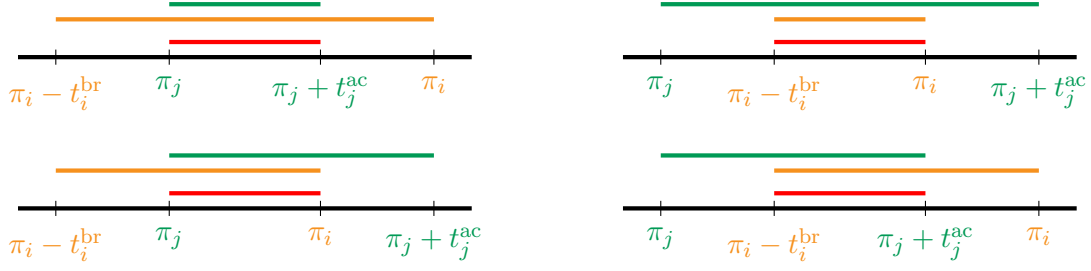
    This yields
	\begin{align*}
		o_{ji} &= \min \{t_j^{\ac},\, t_i^{\br},\, (\pi_i - \pi_j) \bmod T,\, (t_j^{\ac} + t_i^{\br} - (\pi_i - \pi_j) \bmod T) \bmod T \} \\
		&= \min \{t_j^{\ac},\, t_i^{\br},\, x_{ji},\, (t_j^{\ac} + t_i^{\br} - x_{ji}) \bmod T \} \\
		&= \min \{t_j^{\ac},\, t_i^{\br},\, x_{ji},\, t_j^{\ac} + t_i^{\br} - x_{ji} \}\\
		&= \min \{t_{ji}^{\min},\, x_{ji},\, t_{ji}^{\min} + t_{ji}^{\max} - x_{ji} \} \ge 0.
	\end{align*} 
	The third equation holds by the assumption that we have a non-empty intersection.	Therefore,
	\[\min\{t_{ji}^{\min},\, x_{ji},\, t_{ji}^{\min} + t_{ji}^{\max} - x_{ji}\} \ge 0 \iff [\pi_j, \pi_j + t_j^{\ac}]_T \cap [\pi_i - t_i^{\br}, \pi_i]_T \neq \emptyset.\]
	Hence, for the actual overlap of energy arc $a \in \possenergyact$ we obtain:
	\[o_{a} = \max\bigl\{ \min \{ x_a,\, t^{\min}_a,\, t^{\max}_a + t^{\min}_a - x_a\},\ 0\bigr\}. \qedhere\]
\end{proof}

\begin{figure}
    \centering
    \begin{tikzpicture}
        \draw[<->] (6.1,0) node[right] {$x_a$} -| (0,2.6) node[above] {$o_a$};
        \draw[blue, ultra thick] (0,0) -- (2,2) -- (3,2) -- (5,0) -- (5.9,0);
        \foreach \i/\l in {0/0, 2/{t_a^{\min}}, 3/{t_a^{\max}}, 5/{t_a^{\max} + t_a^{\min}}}
        \node[fill, circle, inner sep=0pt, minimum size=4pt, label={below:$\l$}] at (\i,0) {};
        \node[fill, circle, inner sep=0pt, minimum size=4pt, label={left:$t_a^{\min}$}] at (0,2) {};
    \end{tikzpicture}
    \caption{Brake-traction overlap~$o_a$ as a function of the periodic tension~$x_a$ on energy activity~$a$, see \cite{atmos}}
    \label{fig:overlap}
\end{figure}
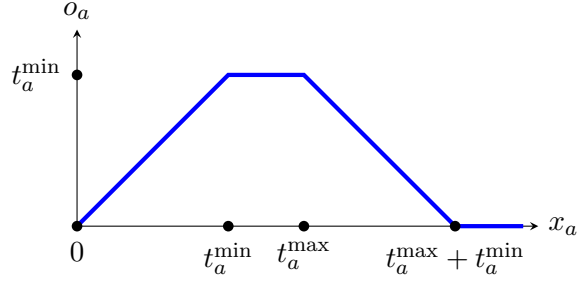

\begin{corollary} \label{cor:full-overlap}
 The maximum possible overlap at $a \in \possenergyact$ is $o_{a} = t_a^{\min}$, which is achieved if and only if $t_a^{\min} \le x_a \le t_a^{\max}$.
\end{corollary}
In this case, we say that there is \emph{full overlap} on $a$.

\subsection{The \pesppassenergy Problem}
\label{ss:pesp-passenergy}

\begin{table}[ht]
	\centering
	\renewcommand{\arraystretch}{1.2}
	\begin{tabular}{ | p{0.15 \linewidth} | p{0.7 \linewidth}|}
		\hline
		\multicolumn{2}{|c|}{\textbf{Bicriteria Problem \textnormal{\emph{\pesppassenergy}}  } }  \\ \hline
		\textbf{Input:} & $\bullet$ $\cE = (E,A)$, an \textbf{event-activity-network}\\
        & \hspace{5pt} with $A = A_\drive \dotcup  A_\wait \dotcup A_\transfer \dotcup \possenergyact$ as in \Cref{ss:prob-input}\\
		&$\bullet$ $T \in \mathbb{N}$, a \textbf{period time}\\
		&$\bullet$ $l \in \Z^{A}$ with $0\leq l < T$, a \textbf{lower bound} vector\\
        & \hspace{5pt} with $l_a = 0$ for all $a \in \possenergyact$ \\
		&$\bullet$ $u \in \Z^{A}$ with $l\leq u < T+l$, an \textbf{upper bound} vector\\ 
        & \hspace{5pt} with $u_a = T-1$ for all $a \in \possenergyact$ \\
        & \hspace{5pt} and $u_a = l_a + T - 1$ for all $a \in A_\transfer$\\
		&$\bullet$ $w \in \R^{A}$ with $w \geq 0$, a \textbf{weight} vector (number of passengers) \\ 
        & \hspace{5pt} with $w_a = 0$ for all $a \in \possenergyact$ \\
		&$\bullet$ $t^{\ac} \in \Z^{E_{\dep}}$ with $t^{\ac} > 0$, an \textbf{acceleration time} vector\\
		&$\bullet$ $t^{\br} \in \Z^{E_{\arr}}$ with $t^{\br} > 0$, a \textbf{braking time} vector\\
        & \hspace{5pt} with $t_j^\ac + t_i^\br < T$ for all $ji \in \possenergyact$ \\
        \hline
		
		\textbf{Output:} 
        &$\bullet$ $\energyact \subseteq \possenergyact$, a \textbf{matching of selected energy activities} \\
        &$\bullet$ $\pi \in \{0,1, \dots , T-1 \}^{E}$, a feasible \textbf{periodic timetable} \\
		&$\bullet$  $x \in \Z^{A}$, \textbf{periodic tensions} with $l\leq x \leq u$ and \\
		& \hspace{5pt} $x_a = (\pi_{j}-\pi_{i}-l_{a})
        \bmod T + l_{a}$ for all $a = ij \in A$ \\
		& $\bullet$ $o\in \Z^{\energyact}$ \textbf{overlap} of brake and traction phases (\Cref{def:overlap})\\
		& $\bullet$  minimizing $\sum_{a \in A} w_a x_a$ \\
		& $\bullet$ maximizing $ \sum_{a \in \energyact} o_a$\\ 
		 \hline
	\end{tabular}
    \caption{Input and output for \pesppassenergy}
	\label{tab:prob-in-out}
\end{table}

Having presented the input data in \Cref{ss:prob-input} and having discussed an expression for quantifying the brake-traction overlap in \Cref{ss:overlap}, we can now introduce the central bicriteria optimization problem of this paper. In contrast to previous work~\cite{Wang2022a,Wang2022}, we do not assume that the set of energy activities between arrivals and departures is given in advance. Instead, in our problem, it has to be selected during the optimization. More specifically, the given set $A_{\energy}$ contains all \emph{possible} energy activities, and the task is to select a subset $\energyact \subseteq \possenergyact$ of \emph{actual} energy activities, which can be used to transfer energy during a brake-traction overlap phase. We assume that the selected subset must form a matching from arrival to departure events, i.e., no braking train can transfer its energy to more than one accelerating train, and no accelerating train can receive energy from multiple braking trains.

\begin{definition}[\pesppassenergy]
    Consider a tuple $(\cE, T, l, u, w, t^{\ac}, t^{\br})$ as described in \Cref{tab:prob-in-out}, where $\cE = (E, A)$. The goal is to select a matching $\energyact \subseteq \possenergyact$ and a feasible periodic timetable $\pi \in [0, T)^E$ with associated periodic tension $x \in \R^A$ and overlap $o \in \R^{\energyact}$ such that $\sum_{a \in A} w_a x_a$ is minimized and $\sum_{a \in \energyact} o_a$ is maximized.
\end{definition}

\Cref{tab:prob-in-out} gives an overview of the problem's input and output.
The problem \emph{\pesppassenergy} therefore integrates the task of periodic timetabling with finding an optimal matching of energy arcs, enabling energy transmission from braking to accelerating trains, and accounting for the fact that energy can only be reused once.

\subsection{MIP Formulation of the Bicriteria Problem}
\label{ss:bicriteria-mip}

In this section, we present a new MIP formulation of \pesppassenergy, which combines constraints from the PESP model from, e.g., \cite{Liebchen2006} with overlap constraints similar to \cite[Constraints~\mbox{[9]}]{Ramos2008} and matching constraints.

\setcounter{equation}{0}
\begin{align}
	& \max &  & \sum_{ji \in A_{\text{energy}}} o_{ji} & \label{mip:obj1} \\[5 pt] 
	&\min
	& &\sum_{ij \in A}w_{ij} x_{ij} \label{mip:obj2} \\[5pt]
	& \text{subject to}
	& x_{ij} &= \pi_{j}-\pi_{i}+p_{ij}T && \forall ij \in A \label{mip:pesp1} \\
	&& l_{ij}&\leq x_{ij}\leq u_{ij} && \forall  ij \in A\label{mip:pesp2} \\
	&& 0 &\leq \pi_{i} \leq T-1 & & \forall  i \in E \label{mip:pesp3}\\[5pt] 
	& & o_{ji} & \leq x_{ji} & &\forall {ji} \in A_{\text{energy}} \label{mip:energy2}\\
	& &o_{ji}  &\leq \alpha_{ji} \cdot t^{\min}_{ji} & &\forall {ji} \in A_{\text{energy}} \label{mip:energy1}\\
	& & o_{ji} &\leq t^{\max}_{ji} + t^{\min}_{ji} - x_{ji} + (1- \alpha_{ji}) \Gamma & &\forall {ji} \in A_{\text{energy}} \label{mip:energy3}\\[5 pt]
	& & \sum_{ji \in A_{\text{energy}} \cap \delta^-(i)} \alpha_{ji} &\leq 1 & & \forall i \in E_{\text{arr}}  \label{mip:energy5}\\
	& & \sum_{ji \in A_{\text{energy}} \cap \delta^+(j)} \alpha_{ji} &\leq 1 & & \forall j \in E_{\text{dep}} \label{mip:energy6}\\ 
	&& o_{ji} &\ge 0,\ \alpha_{ji} \in \{0,1\}& &  \forall ji \in A_{\text{energy}} \label{mip:energy7} \\
	&& x_{ij} &\ge 0,\ p_{ij} \in \Z & & \forall ij \in A \label{mip:pesp4} \\ 
	&& \pi_i &\in \Z && \forall i \in E \label{mip:pesp5}
\end{align}

The variables $p_{ij}$ are called periodic offsets or modulo parameters and are chosen such that the periodic tensions~$x_{ij}$ lie within the bounds. This is ensured by constraints \labelcref{mip:pesp1,mip:pesp2}. Constraints \labelcref{mip:pesp3,mip:pesp5} ensure that the timetable $\pi$ takes only values within $\{0,\dotsc,T-1\}$.

The constraints \eqref{mip:energy5} and \eqref{mip:energy6} ensure that the energy arcs chosen at each station form a matching. The variables $\alpha_{ji}$ define a matching~$\energyact$ with $ji \in \energyact$ if and only if $\alpha_{ji} = 1$.

The variable~$o_{ji}$ determines the brake-traction overlap for the energy arcs in $\energyact$. If $ji \notin \energyact$, then constraints~\eqref{mip:energy1} force $o_{ji}$ to be zero, while constraints~\labelcref{mip:energy2,mip:energy3} do not impose an additional restriction. Note that it is not necessary to add $(1-\alpha_{ji}) \Gamma$ to constraints~\eqref{mip:energy2} because $x_{ji} \ge 0$ due to the periodic shifting. For $ji \in \energyact$ the constraints \labelcrefrange{mip:energy2}{mip:energy3} bound the overlap according to \cref{lem:brake-traction-overlap}. The constant~$\Gamma$ is chosen large enough so that for $\alpha_{ji} = 0$ the constraints \labelcref{mip:energy3} do not impose a relevant bound on $o_{ji}$. It can be set to 
\[\Gamma \coloneqq \max\bigl\{T - (t^{\max}_{ji} + t^{\min}_{ji}) \bigm| ji \in A_{\energy}\bigr\}.\]

Wang et al.~\cite{Wang2022a} introduced an equivalent model for the overlap, which, however, requires more variables and constraints. The equivalence of both models has been shown in \cite{atmos}.

We denote the problem with the single objective of minimizing the passengers' travel time~\eqref{mip:obj2} by \emph{\pesppassenger}. This is the classic problem studied, e.g., in \cite{nac1996,Liebchen2006} that consists of the constraints \eqref{mip:pesp1}--\eqref{mip:pesp3} and  \eqref{mip:pesp4}--\eqref{mip:pesp5}. The problem with the single objective of maximizing the brake-traction overlap~\eqref{mip:obj1}, and therefore the usage of regenerative energy, is denoted by \emph{\pespenergy}, and consists of the constraints \eqref{mip:pesp1}--\eqref{mip:pesp5}.

\section{Complexity Overview}
\label{sec-compl-overview}

Even without an objective function, the Periodic Event Scheduling Problem (\pesp), which consists of deciding whether a feasible timetable exists, is NP-complete for any fixed $T\geq 3$ \cite{odijk1994}. When all activities are free, so that feasibility is guaranteed, finding a periodic timetable with minimum total travel time is NP-hard~\cite{Nachtigall1993,LindnerReisch22}.
Therefore, also the problem of deciding whether there is a feasible timetable with a total travel time of the passengers smaller than some constant $K_1$ and a total overlap of brake and traction phases that is bigger than some constant $K_2$ is NP-complete.

To obtain a better understanding, we investigate the complexity of \emph{\pesppassenergy} on a class of simple but meaningful instances, and for each individual objective.

\subsection{One-Station Networks}

From the application context, it makes sense to consider energy transmission locally at a single station.
Our upcoming complexity investigations are based on a small network of a single transfer station called a \emph{\osn}. Among others, we will show that the (classic) \pesppassenger is already NP-hard on one-station networks.


\begin{definition}[One-Station Network]\label{osnetwork}
	An EAN $\cE_n =(E, A)$ is called a \emph{\osn} with $n$ lines if it is based on one station, i.e., $|\mathcal{V}| = 1$, and $n$ (directed) lines stopping at this station inducing the following events:
	\begin{align*}
		E_{\arr} &\coloneqq \bigl\{ ( \ell, \arr) \bigm| \ell \in [n] \bigr\}, 
		& E_{\dep} &\coloneqq \bigl\{ (\ell, \dep) \bigm|  \ell \in [n] \bigr\}.
	\end{align*}
	The activities $A =  A_\wait \dotcup A_\transfer \dotcup \possenergyact$ connect the events as follows:	
	\begin{align*}
		A_\wait  &\coloneqq \bigl\{ (( \ell, \arr), ( \ell, \dep)) \in E_\arr \times E_\dep \bigr\}, \\
		A_\transfer &\subseteq \bigl\{ ((\ell_1, \arr), (\ell_2 , \dep)) \in E_\arr \times E_\dep \bigm| \ell_1 \neq \ell_2 \bigr\}, \\
		A_{\energy}  &\subseteq \bigl\{ ((\ell_1, \dep), (\ell_2, \arr)) \in E_{\dep} \times E_{\arr} \bigr\}
		.
	\end{align*}
\end{definition}

 There are no driving activities in a \osn. W.l.o.g.\ we use the following simplification: A \osn~$\cE_n^{\mathrm{pass}}$ for \pesppassenger has the arcs $A_\wait \cup A_{\transfer}$, while for \pespenergy, we only consider waiting activities~$A_\wait$ and possible energy activities~$A_\energy$ in the \osn~$\cE_n^{\energy}$. This is a valid reduction, as the activity bounds on both energy and transfer arcs do not impose any restrictions on feasibility (see \Cref{ss:prob-input}), and energy arcs are irrelevant for the objective function of \pesppassenger, as transfer arcs are for \pespenergy.

\subsection{Complexity Results}

\Cref{Tab:complexity} gives an overview on our results concerning the problems' complexity on a \osn. Here, we distinguish between a \osn denoted by $K_{n,n}$ 
that contains all possible transfer activities (\pesppassenger) or energy activities (\pespenergy), respectively, and a more general \osn denoted by $\mathcal E_n \subseteq K_{n,n}$, where an arbitrary set of transfer or possible energy activities is given as input. 
Further, we make a distinction whether arbitrary bounds on the waiting activities can be specified, or all waiting activities are free.

\begin{table}
\centering	

\begin{tabular}{|l|l|l|l|} 
	\hline
		\multicolumn{2}{|c|}{One-Station Network }  & 	\multicolumn{2}{|c|}{ Objective } \\  \hline
	graph & bounds on $A_{\wait}$ & min travel time  & max energy overlap \\  \hline	
	$\mathcal E_n \subseteq K_{n, n}$ & arbitrary & NP-hard (A) & NP-hard (E) \\
	\multirow{2}{*}{$K_{n,n}$} & \multirow{2}{*}{arbitrary} & \multirow{2}{*}{NP-hard (B)} & ? -- polynomial special cases, \\ 
		& & & \hphantom{? -- }NP-hard restrictions (F)\\	
	$\mathcal E_n \subseteq K_{n, n}$ & free & NP-hard (C) & polynomial (G) \\
	$K_{n,n}$ & free & NP-hard (D) & polynomial (H) \\
	  \hline

\end{tabular} 	
\caption{Complexity overview for one-station networks}\label{Tab:complexity} 
\end{table}

As shown in \cref{Tab:complexity}, the problem of minimizing the travel time on $\cE_n^{\mathrm{pass}}$ is NP-hard in all considered cases. The proof of case~(D) uses a reduction from the Maximum Cut Problem and can be found in \cref{sec:one-transfer} (\cref{NPhardPassengers}). As (A), (B), and (C) contain (D) as a special case, this implies the NP-hardness of (A), (B), (C) as well.

For the problem \pespenergy, we obtain a wider range of theoretical insights and complexity results. 
For the most general case (E), we show the NP-hardness in \cref{ss:pesp-energy-complexity} by a reduction from the Directed Hamiltonian Path Problem (\Cref{thm:pesp-energy-subset-hardness}).
On the other hand, when the waiting activities are free (cases (G) and (H)), an optimal solution can be found in polynomial time (\Cref{thm:pesp-energy-complexity-free-waiting}).

For \pespenergy on a \osn with arbitrary bounds on the waiting activities in which every pair of departure and arrival can be matched by an energy activity (case (F)) it remains open whether the problem is NP-hard or solvable in polynomial time. In \cref{sec:pesp-energy-one-station}, however, we provide a polynomial algorithm with an additive performance guarantee (\Cref{thm:boundH}). In \cref{sec-Special-Cases}, we show that some special cases are solvable in polynomial time, namely when the period time is large, when the acceleration or braking times are uniform, or when each train has the same acceleration and braking time.
On the other hand, we investigate restricted versions of (F), where the selected energy activities together with the waiting activities $\energyact \dotcup A_{\wait}$ should form exactly one cycle or at least two cycles in \cref{sec-Variants}. While the first variant is solvable in polynomial time (\Cref{thm:pesp-energy-single-cycle}), the second one is NP-hard (\Cref{thm:pesp-energy-two-cycles}).

We finally remark that the corresponding variants of the problem of deciding whether there is a periodic timetable with travel time at most $K_1$ and overlap at least $K_2$ are NP-hard, because all travel time variants are NP-hard (cf.\ \Cref{Tab:complexity}).

\section{\pesppassenger on a One-Station Network} \label{sec:one-transfer}

In this section, we will consider \pesppassenger minimizing the passengers' travel time, i.e., the ``classic'' PESP.  We strengthen the results of \cite{Nachtigall1993} and \cite{odijk1994} by showing that the problem of finding a timetable minimizing the total transfer time is already hard on one-station networks where all waiting activities are present but do not restrict feasibility.

\begin{theorem} \label{NPhardPassengers}
	The problem \pesppassenger is NP-hard on a \osn, even if all possible transfer activities are present, the period time is fixed to $T = 2$, and $u_a = l_a + 1$ for all $a \in A_{\wait}$.
\end{theorem}

\begin{proof}
	We show NP-hardness by a reduction from the Max-Cut problem. Let $I$ be an arbitrary Max-Cut instance, consisting of a graph~$G=(V,R)$ and a weight function~$\upsilon \colon R \to \R$. We search a bipartition $(S, T)$ of the vertex set $V$ such that the sum of the weights on the edges between the sets $S$ and $T$, $\sum_{\{i, j\} \in R: i \in S, j \in T} \upsilon(\{i,j \})$, is maximum.
	
	Based on $I$ we define an instance~$I'$ of the \pesppassenger problem on a \osn: Let $\cE_n^{\mathrm{pass}} = (E_{\arr} \dotcup E_{\dep},\, A_{\wait} \dotcup A_{\transfer})$ be a \osn with $n \coloneqq |V|$ lines, inducing $n$ arrival events and $n$ departure events, and $A_{\wait} \dotcup A_{\transfer} = E_{\arr} \times E_{\dep}$. This means that $\cE_n^{\mathrm{pass}}$ is the complete bipartite graph $K_{n,n}$ with all edges directed from $E_\arr$ to $E_\dep$. We consider the period time $T \coloneqq 2$. The waiting activities $a \in A_{\wait}$ have $l_a \coloneqq 0$ and $u_a \coloneqq 1$, and their weights are $w_a \coloneqq \sum_{r \in R} \upsilon(r) + 1$. For the transfer activities $((i,\arr),(j,\dep)) \in A_{\transfer}$, define the bounds $l_{(i,\arr),(j,\dep)} \coloneqq 1$ and $u_{(i,\arr),(j,\dep)} \coloneqq 2$ and weights
	\[
	w_{(i,\arr),(j,\dep)} \coloneqq
	\begin{cases}
		\upsilon(\{i,j\}) &\text{if } \{i,j\} \in R, \\
		0 &\text{else.}
	\end{cases}
	\]
	
	We want to show that any optimal solution to $I'$ can be transformed to an optimal solution of $I$. 
	Let $\pi$ be an optimal timetable for $I'$. Then for every $i \in V$ we have $\pi_{(i,\dep)} = \pi_{(i,\arr)}$: Suppose, to the contrary, that there exists some $i \in V$ with $\pi_{(i,\dep)} = \pi_{(i,\arr)} + 1$. Then the waiting arc $((i,\arr),(i,\dep))$ contributes $\sum_{r \in R} \upsilon(r) + 1$ to the total travel time. By moving the departure time earlier to $\pi_{(i,\arr)}$, these costs vanish, while the transfer times from each other line to line~$i$ may increase by at most $1$. However, this increase of transfer time is clearly bounded by $\sum_{r \in R} \upsilon(r)$. Thus, the timetable~$\pi$ cannot minimize the total travel time.
    
    We define $S \coloneqq \{i \in V \mid \pi_{(i,\arr)} = \pi_{(i,\dep)} = 0\}$ and $T \coloneqq \{i \in V \mid \pi_{(i,\arr)} = \pi_{(i,\dep)} = 1\}$.
	To see that $(S,T)$ is an optimal solution to the Max-Cut problem,
    note that $\pi$ minimizes the sum of the weights multiplied with the periodic tensions in $I'$. As every transfer arc has periodic tension~$1$ or $2$, this is the same as maximizing the sum of the weights of arcs with periodic tension $1$, which are exactly those between the sets $S$ and $T$. Hence, $(S,T)$ is a maximum-weight cut.
\end{proof}

\begin{remark}
It is well-known that finding a feasible timetable in the sense of PESP on arbitrary EANs is NP-hard for any fixed period time $T$ with $T \geq 3$ \cite{odijk1994}. This result does not carry over to $T = 2$, as any instance can be preprocessed in such a way that every activity $a$ is either fixed ($l_a = u_a$) or free ($u_a - l_a = T - 1 = 1$). As fixed arcs can always be contracted (cf.\ \cite{Liebchen2006}), no feasibility restrictions remain: Any PESP instance with $T = 2$ is trivially feasible. \Cref{NPhardPassengers} shows however that the hardness persists when asking for an optimal timetable. It is not surprising that simple networks suffice: It has been shown in \cite{Borndoerfer2020} that the separation of a maximally violated cycle or change-cycle inequality for \pesppassenger is already NP-hard on a star network with turnaround loops.
\end{remark}

Now we establish a special case in which we know the structure of an optimal solution. 

\begin{definition}[Basel Solution Structure]
	A timetable~$\pi$ for a one-station network~$\cE_n^{\mathrm{pass}}$ has a \emph{Basel solution structure} if all arrival events are scheduled at the same time~$\pi^\arr$ and all departure events at time~$\pi^\dep$ such that 
    $$(\pi^{\dep} - \pi^\arr) \bmod T = l^{\max} \coloneqq \max \{l_{a} \mid a \in A\}.$$
\end{definition}

The name ``Basel solution structure'' has been chosen since such a structure occurs in integrated interval-fixed interval timetables as they are, e.g., common in Basel, Switzerland. In these timetables, all trains arrive and depart roughly at the same time, respectively, as is also illustrated in \cref{fig:example-basel-a}. The following result confirms that such a solution is highly attractive for passengers.

\begin{proposition}\label{thm:BaselSol}
	Let $\cE_n^{\mathrm{pass}} = (E, A)$ be a \osn with lower and upper bounds $l_{a}, u_{a}$ on the arcs such that $u_{a} = l_{a} + T - 1$ for all transfer arcs $a \in A_{\transfer}$. Then any timetable~$\pi$ with the Basel solution structure minimizes the total travel time independently of the weights if and only if
	$l_{a} = l_{a'}$ for all $a, a' \in A_\transfer \cup A_\wait$.
\end{proposition}

\begin{proof}
	First, let us assume that $l_{a} = l_{a'}$ for all $a, a' \in A_\transfer \cup A_\wait$. Let $\pi$ be a timetable with the Basel solution structure. Then the periodic tensions induced by $\pi$ are 
    $$x_{ij} = (\pi_j - \pi_i -l_{ij}) \bmod T + l_{ij} = (l^{\max} - l^{\max} ) \bmod T + l^{\max} =  l^{\max}$$ for all $ij \in A_\transfer \cup A_\wait$. As we cannot do better than attaining the lower bounds on the tensions, $\pi$ must be optimal.
	
	Let us now assume that $\pi$ is optimal and that there is an arc $a' \in A_\transfer \cup A_\wait$ with $l_{a'}  < l^{\max}$. In the following we find a weight vector~$w$ for which $\pi$ is not optimal. Let $a' = i'j'$. Then the following timetable $\pi'$ achieves a lower objective value than $\pi$ for the following weight vector~$w$:
	\[w_{ij} \coloneqq \begin{cases} 1 &\text{if } ij = i'j', \\ 0 &\text{else,}\end{cases}\qquad\quad \pi_i' \coloneqq \begin{cases} 0 &\text{if } i = i', \\  l_{i'j'} &\text{if } i = j', \\ \text{arbitrary feasible values} &\text{else.} \end{cases}\]
	This is possible as only the waiting activities impose feasibility constraints.
	The weighted sum of the periodic tensions w.r.t.\ $\pi$ is $l^{\max}$, and it is $l_{a'}$ w.r.t.\ $\pi'$. By assumption, $l_{a'} < l^{\max}$, hence $\pi$ is not optimal.
\end{proof}

While a timetable with the Basel solution as in \Cref{thm:BaselSol} is optimal for \pesppassenger, the overlap of the energy activities is forced to be small (see also \cref{fig:example-basel}), as typically transfer, acceleration and braking times are relatively short:

\begin{proposition}
    Let $\cE_n = (E, A)$ be a one-station network and let $\pi$ be a timetable with the Basel solution structure. If $l^{\max} + \max\{t^\ac_{j} + t^\br_{i} \mid ji \in A_\energy\} < T$, then any solution to \pespenergy using $\pi$ has overlap $0$.
\end{proposition}
\begin{proof}
    Let $(M, \pi, x, o)$ be a solution to \pespenergy with $\pi$ having the Basel solution structure. For $ji \in A_\energy$ then holds $x_{ji} = T  - l^{\max}$, which by assumption is strictly larger than $t^\ac_{j} + t^\br_{i}$. By \Cref{lem:brake-traction-overlap}, we find $o_{ji} = 0$.
\end{proof}

\section{\pespenergy on a One-Station Network} \label{sec:pesp-energy-one-station}

Similar to the previous section where \pesppassenger on a \osn was analyzed with respect to its complexity and the structure of an optimal solution has been described for a simple special case, we now investigate the same questions for \pespenergy on a  \osn. We distinguish between the two cases where an arbitrary set~$\possenergyact$ of of event pairs is given between which an energy activity is possible and where an energy activity can be selected between every departure and arrival. While, in the first case, we are able to show the problem's NP-hardness, in the second case, we provide a polynomial-time algorithm with additive performance bound. 

We start with a simple upper bound for the objective value of the maximization of the brake-traction overlap on an arbitrary network.

\begin{proposition}\label{prop:upperbound-matching}
	For an instance of \pespenergy on an EAN $\cE = (E, A)$, consider a feasible solution $S= (\energyact, \pi, x, o)$. Let $W$ be the maximum weight of a (perfect) matching in the graph $G=(E, \possenergyact)$ with respect to the weights~$\w_a \coloneqq t^{\min}_a$ for all $a \in \possenergyact$. Then $\sum_{a \in \energyact} o_a \leq W$.
\end{proposition}

\begin{proof}
	Each overlap is bounded from above by both the corresponding acceleration and the corresponding braking time: $o_{ji} \leq t_{j}^{\ac}$ and $o_{ji} \leq t_{i}^{\br}$.
\end{proof}

The problem can be divided into two stages: first, select a matching of the energy arcs, i.e., determine which trains can transfer their energy to which trains, and second, find a timetable for the selected energy activities. Of course, the timetable affects what constitutes a good matching in the first stage. We first study the second stage of finding a timetable for a given matching.

\subsection{The Timetable for a Given Matching of Energy Activities}

In this subsection, we investigate for a \osn how to obtain an optimal timetable for a given matching of energy activities, and how close its achieved overlap is to the weight of the matching.
The input~$\cE_n^{\energy}$ is a bipartite graph with partition classes $E_\arr$ and $E_\dep$. The set of waiting activities constitutes a perfect matching from $E_\arr$ to $E_\dep$. These activities impose the only feasibility constraints on the timetable~$\pi$. In contrast, the energy activities solely influence the objective value. Hence, there are no restrictions between arrival and departure times of different lines.

For any matching~$\energyact \subseteq \possenergyact = E_\dep \times E_\arr$, the graph $(E,\,\energyact \dotcup A_\wait)$ is a union of node-disjoint directed cycles and directed paths. By \cref{cor:full-overlap}, we achieve full overlap on a cycle~$C$ if and only if the periodic tension~$x$ satisfies $t_{a}^{\min} \le x_a \le t_a^{\max}$ for all energy activities $a \in C \cap \energyact$. On the other hand, for any feasible periodic tension, it has to hold that $l_a \le x_a \le u_a$ for all $a \in A_{\wait}$ and $\sum_{a \in C} x_a \in T \Z$.

\begin{definition}
    \label{def:delta}
    For a matching $\energyact \subseteq \possenergyact$ and a cycle $C \subseteq \energyact \dotcup A_\wait$, we define the distance from the interval that allows full overlap to the nearest multiple of the period time by
	\[ \delta_C \coloneqq d([L_C, U_C], T\Z) = \min \{|x - \gamma T| \mid x \in [L_C, U_C],\ \gamma \in \Z\},
    \]
    where $L_C \coloneqq \sum_{a \in C \cap A_{\wait}} l_a + \sum_{a \in C \cap \energyact} t_a^{\min}$ and $U_C \coloneqq \sum_{a \in C \cap A_{\wait}} u_a + \sum_{a \in C \cap \energyact} t_a^{\max}$.
\end{definition}

A timetable with maximum brake-traction overlap on the selected energy activities can be determined for each connected component of $(E, \energyact \dotcup A_\wait)$ separately. The following \lcnamecref{thm:struc-cycle} describes the structure of an optimal timetable for a directed cycle. Recall that $\w_a = t_a^{\min}$ for $a \in \possenergyact$. For a set~$B \subseteq A$ of activities, we write $\w(B) \coloneqq \sum_{a \in B \cap \possenergyact} \w_a$ and $o(B) \coloneqq \sum_{a \in S \cap \possenergyact} o_a$.

\begin{theorem}\label{thm:struc-cycle}
Let $\energyact \subseteq \possenergyact$ be a matching, let $C \subseteq \energyact \dotcup A_\wait$ be a cycle, and let $a_0 \in C$ with $\w_{a_0} \leq \w_a$ for all $a \in C \cap \energyact$. Then there is a solution $(\energyact, \pi, x, o)$ to \pespenergy such that 
\begin{enumerate}[\normalfont (a)]
    \item $o_a = \w_a$ for all $a \in C \cap \energyact \setminus \{a_0\}$,
    \item $o(C) = \w(C) - \min\{t^{\min}_{a_0},\, \delta_C\}$,
    \item $o(C)$ is maximum among all solutions that fix the energy activities~$\energyact$.
\end{enumerate}
\end{theorem}
\begin{proof}    
    Suppose first $\delta_C = 0$. Then there is a vector $x \in \mathbb R^C$ such that $l_a \leq x_a \leq u_a$ for all $a \in C \cap A_\wait$ and $t_a^{\min} \leq x_a \leq t_a^{\max}$ for all $a \in C \cap \energyact$ such that $x(C) = \sum_{a \in C} x_a $ is an integer multiple of $T$. In particular, $x$ is a periodic tension on $C$, and by \cref{cor:full-overlap}, we can realize full overlap $o_a = t^{\min}_a = \w_a$ on all energy activities in $C$.

    Now assume that $\delta_C > 0$. Without loss of generality, we assume that $\delta_C$ is the distance from $L_C$ to the next integer multiple of $T$, i.e., there is a $\gamma \in \mathbb Z$ such that $L_C - \delta_C = \gamma T$ and $U_C + \delta_C \leq \gamma T + T$. Define $x \in \mathbb R^C$ via
    $$
    x_a \coloneqq \begin{cases}
        l_a & \text{ if } a \in C \cap A_\wait,\\
        t_a^{\min} & \text{ if } a \in C \cap \energyact \setminus\{a_0\},\\
        (t_{a_0}^{\min} - \delta_C) \bmod T & \text{ if } a = a_0.
    \end{cases}
    $$
    Then $x(C) \equiv L_C - \delta_C \equiv \gamma T \equiv 0 \pmod T$. We infer that $x$ is a periodic tension on $C$ that defines full overlap $o_a = t^{\min}_a = \w_a$ on all $a \in C \cap \energyact \setminus \{a_0\}$. In the case that $\delta_C \leq t_{a_0}^{\min}$, we have $x_{a_0} = t_{a_0}^{\min} - \delta_C$ and therefore $o_{a_0} = t_{a_0}^{\min} - \delta_C$, so that
    $$ o(C) = \sum_{a \in C \cap \energyact \setminus \{a_0\}} \w_a + t_{a_0}^{\min} - \delta_C = \w(C) - \delta_C = \w(C) - \min\{t_{a_0}^{\min},\, \delta_C\}.$$    
    Otherwise, if $\delta_C > t_{a_0}^{\min}$, we have $x_a = t_{a_0}^{\min} - \delta_C + T$, as $0 < \delta_C < T$. Since then
    \begin{align*}
    x_{a_0} &=  t_{a_0}^{\min} - \delta_C + T \\
    &\geq t_{a_0}^{\min} + U_C - \gamma T \\
    &= t_{a_0}^{\min} + (U_C - L_C) + \delta_C \\
    &> t_{a_0}^{\min} + (t_{a_0}^{\max} - t_{a_0}^{\min}) + t_{a_0}^{\min} \\
    &= t_{a_0}^{\min} + t_{a_0}^{\max},
    \end{align*}
    we find $o_{a_0} = 0$. This yields
    $$ o(C) = \sum_{a \in C \cap \energyact \setminus \{a_0\}} \w_a = \w(C) - t_{a_0}^{\min} = \w(C) - \min\{t_{a_0}^{\min},\, \delta_C\}.$$  
    The argument when $\delta_C$ is the distance from $U_C$ to the nearest integer multiple of $T$, is analogous, so that we have established (a) and (b).

    It remains to show (c), i.e., that no feasible periodic tension can induce larger overlap on $C$. To this end, consider an optimal solution $(\energyact, \pi^*, x^*, o^*)$. Let 
    \[\delta^* \coloneqq d([L_C, U_C], x^*(C)) = \min \{|x-x^*(C)| \mid x \in [L_C, U_C]\}.\]
    Since $x^*(C)$ is an integer multiple of $T$, we have $\delta^* \geq \delta_C$. The periodic tension $x^*$ hence misses the full overlap interval $[L_C, U_C]$ by $\delta$. By definition of the overlap function (cf.~\cref{fig:overlap}), we lose an overlap of at least $\min\{t_{a_0}^{\min}, \delta\}$. We obtain
    $$o^*(C) \leq \w(C) - \min\{t_{a_0}^{\min}, \delta^*\} \leq \w(C) - \min\{t_{a_0}^{\min}, \delta_C\} = o(C),$$
    showing that the periodic tension $x$ constructed above realizes the maximum possible overlap on $C$.
\end{proof}

We can interpret directed paths that arise as components of $(E, \energyact \dotcup A_\wait)$ as cycles with a missing edge of overlap~0.
Hence, by \cref{thm:struc-cycle}, there is an optimal timetable for such a path such that all energy edges have full overlap.

\subsection{Complexity Results for \pespenergy}
\label{ss:pesp-energy-complexity}

We have seen that finding an optimal timetable for given energy activities is an easy task. The complexity of \pespenergy thus comes from the task of selecting the energy activities that admit an optimal timetable with a large brake-traction overlap. We start with the simple case where we have no restrictions on the waiting activities.

\begin{theorem}\label{thm:pesp-energy-complexity-free-waiting}
    Let $\cE_n^{\energy}$ be a \osn such that all waiting activities are free. Then a maximum-weight matching w.r.t.~$\w \coloneqq t^ {\min}$ yields an optimal solution to \pespenergy whose overlap equals the matching's weight. In particular, \pespenergy is solvable in polynomial time.
\end{theorem}
\begin{proof}
    We first compute a maximum-weight matching $\energyact$ within the set $\possenergyact$ of possible energy activities as in \Cref{prop:upperbound-matching}. By \Cref{thm:struc-cycle}, we find a timetable $\pi$ that realizes the maximum possible overlap $o_a = t_a^{\min}$ for all energy arcs except possibly for one arc per cycle component $C$ of $A_\wait \dotcup \energyact$. Since every such cycle $C$ necessarily contains a free waiting activity, we can add (modulo $T$) any number $\delta$ to the periodic tension of the critical arc, and subtract (modulo $T$) the same number $\delta$ from the periodic tension of an arbitrary waiting activity in $C$, while maintaining feasibility, and not altering the overlap of the other components of $A_\wait \dotcup \energyact$. This way, we can achieve the maximum overlap on all energy arcs of $C$. By \Cref{prop:upperbound-matching}, this solution is optimal.
\end{proof}

Let us now consider the more general case of \pespenergy with restricted waiting times and a restricted set of compatible arrivals and departures between which an energy transfer is possible.

\begin{theorem}
    \label{thm:pesp-energy-subset-hardness}
	The problem \pespenergy on a \osn~$\cE_n^{\energy}$ with an arbitrary set of possible energy activities and arbitrary bounds on the waiting activities is NP-hard. 
\end{theorem}

\begin{proof}
	We prove the statement through a reduction from the Hamiltonian Path Problem on a directed graph $D=(V, R)$ with $|V| = n$. Let us define the following instance of \pespenergy on a \osn. Let $\cE_n^{\energy} =(E, A)$ be a \osn with $n$ lines, where a pair $((j, \dep),(i, \arr))$ is contained in $\possenergyact$ if and only if there is an arc $ji \in R$. Let the bounds on the waiting activities be given by $l_a=u_a=0$ for all $a \in A_{\wait}$, hence, the departure time of a line will be its arrival time. The braking and acceleration times of all lines are set to $t^{\ac}_i = t^{\br}_i=1$ for all $i \in [n]$. The period time is chosen to be equal to the number of lines plus one: $T = n+1$.
    
    Now, there is a Hamiltonian path in $D$ if and only if there is a solution to our \pespenergy instance with a total overlap of at least $n-1$. Assuming there is a Hamiltonian path from $s$ to $t$ in $D$, we select the energy activities corresponding to the arcs of the path. We can then start at the departure node of the line associated with $s$ and set its departure time to $0$. Then, given that we have visited node $i$, for node $i+1$ of the Hamiltonian path we fix the arrival and departure times of the corresponding line to $\pi((i+1, \arr))=\pi((i+1, \dep))= \pi((i, \arr))+1$. Like this, we obtain full overlap on the energy arcs given by the Hamiltonian path. As the Hamiltonian path consists of $n-1$ arcs, we achieve an overlap of $n-1$. 
    
    Consider now a solution $S=(\energyact, \pi, x, o)$ to \pespenergy with overlap of at least $n-1$. Then there must be full overlap on at least $n-1$ energy arcs. Let $A_{\text{full}}$ be the se of these arcs. Then $A_{\text{full}} \cup A_{\wait}$ induces a union of paths and cycles in $\cE_n$. If there is a cycle of full overlap arcs, then every arc on that cycle must have a periodic tension (duration) of $1$, which is not possible because $T > n$. Therefore, there are only paths. Since $A_{\text{full}}$ consists of $n-1$ arcs, there can only be a single path, which is thus a Hamiltonian path in $\cE_n^{\energy}$. This corresponds to a Hamiltonian path in the original graph~$D$.
\end{proof}

The cases that are not covered by \Cref{thm:pesp-energy-complexity-free-waiting,thm:pesp-energy-subset-hardness}, i.e., when waiting activities can have arbitrary bounds and $\cE_n$ is complete are studied in the following subsections.

\subsection{Perfect and Greedy Matchings of Energy Arcs} \label{subsec:matching}

While \cref{thm:struc-cycle} describes the structure of an optimal timetable for each connected component resulting from a fixed set~$\energyact$ of energy activities, we will now investigate the problem of selecting the energy activities.
In the following, we assume that an energy activity is possible between any departure and arrival event ($\cE_n^{\energy} = K_{n,n}$). The next result provides a lower bound on the number of energy activities in a globally optimal solution. 

\begin{theorem}\label{thm:matchingsize}
	In every optimal solution $S = (\energyact, \pi, x, o)$ to \pespenergy on a \osn $\cE_n^{\energy}$ there are at least $n-1$ energy activities in $\energyact$. 
\end{theorem}

\begin{proof}
	Let $S = (\energyact, \pi, x, o)$ be an optimal solution to \pespenergy with $|\energyact| < n-1$, so at least two connected components of $(E, A_\wait \dotcup \energyact)$ are directed paths~$P_1, P_2$. For $k \in \{1,2\}$ let $i_k \in E_\arr$ be the start and $j_k \in E_\dep$ be the end node of $P_k$.
	
	Let $\energyact' \coloneqq \energyact \cup \{ j_1i_2 \}$, $c \coloneqq \pi_{j_1} + t^{\min}_{j_1i_2} - \pi_{i_2}$, and let us define the following timetable
	\[
	\pi'_v \coloneqq \begin{cases} \pi_v &\text{if } v \in E\setminus V(P_2), \\
		(\pi_v + c) \bmod T &\text{if } v \in V(P_2).
	\end{cases}
	\]
    We note that $\pi'$ is feasible. 
    For the new energy activities $\energyact'$ and the new timetable $\pi'$, we obtain the periodic tensions
	\[x'_{ji} = \begin{cases}
		(\pi'_i - \pi'_j) \bmod T = (\pi_i - \pi_j) \bmod T = x_{ji} &\text{for } ji \in \energyact,\\
		(\pi'_{i_2} - \pi'_{j_1}) \bmod T = t_{j_1i_2}^{\min} &\text{for } j=j_1, i=i_2,
	\end{cases}\]
	and hence the brake-traction overlaps
	\[
	o'_{ji} = \begin{cases}
		o_{ji} &\text{if } ji \in \energyact, \\
		t^{\min}_{j_1i_2} &\text{if } ji = j_1i_2.
	\end{cases}
	\]
	Consequently, as $t^{\min}_{j_1i_2} > 0$, $S' \coloneqq (\energyact', \pi', x', o')$ yields a better objective value than $S$, so $S$ cannot be optimal.
\end{proof}

\begin{corollary} \label{cor:extension-perfect-matching}
 If $\energyact$ is a matching that maximizes the achievable overlap and is not perfect, then there is a unique perfect matching $\energyact^p$ extending $\energyact$, and this leads to the same overlap.
\end{corollary}

This yields an alternative way of looking at optimal solutions to \pespenergy on a \osn. A perfect matching~$\energyact^{p} \subseteq \possenergyact$ corresponds one-to-one to a permutation $\varphi \colon [n] \to [n]$ of the trains (lines) in a \osn. This is given by $\varphi(\ell) = k$ if and only if $((\ell,\dep),(k,\arr)) \in \energyact^p$. The directed cycles in $\energyact^p \dotcup A_\wait$ then correspond to the cycles of the permutation.

Recall that we can find an upper bound for the objective value of a \pespenergy instance by calculating the maximum-weight perfect matching within the set of possible energy activities~$a \in \possenergyact$ w.r.t.\ the weights~$\w \colon A_{\energy} \to \R$ with $\w_{ji} = \min\{ t_{j}^{\ac}, t_{i}^{\br}\}$ (cf.~\cref{prop:upperbound-matching}). It is known that for such weight structure, the maximum-weight perfect matching can be found easily by a greedy algorithm for the weights in our problem. Sorting both $t^{\br}_i$, $i \in E_{\arr}$, and $t^{\ac}_j$, $j \in E_{\dep}$, according to their sizes, we obtain permutations~$\sigma$ and $\varphi$ with $t^{\ac}_{\sigma(1)}\leq \dots \leq t^{\ac}_{\sigma(n)}$ and  $t^{\br}_{\varphi(1)} \leq \dots \leq t^{\br}_{\varphi(n)}$.

\begin{lemma}[{e.g.~\cite[Lemma~1]{vanDal}}]\label{lem:greedy}
$\energyact_\mathrm{greedy} \coloneqq \bigl\{(\sigma(i), \varphi(i)) \in E_{\dep} \times E_{\arr} \bigm| i \in [n]\bigr\}$ is a maximum-weight perfect matching w.r.t.\ $\w$ in the graph $(E, A_\energy)$.
\end{lemma}

\subsection{A Hamiltonian Path Algorithm}

Having unraveled the connection to permutations, we will now present a heuristic based on Hamiltonian paths.

\begin{proposition}\label{prop:lowerbound-ham}
    Let $P$ be a Hamiltonian path in $\cE_n^\energy$ that traverses all waiting activities, and let $a'$ be the arc from the end node to the start node of $P$. Then the matching $\energyact \coloneqq (P \cap A_\energy) \cup \{a'\}$ admits a solution with overlap $\sum_{a \in \energyact} o_a \ge \sum_{a \in P \cap A_\energy} \w_a$.
\end{proposition}
In particular, the total weight of any Hamiltonian cycle containing all waiting activities is a lower bound on the optimal objective value.

\begin{proof}
    Let $C \coloneqq P \cup \{a'\}$.
    By applying \Cref{thm:struc-cycle} to 
    $C$, we find a timetable that realizes the full overlap $o_a = t_a^{\min} = \w_a$ for all energy arcs $a$ on $P$.
\end{proof}

In \cref{fig:exampleHP} we can see that neither a maximum-weight perfect matching nor a maximum-weight Hamiltonian path w.r.t.\ $\w$ necessarily yield an optimal solution to \pespenergy. In \cref{fig:celeft} the maxmimum-weight perfect matching (obtained by greedy sorting, see \cref{lem:greedy}) together with the waiting activities $\energyact\dotcup A_{\wait}$ decomposes into three cycles. Due to the cycle periodicity, however, we cannot obtain full overlap in the second cycle. There is no overlap on the dashed arc. While the greedy matching has weight $38$, only an overlap of $32$ can be obtained from this matching. 
In \cref{fig:cemiddle} an optimal solution is depicted. We can see two cycles, one of them containing a dashed arc with no overlap. The achieved overlap is $35$.
In \cref{fig:ceright}, we can see a maximum-weight Hamiltonian path with full overlap on all energy arcs. In total an overlap of $34$ is achieved. Due to the cycle periodicity it is not possible to obtain overlap on the missing energy arc.

\begin{figure}
    \centering
	\tikzset{labelnode/.style={sloped, above, font=\scriptsize}}
	\begin{subfigure}{0.3\textwidth}
		\centering
		\begin{tikzpicture}[->, shorten >=1pt, scale=0.3]
			
			\filldraw[station] (-0.2,-6.4) rectangle ++ (14.4,19.8);
			
			\node[Arr] (5) at (2,12) {0};
			\node[Dep] (7) at (12,12) {5};
			
			\node[Arr] (2) at (2, 9){13};
			\node[Dep] (3) at (12,9) {3};
			
			\node[Arr] (1) at (2,5) {0};
			\node[Dep] (4) at (12,5) {5};
			
			\node[Arr] (6) at (2,2) {12};
			\node[Dep] (8) at (12,2) {2};
			
			\node[Arr] (9) at (2,-2) {0};
			\node[Dep] (10) at (12,-2) {5};
			
			\node[Arr] (11) at (2,-5) {7};
			\node[Dep] (12) at (12,-5) {12};
			
			\node[font=\scriptsize, anchor=east] at (1.3, 9){8};
			\node[font=\scriptsize, anchor=west] at (12.7, 9){12};
			
			\node[font=\scriptsize, anchor=east] at (1.3, 12){12};
			\node[font=\scriptsize, anchor=west] at (12.7, 12){8};
			
			\node[font=\scriptsize, anchor=east] at (1.3, 5){6};
			\node[font=\scriptsize, anchor=west] at (12.7, 5){7};
			
			\node[font=\scriptsize, anchor=east] at (1.3, 2){7};
			\node[font=\scriptsize, anchor=west] at (12.7, 2){6};
			
			\node[font=\scriptsize, anchor=east] at (1.3, -2){5};
			\node[font=\scriptsize, anchor=west] at (12.7, -2){2};
			
			\node[font=\scriptsize, anchor=east] at (1.3, -5){2};
			\node[font=\scriptsize, anchor=west] at (12.7, -5){3};
			
			\draw (1) edge node[labelnode] {5} (4); 
			\draw (6) edge node[labelnode] {5} (8); 
			
			\draw (5) edge node[labelnode, below] {5} (7);
			\draw[black] (2) edge node[labelnode,below] {5} (3);
			
			\draw (9) edge node[labelnode] {5} (10); 
			\draw (11) edge node[labelnode] {5} (12);
			
			
			\draw[red] (3) edge node[labelnode, pos=0.25] {12} (5);
			\draw[red] (7) edge node[labelnode, pos=0.3] {8} (2);

			\draw[red, dashed] (8) edge node[labelnode, pos=0.3] {6} (1);
			\draw[red] (4) edge node[labelnode, pos=0.3] {7} (6);
			
			\draw[red] (12) edge node[labelnode, pos=0.3] {3} (9);
			\draw[red] (10) edge node[labelnode, pos=0.3] {2} (11);
			
		\end{tikzpicture}
		\caption{Greedy matching, \\ weight: 38, total overlap: 32}
		\label{fig:celeft}
	\end{subfigure}
	\begin{subfigure}{0.3\textwidth}
		\centering
		\begin{tikzpicture}[->, shorten >=1pt, scale=0.3]
			
			\filldraw[station] (-0.2,-6.4) rectangle ++ (14.4,19.8);
			
			
			\node[Arr] (5) at (2,12) {0};
			\node[Dep] (7) at (12,12) {5};
			
			\node[Arr] (2) at (2, 9){13};
			\node[Dep] (3) at (12,9) {18};
			
			\node[Arr] (1) at (2,5) {8};
			\node[Dep] (4) at (12,5) {13};
			
			\node[Arr] (6) at (2,2) {5};
			\node[Dep] (8) at (12,2) {10};
			
			\node[Arr] (9) at (2,-2) {0};
			\node[Dep] (10) at (12,-2) {5};
			
			\node[Arr] (11) at (2,-5) {0};
			\node[Dep] (12) at (12,-5) {5};

			\node[font=\scriptsize, anchor=east] at (1.3, 9){8};
			\node[font=\scriptsize, anchor=west] at (12.7, 9){12};
			
			\node[font=\scriptsize, anchor=east] at (1.3, 12){12};
			\node[font=\scriptsize, anchor=west] at (12.7, 12){8};
			
			\node[font=\scriptsize, anchor=east] at (1.3, 5){6};
			\node[font=\scriptsize, anchor=west] at (12.7, 5){7};
			
			\node[font=\scriptsize, anchor=east] at (1.3, 2){7};
			\node[font=\scriptsize, anchor=west] at (12.7, 2){6};
			
			\node[font=\scriptsize, anchor=east] at (1.3, -2){5};
			\node[font=\scriptsize, anchor=west] at (12.7, -2){2};
			
			\node[font=\scriptsize, anchor=east] at (1.3, -5){2};
			\node[font=\scriptsize, anchor=west] at (12.7, -5){3};
			
			\draw (1) edge node[labelnode] {5} (4); 
			\draw (6) edge node[labelnode] {5} (8); 
			
			\draw (5) edge node[labelnode, below] {5} (7);
			\draw[black] (2) edge node[labelnode,below] {5} (3);
			
			\draw (9) edge node[labelnode] {5} (10); 
			\draw (11) edge node[labelnode] {5} (12);
			
			
			\draw[red] (3) edge node[labelnode, pos=0.25] {12} (5);
			\draw[red] (7) edge node[labelnode, pos=0.3] {8} (2);

			\draw[red] (8) edge node[labelnode, pos=0.3] {5} (9);
			\draw[red] (4) edge node[labelnode, pos=0.35] {7} (6);

			\draw[red] (12) edge node[labelnode, pos=0.3] {3} (1);
			\draw[red, dashed] (10) edge node[labelnode, pos=0.4] {2} (11);
			
		\end{tikzpicture}
		\caption{Optimal solution, \\ total overlap: 35}
		\label{fig:cemiddle}
	\end{subfigure}	
	\begin{subfigure}{0.3\textwidth}
		\centering
		\begin{tikzpicture}[->, shorten >=1pt, scale=0.3]
			
			\filldraw[station] (-0.2,-6.4) rectangle ++ (14.4,19.8);
			
			\node[Arr] (5) at (2,12) {7};
			\node[Dep] (7) at (12,12) {12};
			
			\node[Arr] (2) at (2, 9){5};
			\node[Dep] (3) at (12,9) {10};
			
			\node[Arr] (1) at (2,5) {8};
			\node[Dep] (4) at (12,5) {13};
			
			\node[Arr] (6) at (2,2) {4};
			\node[Dep] (8) at (12,2) {9};
			
			\node[Arr] (9) at (2,-2) {14};
			\node[Dep] (10) at (12,-2) {4};
			
			\node[Arr] (11) at (2,-5) {0};
			\node[Dep] (12) at (12,-5) {5};

			\node[font=\scriptsize, anchor=east] at (1.3, 9){8};
			\node[font=\scriptsize, anchor=west] at (12.7, 9){12};
			
			\node[font=\scriptsize, anchor=east] at (1.3, 12){12};
			\node[font=\scriptsize, anchor=west] at (12.7, 12){8};
			
			\node[font=\scriptsize, anchor=east] at (1.3, 5){6};
			\node[font=\scriptsize, anchor=west] at (12.7, 5){7};
			
			\node[font=\scriptsize, anchor=east] at (1.3, 2){7};
			\node[font=\scriptsize, anchor=west] at (12.7, 2){6};
			
			\node[font=\scriptsize, anchor=east] at (1.3, -2){5};
			\node[font=\scriptsize, anchor=west] at (12.7, -2){2};
			
			\node[font=\scriptsize, anchor=east] at (1.3, -5){2};
			\node[font=\scriptsize, anchor=west] at (12.7, -5){3};
			
			\draw (1) edge node[labelnode] {5} (4); 
			\draw (6) edge node[labelnode] {5} (8); 
			
			\draw (5) edge node[labelnode, below] {5} (7);
			\draw[black] (2) edge node[labelnode,below] {5} (3);
			
			\draw (9) edge node[labelnode] {5} (10); 
			\draw (11) edge node[labelnode] {5} (12);
			
			
			\draw[red] (3) edge node[labelnode, pos=0.3] {12} (5);
			\draw[red] (7) edge node[labelnode, pos=0.6] {7} (6);
			
			\draw[red] (8) edge node[labelnode, pos=0.3] {5} (9);
			\draw[red] (4) edge node[labelnode, pos=0.3] {7} (2);

			\draw[red] (12) edge node[labelnode, pos=0.3] {3} (1);
			
		\end{tikzpicture}
		\caption{Hamiltonian path, \\ total overlap: 34}
		\label{fig:ceright}
	\end{subfigure}

	\caption{Example of non-optimal greedy and Hamiltonian path matchings for $T = 15$, see \cite{atmos}. The timetable~$\pi$ is written in the nodes. To the left of the departure nodes (white) the acceleration times are given, and to the right of the arrival nodes (blue) there are the braking times. The numbers on the waiting (black) and energy (red) arcs correspond to the weight $\w_a$. For the solid red arcs, they also correspond to the achieved overlap $o_a$, while $o_a = 0$ for the dashed arcs.}
	\label{fig:exampleHP}
\end{figure}
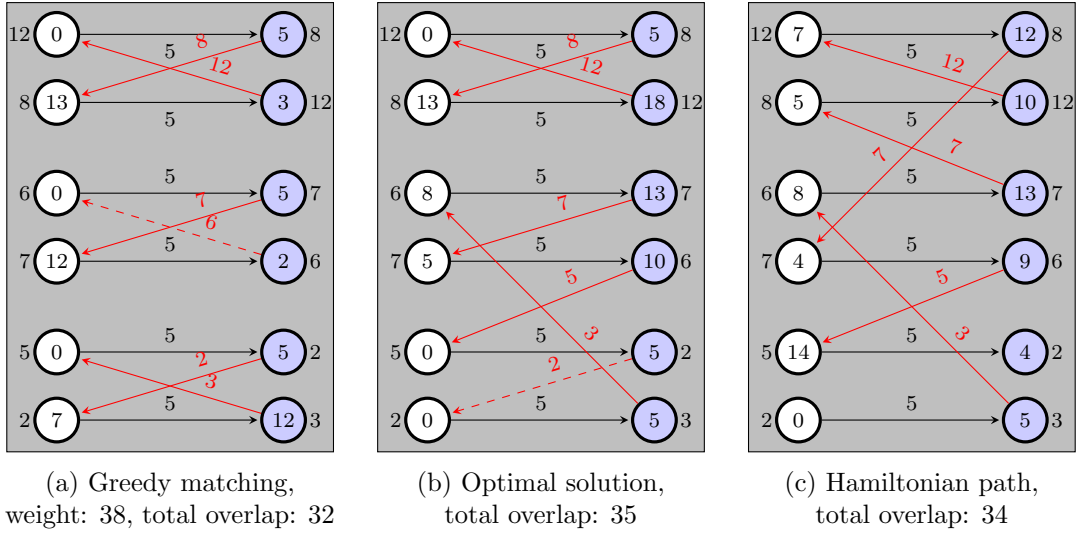

Although Hamiltonian paths might sound impractical at first glance, we show now that in our case they can be computed efficiently. 

\begin{theorem}\label{thm:findpath}
	A Hamiltonian path 
    in $\mathcal E_n^\energy$ traversing all waiting activities that has maximum weight w.r.t.\ $\w$ can be found in polynomial time.
\end{theorem}

In order to prove this theorem, we show that this problem can be transformed to the so-called \emph{Large TSP},
for which a polynomial-time algorithm is known, see~\cite{vanDal}
. To simplify the notation, in this section we write $t_\ell^{\ac} \coloneqq t_{(\ell,\dep)}^{\ac}$ and $t_\ell^{\br} \coloneqq t_{(\ell,\arr)}^{\br}$ for $\ell \in [n]$. 
Moreover, we can contract the waiting edges because they must appear in any allowed Hamiltonian path. This results in the complete bi-directed graph~$G = ([n], \cA)$ on the set of all lines and weights $\w_G \colon \cA \to \R$ defined by $\w^G(k\ell) \coloneqq \min \{t_k^{\ac}, t_\ell^{\br}\}$. This is formalized in the following \lcnamecref{lem:equivalence-hamiltonian-cycles}.

\begin{lemma} \label{lem:equivalence-hamiltonian-cycles}
	If $P_G$ is a maximum-weight Hamiltonian path in $G$ w.r.t.\ $\w^G$, then 
    $$P \coloneqq \{((k,\dep),(\ell,\arr)) \mid k\ell \in P_G\} \cup \{((\ell,\arr),(\ell,\dep)) \mid \ell \in [n]\}$$
    is a maximum-weight Hamiltonian path in $\cE_n^\energy$ w.r.t.\ $\w$ traversing all waiting activities, where we set $\w((\ell, \arr), (\ell, \dep)) \coloneqq 0$ for all $\ell \in [n]$.
\end{lemma}
\begin{proof}
	First, $P$ is the arc set of a Hamiltonian path and contains all waiting activities. Consider a maximum-weight Hamiltonian path $P^{\opt}$ that visits all waiting activities. Then
    $(P^{\opt})_G \coloneqq \{k \ell \mid ((k,\dep),(\ell,\arr)) \in P^{\opt}\}$
    is a Hamiltonian path in $G$. In particular, $\w^G((P^{\opt})_G) \leq \w^G(P_G)$, so that
    $$ \w(P^\opt) = \w^G((P^\opt)_G) 
    \leq \w^G(P_G) 
    = \w(P),$$
    and we conclude $\w(P) = \w(P^\opt)$.
\end{proof}

The next two lemmata deal with Hamiltonian cycles.

\begin{lemma} \label{lem:pathtour}
	Any maximum-weight Hamiltonian cycle in $G$ w.r.t.\ $\w^G$ contains a maximum-weight Hamiltonian path. Conversely, any maximum-weight Hamiltonian path can be closed to a maximum-weight Hamiltonian cycle.
\end{lemma}

\begin{proof}
	Let $k \coloneqq \argmin\{ \min\{ t_{\ell}^{\ac}, t_\ell^{\br} \} \mid \ell \in [n]\}$. W.l.o.g.\ let us assume $t_{k}^{\br} \leq t_k^{\ac}$. We know that all incoming arcs of $k$ have weight $\w^G_{\min} \coloneqq t_{k}^{\br}$.
	Let $C^{\opt}$ be a maximum-weight Hamiltonian cycle. Since $C^{\opt}$ must visit $k$, it must contain an arc~$a$ of weight $\w^G_{\min}$. Then $P \coloneqq C^{\opt}\setminus \{a\}$ is a Hamiltonian path with weight $\w^G(P) = \w^G(C^{\opt}) - \w^G_{\min}$.
	
	Let now $P^{\opt}$ be a maximum-weight Hamiltonian path, and let $v$ be the first and $u$ be the last vertex in $P^{\opt}$. Then $C \coloneqq P^{\opt} \cup \{uv\}$ is a Hamiltonian cycle with weight 
    $$\w^G(C) = \w^G(P^{\opt}) + \w^G(uv) \ge \w^G(P^{\opt}) + \w^G_{\min} \geq \w^G(P) + \w^G_{\min} = \w^G(C^{\opt}).$$
	
	Therefore, $\w^G(C) = \w^G(C^{\opt})$ and $\w^G(P) = \w^G( P^{\opt})$, so that
    $P$ and $C$ must be optimal.
\end{proof}

This shows that a Hamiltonian path as in \cref{thm:findpath} can be found by computing a maximum-weight Hamiltonian cycle in $G$ w.r.t.\ $\w^G$. This is obviously equivalent to computing a minimum-weight Hamiltonian cycle for the weights $-\w^G_{k\ell} = \max\{-t_k^\ac, -t_\ell^\br\}$. This is an instance of the Large TSP, which is well known to be solvable in polynomial time, see e.g.~\cite{vanDal}.

The maximum-weight Hamiltonian path yields a feasible solution to \pespenergy on a \osn. We can guarantee that the objective value of this solution is not further away from the optimal objective value than the smaller of the largest acceleration and the largest braking time. This follows from the following \lcnamecref{thm:boundH}, which bounds the difference between the lower bound from \Cref{prop:lowerbound-ham} and the upper bound from the maximum-weight matching (\Cref{prop:upperbound-matching}).

\begin{theorem}\label{thm:boundH}
	Let $P$ be a Hamiltonian path traversing all waiting activities in $\cE_n^{\energy}$ of maximum weight. Then it holds
	\[\w(\energyact_\mathrm{greedy}) - \w(P \cap A_\energy) \leq  \min\bigl\{ \max\{t^{\br}_\ell \mid \ell \in [n]\} ,\ \max\{t^{\ac}_\ell \mid \ell \in [n]\}\bigr\}.\]
\end{theorem}

\begin{proof}
	We iteratively convert the greedy matching $\energyact^{0}_\text{greedy}$ into a matching inducing a Hamiltonian cycle and bound the total reduction of weight in this process. Finally, we delete one edge to obtain a Hamiltonian path.
	
	Let us assume that $t^{\ac}_1 \leq \dots \leq t^{\ac}_n$ holds, and let $\varphi^{0}$ denote the permutation obtained by $\energyact^{0}_\text{greedy}$ such that $t^{\br}_{\varphi^{0}(1)} \leq \dots \leq t^{\br}_{ \varphi^{0}(n)}$. Then $\energyact^{0}_\text{greedy} = \{( (\ell, \dep), (\varphi^{0}(\ell), \arr)) \mid \ell \in [n]\}$. 
	The permutation $\varphi^{i}$ corresponds to the perfect matching $\energyact^{i}$ obtained in iteration $i$.	
	In each iteration, we obtain the matching $\energyact^{i}$ as follows from the matching $\energyact^{i-1}$.
	Let $C \subseteq \energyact^{i-1} \cup A_\wait$ be the cycle containing $(1, \dep)$. If $C$ is a Hamiltonian cycle, we are done. Otherwise, there is a smallest $\ell_i$ in iteration $i$ such that $(\ell_i, \dep) \in C$ but $(\ell_i +1, \dep) \not \in C$. 
    We define the new permutation $\varphi^{i}$ as follows:
	\[\varphi^{i}(x) \coloneqq \begin{cases}
		\varphi^{i-1}(\ell_i +1) &\text{if  } x = \ell_i, \\
		\varphi^{i-1}(\ell_i) &\text{if  } x = \ell_i +1, \\
		\varphi^{i-1}(x) &\text{else.} 
	\end{cases}\]
	
	For the new matching $\energyact^{i}$ we have:
    \begin{align*}
    \energyact^{i} = \energyact^{i-1} &\cup \bigl\{((\ell_i, \dep), (\varphi^{i-1}(\ell_i+1), \arr)), ((\ell_i+1,\dep),(\varphi^{i-1}(\ell_i),\arr))\bigr\} \\
    &\setminus \bigl\{((\ell_i, \dep), (\varphi^{i-1}(\ell_i), \arr)), ((\ell_i+1, \dep), (\varphi^{i-1}(\ell_i+1), \arr))\bigr\}.
    \end{align*}
    The cycle's length increases by this operation and for the weight of the new matching $\energyact^{i}$, we get
	\begin{align*}
		\w(\energyact^{i}) &= \w(\energyact^{i-1}) + \min\{t^{\ac}_{\ell_i}, t^{\br}_{\varphi^{i-1}(\ell_i+1)}\} + \min\{t^{\ac}_{\ell_i+1}, t^{\br}_{\varphi^{i-1}(\ell_i)}\} \\ & \hspace{62 pt} - \min\{t^{\ac}_{\ell_i}, t^{\br}_{\varphi^{i-1}(\ell_i)}\} - \min\{t^{\ac}_{\ell_i+1}, t^{\br}_{\varphi^{i-1}(\ell_i+1)}\} \\
		&= \w(\energyact^{i-1}) - | [t^{\ac}_{\ell_i}, t^{\ac}_{\ell_i+1} ] \cap [t^{\br}_{\varphi^{i-1}(\ell_i)}, t^{\br}_{\varphi^{i-1}(\ell_i+1)} ] |.
	\end{align*}
    Note that in each iteration $i$ we have $t_{\varphi^{i-1}(\ell_i)}^{\br} \le t_{\varphi^{i-1}(\ell_i+1)}^{\br}$, as the inequality holds for $\varphi^0$, and the image of $\ell_i + 1$ has not been changed in iterations~$1,\dotsc,i-1$, while the image of $\ell_i$ can only have decreased. 
    Therefore, we have $\varphi^{i-1}(\ell_i) \le \varphi^0(\ell_i) \le \varphi^0(\ell_i+1) = \varphi^{i-1}(\ell_i+1)$.
    
	Further, we know that $[t^{\br}_{\varphi^{i-1}(\ell_i)}, t^{\br}_{\varphi^{i-1}(\ell_i+1)} ] \subseteq [t^{\br}_{\varphi^{0}(1)}, t^{\br}_{\varphi^{0}(n)} ]$.
	Hence, the length of the intersection can be bounded by \[| [t^{\ac}_{\ell_i}, t^{\ac}_{\ell_i+1} ] \cap [t^{\br}_{\varphi^{i-1}(\ell_i)}, t^{\br}_{\varphi^{i-1}(\ell_i+1)} ] | \leq | [t^{\ac}_{\ell_i}, t^{\ac}_{\ell_i+1} ] \cap [t^{\br}_{\varphi^{0}(1)}, t^{\br}_{\varphi^{0}(n)} ] | \] and, therefore 
	$$\w(\energyact^{i}) \geq \w(\energyact^{i-1}) - | [t^{\ac}_{\ell_i}, t^{\ac}_{\ell_i+1} ] \cap [t^{\br}_{\varphi^{0}(1)}, t^{\br}_{\varphi^{0}(n)} ] |.$$
	
	Let $k$ be the number of iterations until we obtain a matching $\energyact^k$ such that $\energyact^k \dotcup A_\wait$ corresponds to a Hamiltonian cycle.  It holds $k \leq n - 1$ as there are $n$ different trains.
	\begin{align*}
		\w(\energyact^{k}) &\geq \w(\energyact^{0}_{\text{greedy}}) - \sum_{i=1}^{k} | [t^{\ac}_{\ell_{i}}, t^{\ac}_{\ell_{i}+1} ] \cap [t^{\br}_{\varphi^{0}(1)}, t^{\br}_{\varphi^{0}(n)}] | 
	\end{align*}
    We can bound the sum of the intersections as follows:
	\begin{align*}
		\sum_{i=1}^{k}  | [t^{\ac}_{\ell_i}, t^{\ac}_{\ell_i+1} ] \cap [t^{\br}_{\varphi^{0}(1)}, t^{\br}_{\varphi^{0}(n)} ] | &\leq \biggl\lvert\bigcup_{j=1}^{n} [t^{\ac}_j, t^{\ac}_{j+1} ] \biggr\rvert  = t^{\ac}_n - t^{\ac}_1,\\
		\sum_{i=1}^{k}  | [t^{\ac}_{\ell_i}, t^{\ac}_{\ell_i+1} ] \cap [t^{\br}_{\varphi^0(1)}, t^{\br}_{\varphi^0(n)} ] | &\leq |[t^{\br}_{\varphi^0(1)}, t^{\br}_{\varphi^0(n)} ]| = t^{\br}_{\varphi^{0}(n)} - t^{\br}_{\varphi^{0}(1)}.
  	\end{align*}
	Thus, $\w(\energyact^k) \geq \w(\energyact^{0}_\text{greedy}) - \min\{ t^{\ac}_n - t^{\ac}_1, t^{\br}_{\varphi^{0}(n)} - t^{\br}_{\varphi^{0}(1)} \}$. In order to receive a Hamiltonian path $P \subseteq \energyact^k \dotcup A_\wait$, we delete one edge from the matching~$\energyact^k$. As we want to maximize the path's weight, we choose the edge with the lowest weight. Due to the weight structure, this weight is $\min \{t^{\ac}_1, t^{\br}_{\varphi^{0}(1)} \}$. For the difference of the weights of the greedy matching  $\energyact_{\text{greedy}}$ and the weight of the energy arcs in $H$, we get:
	\begin{align*}
		\w(\energyact_\text{greedy}) - \w(P \cap A_\energy) &\leq  \min\{ t^{\ac}_n - t^{\ac}_1,\, t^{\br}_{\varphi^{0}(n)} - t^{\br}_{\varphi^{0}(1)} \} + \min \{t^{\ac}_1, t^{\br}_{\varphi^{0}(1)} \}\\
		& \leq   \min\{ t^{\ac}_n , t^{\br}_{\varphi^{0}(n)} \}.
	\end{align*}
	The weight of the path $P$ is a lower bound for the weight of an optimal Hamiltonian path.
\end{proof}

\section{Polynomial Cases of \pespenergy}
\label{sec-Special-Cases}
In this section, we identify special cases of \pespenergy on a \osn $K_{n,n}$ with arbitrary bounds on waiting activities that can be solved in polynomial time.
Each special case is based on a particular choice of the acceleration and braking times.

\subsection{Special Case: Large Period Time $T$}
\label{sec-SC-T-big}
In the first case, we consider a period time that is so large that no cycle $C \subseteq \energyact\dotcup A_{\wait}$ can have full overlap on all energy arcs. This corresponds to an aperiodic timetabling problem.

\begin{proposition}\label{prop:specialcase2}
    Let $u^{\max} \coloneqq \max\{u_a \mid a \in A_\wait\}+ \max\{t^{\min}_{a}+
    t^{\max}_{a} \mid  a \in A_\energy\}$, and let 
    $T \geq n \cdot u^{\max}$. Then, any matching $\energyact^{H}$ inducing a Hamiltonian path of maximum weight w.r.t.\ $\w \coloneqq t^{\min}$ defines an optimal solution $S=(\energyact^{H}, \pi^H, x^H, o^H)$. 
\end{proposition}

\begin{proof}
    Let $S = (\energyact, \pi, x, o)$ be an optimal solution and let $\energyact^H$ be any matching that induces a maximum-weight Hamiltonian path $P^H =\energyact^H \dotcup A_\wait$ w.r.t.\ $\w$. Applying \cref{thm:struc-cycle} to the cycle obtained by adding the missing arc to $P^H$, we find a timetable $\pi^H$ with overlap~$o^H$ such that $o^H_a = t^{\min}_a$ for all $a \in P^H \cap A_\energy \energyact^H$. Then $\sum_{a \in \energyact^H} o^H_a = \sum_{a \in P^H \cap A_\energy} \w_a$.
    
    Now consider an arbitrary matching~$\energyact \subseteq \possenergyact$.
    By \cref{cor:extension-perfect-matching}, we can assume that $\energyact$ is perfect. Let $C$ be an arbitrary directed cycle in $\energyact \dotcup A_\wait$.  By \cref{thm:struc-cycle} (a), we can assume that $\pi$ induces full overlap on all but the arc of $C$ with smallest $t^{\min}$, say $a_0$. Recalling \cref{def:delta}, due to $l \geq 0 $ and $t > 0$, we have $L_C \geq t_{a_0}^{\min}$. As also $U_C + t_{a_0}^{\min} \leq n \cdot u^{\max} \leq T$, we conclude $\delta_C \geq t_{a_0}^{\min}$, and hence $o_{a_0} = 0$ and $o(C) = \w(C) - t_{a_0}^{\min}$ by the formula in \cref{thm:struc-cycle} (b).
    
     Removing $a_0$ from $\energyact$, we still maintain a matching with maximum overlap. We can delete the no-overlap arcs~$a_0^1,\dotsc,a_0^m$ in all cycles~$C_1,\dotsc,C_m$ of $M \dotcup A_\wait$ and obtain 
     $$o(M) = \sum_{k=1}^m (\w(C_k) - t_{a_0^k}^{\min}) = \sum_{k=1}^m \w(C_k \setminus \{a_0\}).$$
     By adding additional arcs, we can connect the paths $C_k \setminus \{a_0\}$ to a Hamiltonian path~$P$. Therefore, $o(M) \le \w(P) \le \w(P^H)$.
\end{proof}

\subsection{Special Case: Equal Acceleration or Equal Braking Times}
\label{sec-SC-ac-equal-br-equal}
In the second case, all trains need the same time for the acceleration (or the braking) process, i.e., 
$t_i^{\ac} = t^{\ac}$ for all $i \in E^{\dep}$ or $t_i^{\br} = t^{\br}$ for all $i \in E^{\arr}$. W.l.o.g., let us assume that all acceleration times are equal, i.e., $t^{\ac} = t^{\ac}_j$ for all $i,j \in [n]$. Then the possible overlap between one braking train and any accelerating train is the same. Hence, it becomes irrelevant for which two trains the departure and the arrival are synchronized. Therefore, the solution of one big cycle of trains in arbitrary order together with a timetable maximizing the overlap on this cycle is optimal.

\begin{proposition}\label{prop:specialcase1}
	Suppose that all acceleration times are equal, i.e., $t^{\ac} = t^{\ac}_j$ for all $i,j \in [n]$, or all braking times are equal, i.e., $t^{\br} = t^{\br}_i$ for all $i,j \in [n]$. Then, there is an optimal solution consisting of one cycle of all lines in arbitrary order. This can be found in polynomial time.
\end{proposition}

\begin{proof}
	Let us assume that there is no optimal solution consisting of a single cycle, and consider an optimal solution~$(\energyact, \pi, x, o)$ with the minimum number of cycles. Let $C_1, C_2$ be two different cycles in $\energyact \dotcup A_\wait$, and let $a_k = (j_k, i_k) \in A_{\energy} \cap C_k$ for $k=1,2$ 
    Consider the alternative solution with $\energyact' \coloneqq (\energyact \setminus \{a_1, a_2\}) \cup \{j_1i_2, j_2i_1\}$. Then $\energyact'$ induces a big cycle on the node set $V(C_1) \cup V(C_2)$. Let $c \coloneqq \pi_{j_1} + x_{a_1} - \pi_{i_2}$, and set
	\[
	\pi'_v \coloneqq
	\begin{cases}
		\pi_v &\text{if } v \in E\setminus V(C_2), \\
		(\pi_v + c) \bmod T &\text{if } v \in V(C_2).
	\end{cases}
	\]
	Then the arc $j_1i_2$ with braking time $t^{\br}_{i_2} $ has the new periodic tension $x_{j_1i_2}' = (\pi'_{i_2}- \pi'_{j_1})\bmod T = (\pi_{i_2} + c - \pi_{j_1})\bmod T = x_{a_1}$, i.e.,  it also has the same overlap as there was on arc $a_1$ because $t^{\ac}_{i_1} = t^{\ac}_{i_2}$.
    Moreover, the arc $j_2i_1$ has periodic tension
	$x'_{j_2i_1} = (\pi_{i_1}' - \pi_{j_2}') \bmod T = (\pi_{i_1} - \pi_{j_2} - c) \bmod T = (\pi_{i_1} - \pi_{j_1} + \pi_{i_2} - \pi_{j_2} - x_{a_1}) \bmod T = x_{a_2}$, i.e., the overlap is also equal to the one on arc $a_2$ as $t^{\ac}_{i_1} = t^{\ac}_{i_2}$. Therefore, together the overlap on the two new arcs is the same as on the two old arcs. So we have found an optimal solution with fewer cycles, which constitutes a contradiction.
\end{proof}

\subsection{Special Case: Acceleration Time Equals Braking Time}

In the following, we investigate the case that each train accelerates as long as it brakes, i.e., $t_i^{\ac} = t_i^{\br}$ for all $i \in [n]$. Hence, we can denote the time that train $i$ needs for accelerating/braking simply by $t_i$. We assume that the $n$ trains in a one-station network are ordered such that $t_i \leq t_{i+1}$ for all $i \in [n-1]$. 
First we describe the structure of the matching induced by a maximum-weight Hamiltonian cycle.

\begin{lemma}\label{lem:t-matching}
	The matching $\energyact_H \coloneqq \{ ((i, \dep),(i-1, \arr)) \mid i \in \{2, \dots , n \} \} \cup \{((1, \dep),(n, \arr)) \}$ is of maximum weight w.r.t.\ $\w$ among the matchings $\energyact$ that result in exactly one cycle $C = \energyact\dotcup A_{\wait}$.
\end{lemma}

\begin{proof}
Let $\energyact^{*}$ be a matching of maximum weight such that $\energyact^{*} \dotcup A_{\wait}$ is a cycle. Now, we know that for the arcs in $\energyact^{*}$ incident to the nodes $(1, \dep)$ and $(1, \arr)$, we obtain the minimum weight $t_1$. Hence, for the second smallest train, we also know that the weight $t_2$ must be obtained on at least one of the arcs in $\energyact^{*}$ incident to $(2, \dep)$ or $(2, \arr)$. Otherwise, they would both be matched to nodes with a smaller weight, but such a matching would result in a closed cycle and contradict the assumption that $\energyact^{*} \dotcup A_{\wait}$ forms one cycle. Similar, for any $0 < i < n$, we can argue that $\energyact^{*}$ must contain at least $i+1$ arcs with a weight smaller or equal to the one of the $i^{\text{th}}$ smallest train. Finally, we know that, in case it holds $t_{n} > t_{n-1}$, there is no arc in $\energyact^{*}$ with the weight $t_n$. However, if it holds $t_{n} = t_{n-1} = \dots = t_{n-k}$, then there can be at most $k-1$ arcs in $\energyact^{*}$ with weight $t_n$.
This means that we cannot do better than obtaining the weights of the matching $\energyact_H$, i.e. $t_1 + \sum_{i=1}^{n-1} t_{i}$.
\end{proof}

Next, we want to calculate the optimal objective value of the overlap under the condition that the given trains are supposed to be matched into one cycle.

\begin{lemma}
\label{lem:t-matching-overlap}
	The matching $\energyact_H$ from \Cref{lem:t-matching} yields a cycle $C = \energyact_H \dotcup A_{\wait}$ with maximum overlap 
	  \[o(C) = \sum_{i =1}^{n} t_i - t_{n} + \max\{t_{1} - \delta_{C},\, 0 \},\]
	where $t_1$ is the smallest acceleration/braking time of all $i \in C$ and $t_n$ is the largest.      
\end{lemma}

\begin{proof}
We apply the formula in \cref{thm:struc-cycle} (b). Computing $\w(C)$ as in the proof of \cref{lem:t-matching},
\[ o(C) = \w(C)  - \min\{t_1, \delta_C\} = t_1 + \sum_{i=1}^{n-1} t_{i} - \min\{t_1, \delta_C\} = \sum_{i=1}^n t_i - t_n - \min\{0, t_1 - \delta_C\}.\]
\end{proof}

Now, we go back to the original problem and consider a matching $\energyact$ that maximizes the total overlap. Then for each connected component (cycle) of $\energyact \dotcup A_{\wait}$, we can assume that the cycle follows the order described above. Therefore, it suffices to specify which trains are contained in each cycle.
Let us now investigate the relation of the cycles of $\energyact \dotcup A_{\wait}$ to each other.

\begin{definition}
	Let $C_1, C_2 \subseteq \energyact\dotcup A_{\wait}$ be two cycles obtained from a matching of energy arcs. We denote by $s_1 = \text{arg}\min \{t_i \mid i \in C_1 \}$ the train with smallest acceleration/braking time in $C_1$ and by $s_2 = \text{arg}\min \{t_i \mid i \in C_2 \}$ the train with smallest acceleration/braking time in $C_2$. Without loss of generality, let $t_{s_1} \leq t_{s_2}$. Let $l_1 = \text{arg}\max \{t_i \mid i \in C_1 \}$ denote the train with largest acceleration/braking time in $C_1$. 
	
	$C_1$ and $C_2$ are called \textit{crossing cycles} if $t_{l_1} \geq t_{s_2}$.
\end{definition}

Note that for crossing cycles, we have $\min\{t_{l_1}, t_{l_2}\} \ge t_{s_2}$.

\begin{lemma}\label{lem:crossing}
	Let $C_1 \cup C_2 = \energyact\dotcup A_{\wait}$ be two crossing cycles obtained from a matching of energy arcs on a one-station network with $n$ trains. Then, the maximum overlap obtained by $\energyact$ is smaller or equal to the maximum overlap obtained by a maximum-weight Hamiltonian path matching $\energyact_H$, i.e. $o(\energyact) \leq o(\energyact_H)$.
\end{lemma}

\begin{proof}
	For the overlaps we obtain by \cref{lem:t-matching-overlap}
	\begin{align*}
		o(C_1) &= \sum_{i \in C_1} t_i - t_{l_1} + \max\{t_{s_1} - \delta_{C_1},\, 0 \}, \\
		o(C_2) &= \sum_{i \in C_2} t_i - t_{l_2} + \max\{t_{s_2} - \delta_{C_2},\, 0 \}, \\
		o(C) &= \sum_{i=1}^n t_i - t_n + \max\{t_1 - \delta_{C},\, 0 \}.
	\end{align*}
	This means
	\begin{align*}
		o(C_1) + o(C_2) = \sum_{i=1}^n t_i - t_{l_1} - t_{l_2} + \max\{t_{s_1} - \delta_{C_1},\, 0 \} + \max\{t_{s_2} - \delta_{C_2},\, 0 \}.
	\end{align*}
	We want to show that $o(C_1) + o(C_2) \leq o(C)$.
	
	\begin{itemize}
		\item If $t_{s_1} - \delta_{C_1} \leq 0$ and $t_{s_2} - \delta_{C_2} \leq 0$, then 
		\begin{align*}
			o(C_1) + o(C_2) = \sum_{i=1}^n t_i - t_{l_1} - t_{l_2} \leq \sum_{i=1}^n t_i - t_n \leq o(C),
		\end{align*}
		as one of $t_{l_1}$ and $t_{l_2}$ equals $t_n$.
		\item If $t_{s_1} - \delta_{C_1} \geq 0$ and $t_{s_2} - \delta_{C_2} \leq 0$, then 
		\begin{align*}
			o(C_1) + o(C_2) = \sum_{i=1}^n t_i - t_{l_1} - t_{l_2} + t_{s_1} - \delta_{C_1} \leq \sum_{i=1}^n t_i - t_n \leq o(C),
		\end{align*}
		because $\max\{t_{l_1}, t_{l_2}\} = t_n$, $\min\{t_{l_1}, t_{l_2}\} \geq t_{s_1}$, and $\delta_{C_1} \geq 0$.
		\item If $t_{s_1} - \delta_{C_1} \leq 0$ and $t_{s_2} - \delta_{C_2} \geq 0$, then 
		\begin{align*}
			o(C_1) + o(C_2) = \sum_{i=1}^n t_i - t_{l_1} - t_{l_2} + t_{s_2} - \delta_{C_2} \leq \sum_{i=1}^n t_i - t_n \leq o(C),
		\end{align*}
		because $\max\{t_{l_1}, t_{l_2}\} = t_n$, $\min\{t_{l_1}, t_{l_2}\} \geq t_{s_2}$ as $C_1$ and $C_2$ are crossing cycles, and $\delta_{C_2} \geq 0$.
		\item Finally, if $t_{s_1} - \delta_{C_1} \geq 0$ and $t_{s_2} - \delta_{C_2} \geq 0$, then
		\begin{align*}
			o(C_1) + o(C_2) &= \sum_{i=1}^n t_i - t_{l_1} - t_{l_2} + t_{s_1} - \delta_{C_1} + t_{s_2} - \delta_{C_2}\\ & = \sum_{i=1}^n t_i - t_{n} + t_{s_1} - (t^{\min}_{l} - t_{s_2}) - \delta_{C_1} - \delta_{C_2}  \\
            &= \sum_{i=1}^n t_i - t_{n} + t_1 - k  - \delta_{C_1} - \delta_{C_2},
		\end{align*}
		as it holds $t_{s_1} = t_1$ and $\max\{t_{l_1}, t_{l_2}\} = t_n$. Further, we define $t^{\min}_{l}:= \min\{t_{l_1}, t_{l_2}\} \geq t_{s_2}$, which holds as $C_1$ and $C_2$ are crossing cycles, and set $k \coloneqq t^{\min}_{l}- t_{s_2}$.
        If we consider the sum of the lower bounds and minima of the departure/arrival times on the energy arcs on cycles, we obtain
        \begin{align*}
            L_C &= L_{C_1} + L_{C_2} + k.
        \end{align*}

        Further, we know that that the sum of all lower and all upper bounds is independent of the matching. Hence, it holds
        \begin{align*}
            L_C + U_C &= L_{C_1} + U_{C_1} + L_{C_2} + U_{C_2}.
        \end{align*}
This yields
\begin{align*}
            U_C &= U_{C_1} + U_{C_2} - k.
        \end{align*}

        Hence, we get the following estimation for $\delta_C$:
        \begin{align*}
         \delta_C = d([L_C, U_C], T\Z) &=  d([L_{C_1} + L_{C_2} + k, U_{C_1} + U_{C_2} - k], T\Z) \\ 
        &\leq 
        d([L_{C_1} + L_{C_2}, U_{C_1} + U_{C_2}], T\Z) + k \\ 
        &\le  d([L_{C_1},\, U_{C_1}], T\Z) + d([L_{C_2}, U_{C_2}], T\Z) + k
        \\
        &= \delta_{C_1} + \delta_{C_2} + k.
        \end{align*}
		Therefore, we obtain the following result for the overlap of $C$:
		\begin{align*}
        o(C) &= \sum_{i=1}^{n} t_i - t_n + t_1 - \min\{t_1, \delta_C\}\\
        &\geq \sum_{i=1}^{n} t_i - t_n + t_1 - \min\{t_1, \delta_{C_1} + \delta_{C_2} + k\}\\
        &\geq \sum_{i=1}^{n} t_i - t_n + t_1  - k - \delta_{C_1} - \delta_{C_2} \\
        &= o(C_1) + o(C_2).
		\end{align*}
	\end{itemize}
\end{proof}

\begin{corollary}\label{cor:nocrossing}
    There is an optimal matching of energy activities that does not induce any crossing cycles.
\end{corollary}

We can exploit this knowledge in order to obtain an efficient algorithm. Any cycle in such a solution consists of subsequent trains according to the acceleration and braking times. Let $C_{s}^{l} = \{ s, s+1, \dots, l-1, l\}$ denote the cycle containing all trains from $s$ to $l$ with $s \leq l$. Then $o(C_{s}^{l}) = \sum_{i=s}^{l-1} t_i  + \max\{t_{s} - \delta_{C_s}^{l},\, 0 \}$.

\begin{theorem}
	The dynamic programming algorithm given in \cref{alg:dynamicprogram} computes an optimal solution to \pespenergy on a \osn with $n$ trains in $\mathcal{O}(n^2)$ timein the case that the acceleration time of each train equals its braking time.
\end{theorem}

\begin{algorithm}[ht]
	\caption{Dynamic Program}\label{alg:dynamicprogram}
	
	\KwIn{\pespenergy instance $(\cE_n^{\energy}, T, l, u, t^{\br}, t^{\ac})$ with $t^{\br}_i = t^{\ac}_i$ for all $i \in [n]$}
	\KwOut{optimal overlap }
	$\opt(0) \gets 0 $ \\
	\For{$i \leftarrow 1$ \KwTo $n$}{
		
		$\opt(i) \leftarrow \underset{j=1, \dots, i}{\max} \{ o(C_{j}^{i}) + \opt( j-1) \}$
	}
	
	\KwRet{$\opt(n)$}
\end{algorithm}

\begin{proof}
    We show that the optimum objective value~$\opt(i)$ for the first $i$ trains, $i \in \{0,\dotsc,n\}$, satisfies the recurrence relation used in the algorithm:
    \[\opt(i) = \begin{cases*} 0 &for $i=0$,\\ \max_{j=1,\dotsc,i} o(C_j^i) + \opt(j-1) &else.\end{cases*}\]
    This is trivial for $i=0$. For every $j$, combining the optimal solution for the first $j-1$ trains with the cycle $C_j^i$ yields a solution of value $o(C_j^i) + \opt(j-1)$, so $\opt(i) \ge o(C_j^i) + \opt(j-1)$ for all $j = 0,\dotsc,i$. On the other hand, by \cref{cor:nocrossing}, there is an optimal solution for the first $i$ items without crossing cycles. Therefore, it includes a cycle consisting of the trains $j,\dotsc,i$ for some $j \in \{0,\dotsc,i\}$, which, by \cref{lem:t-matching}, can be assumed to be $C_j^i$. Since the remainder of the optimal solution is a feasible solution for the first $j-1$ trains, its objective value is upper bounded by $\opt(j-1)$. Therefore, $\opt(i) \le o(C_j^i) + \opt(j-1)$.
    
    Now it follows by induction on $i$ that the algorithm correctly computes the optimal objective values of the subinstances.

    Second, let us show that the algorithm can be implemented to run in $\mathcal O(n^2)$ time. To this end, we store in iteration~$i$ the values $L_{C_j^i}$, $U_{C_j^i}$ and $T_j^i \coloneqq \sum_{k=j}^{i-1} t_k$ for $j=1,\dotsc,i$. Based on these three values, the overlap $o(C_j^i) = T_j^i + \max\{t_j - d([L_{C_j^i}, U_{C_j^i}], T\mathbb Z), 0\}$ can be computed in constant time. Also the new values in the next iteration~$i+1$ can be computed in constant time for each $j=1,\dotsc,i$ as 
    \begin{align*}
     L_{C_j^{i+1}} &= L_{C_j^i} + t_i + l_{((i+1,\arr),(i+1,\dep))},\\ 
     U_{C_j^{i+1}} &= U_{C_j^i} + t_{i+1} + u_{((i+1,\arr),(i+1,\dep))}, \\
     T_j^{i+1} &= T_j^i + t_i.
     \end{align*}
     Moreover, 
     \begin{align*}
      L_{C_{i+1}^{i+1}} &= t_{i+1} + l_{((i+1,\arr),(i+1,\dep))}, \\
      U_{C_{i+1}^{i+1}} &= t_{i+1} + u_{((i+1,\arr),(i+1,\dep))},\\ 
      T_{i+1}^{i+1} &= 0.
     \end{align*}
     Therefore, for every $1 \le j \le i \le n$, the algorithm requires constant running time, and there are $\mathcal O(n^2)$ such pairs.
\end{proof}

By backtracking the choices of the algorithm, an optimal solution can be reconstructed.

\section{Cycle-Restricted Variants of \pespenergy}
\label{sec-Variants}

While we were able to show that \pesppassenger is NP-hard on a \osn, for \pespenergy it is still open whether solutions can be found in polynomial time for one-station networks. 
In this section, we determine the complexity status of two variants of \pespenergy. These variants of the original problem demand a certain structure of the matching of energy arcs in the solution. They differ in the amount of cycles that are allowed to be formed by the (perfect) matching together with the waiting activities $\energyact\dotcup A_{\wait}$ in a feasible solution. 

\subsection{Exactly One Cycle of Trains Induced by $\energyact$}

In the following we will show that the cycle obtained by a maximum-weight Hamiltonian path has maximum overlap among all solutions whose matching induces a single cycle.

\begin{theorem}
    \label{thm:pesp-energy-single-cycle}
	Let $\energyact^H$ be a matching specifying the energy activities in a \pespenergy instance~$\mathcal E_n^\energy$ on a one-station network such that $C^H = \energyact^H \cup A_{\wait}$ is a Hamiltonian cycle in $\mathcal E_n^{\energy}$ of maximum weight w.r.t.\ $\w \coloneqq t^{\min}$.  Then, there is a timetable $\pi^H$ such that the solution $(\energyact^H, \pi^H, x^H, o^H)$ results in the maximum overlap among all solutions $(\energyact, \pi, x, o)$ that yield a single cycle $\energyact \dotcup A_{\wait}$.
\end{theorem}

\begin{proof}
    Let $(\energyact, \pi, x, o)$ be optimal such that $C \coloneqq \energyact \cup A_\wait$ is a single cycle. Since $\energyact$ and $\energyact^H$ are both perfect matchings,
    $$L_C + U_C = \sum_{a \in A_\wait} (l_a + u_a) + \sum_{i \in E_{\arr}} t_i^\br + \sum_{j \in E_{\dep}} t_j^\ac = L_{C^H} + U_{C^H}.$$
    Moreover, 
    $$ \min\{t_a^{\min} \mid a \in 
    \energyact\}  = \min\{ \min\{ t_i^\br \mid i \in E_{\arr}\},  \min\{t_j^\ac \mid j \in E_{\dep}\}\} = \min\{t_a^{\min} \mid a \in 
    \energyact^H\},$$
    and we call this quantity $\vartheta$.
    Since $C^H$ is of maximum weight, we have $k \coloneqq L_{C^H} - L_C = \w(C^H) - \w(C) \geq 0$, and therefore $U_{C^H} = L_C + U_C - L_{C^H} = U_C - k$, so that $[L_{C^H}, U_{C^H}] = [L_C+k, U_C-k] = [L_C, U_C]$. We derive that $\delta_{C^H} \leq \delta_C + k$ and, by \cref{thm:struc-cycle} (b), $$o(C^H) = \w(C^H) - \min\{\vartheta, \delta_{C_H}\} \geq \w(C) + k - \min\{\vartheta, \delta_C + k\} \geq \w(C) - \min\{\vartheta, \delta_C\} = o(C),$$
    so that $(\energyact^H, \pi^H, x^H, o^H)$ is optimal.
\end{proof}

By \cref{thm:findpath} it follows that this variant of \pespenergy can be solved to optimality in polynomial time.

\subsection{At Least Two Cycles of Trains Induced by $\energyact$}

Let us now consider the other case. We require every solution $S=(\energyact, \pi, x, o)$ to \pespenergy to contain at least two cycles in the union of $\energyact \dotcup A_{\wait}$. The following lemma shows that \pespenergy with this additional constraint on the solution structure is NP-complete on a \osn. 

\begin{theorem}
\label{thm:pesp-energy-two-cycles}
	It is NP-complete to decide whether there is a solution $S=(\energyact, \pi, x, o)$ to \pespenergy on a \osn with a total overlap bigger or equal to a given $O \in \R$ 
    such that $\energyact\dotcup A_{\wait}$ consists of at least two cycles.
\end{theorem}
\begin{proof}
    The proof works by reduction from the partition problem. For the partition problem we are given a set of natural numbers $c_1, \dots, c_n \in \mathbb N$. Now, we want to find a subset $I \subseteq \{1, \dots, n\}$ such that $\sum_{i \in I} c_i = \frac{1}{2}\sum_{i=1}^n c_i$. Therefore, we construct a \pespenergy instance on a \osn of size $n$ with $t^{\ac} = t^{\br}=1$ and $l_{a_i} = u_{a_i} = c_i -1$ for all $a_i \in A_{\wait}$. We choose the period time to be $T= \frac{1}{2}\sum_{i=1}^n c_i$. There is a solution to the partition problem if and only if there is a solution to the defined two-cycle \pespenergy instance with an overlap greater or equal to $n$. Given a solution $I$ to the partition problem, we can choose an arbitrary matching $\energyact$ such that $\energyact \dotcup A_{\wait}$ contains a cycle $C_I$ with exactly the waiting activities~$a_i$ for $i \in I$ and a cycle $C_{\overline{I}}$ containing exactly all the other waiting activities. Due to the choice of $T$, it is possible to achieve full overlap on $C_I$ by choosing $x_{a} = 1$ for all energy activities $a \in A_{\energy}$ because 
    $$\sum_{a \in C_I} x_a = \sum_{a_i \in C_I \cap A_{\wait}} (c_i - 1) + \sum_{a \in C_I \cap A_{\energy}} 1 = \sum_{i \in I} c_i = T.$$
    The same can be done with the cycle $C_{\overline{I}}$. Hence, we have found a solution to \pespenergy with an overlap of at least $n$. On the other hand, we know that for each cycle $C$ in a solution to \pespenergy with overlap $n$ we must have $\sum_{a \in C} x_a = kT$ for a $k \in \mathbb N_0$. As we are required to have at least two cycles, we have $k = 1$. Hence, we have found a valid partition.
\end{proof}

\section{Case Studies}
\label{sec:numerics}

In this section, we present two case studies to explore qualitative conclusions that can be drawn from our integrated bicriteria model \pesppassenergy, and to examine its computational behavior. Both case studies focus on a one-station network. The smaller one is artificial with four trains, while the second is placed in a realistic setting with 14 trains.

\subsection{Case Study 1: Artificial Network with Four Trains}

We begin with analyzing a small instance on a \osn $\mathcal E_4$ with $n = 4$ trains and a period time of $T = 20$.

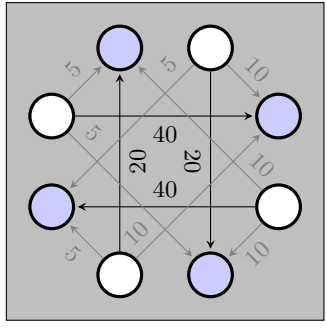
\begin{figure}[ht]
    \centering
	\begin{tikzpicture}[->,shorten >=1pt,auto, scale = 0.3]
    \tikzset{labelnode/.style={sloped, above, font=\footnotesize}}
    
    \filldraw[station] (0,0) rectangle ++ (14,14);
    
    
    \node[Arr] (4) at (12,5) {};
    \node[Dep] (1) at (2,5) {};
    
    \node[Arr] (2) at (2, 9){};
    \node[Dep] (3) at (12,9) {};
    
    \node[Arr] (6) at (5,2) {};
    \node[Dep] (5) at (5,12) {};
    
    \node[Arr] (7) at (9,12) {};
    \node[Dep] (8) at (9,2) {};
    
    \draw (4) edge node[labelnode, above] {40} (1); 
    \draw (6) edge node[labelnode, below] {20} (5); 
    \draw (7) edge node[labelnode, below] {20} (8);
    
    \draw[black] (2) edge node[labelnode,below] {40} (3);
    
    \draw[gray] (2) edge node[labelnode, pos=0.1] {5} (8);
    \draw[gray] (7) edge node[labelnode] {10} (3);
    \draw[gray] (7) edge node[labelnode, pos=0.1] {5} (1);
    \draw[gray] (4) edge node[labelnode, pos=0.1] {10} (5);
    \draw[gray] (4) edge node[labelnode, below] {10} (8);
    \draw[gray] (6) edge node[labelnode, pos=0.1] {10} (3);
    \draw[gray] (6) edge node[labelnode, below] {5} (1);
    \draw[gray] (2) edge node[labelnode] {5} (5);		
    
\end{tikzpicture}
	\caption{Artificial four-train instance with activity weights for the waiting (black) and transfer activities (gray), see \cite{atmos}.}
	\label{fig:mini-instance}
\end{figure}

\paragraph{Instance description.}
For all four waiting activities $a \in A_\wait$, we choose uniform bounds $l_\wait = 1$ and $u_\wait = 4$. We set a uniform minimum transfer time of $l_\transfer = 3$ for all transfer activities in $A_\transfer$ and set the upper bound to $l_\transfer + T - 1 = 22$. As acceleration and braking times, we choose $t^\ac = 5$ and $t^\br = 6$, respectively, for all 16 possible energy arcs $a \in A_\energy$. Finally, we assign a symmetric distribution of weights as indicated in \cref{fig:mini-instance}.

\paragraph{Computing the maximum overlap.}
The maximum overlap in the sense of the \pespenergy problem can be determined right away: In particular, since a perfect matching of energy arcs has size $n = 4$, the overlap cannot be larger than $n \cdot t^{\min} = 20$ by \Cref{prop:upperbound-matching}. Since all acceleration times are equal in our instance, by \Cref{prop:specialcase1}, there is an optimal solution consisting of one Hamiltonian cycle $C$. We compute $L_C = n \cdot l_\wait + n \cdot t^{\min} = 4 \cdot 1 + 4 \cdot 5 = 24$ and $U_C = n \cdot u_\wait + n \cdot t^{\max} = 4 \cdot 4 + 4 \cdot 6 = 40 = 2T$, so that $\delta_C = 0$ and hence $o(C) = \w(C) = 20$ by \Cref{thm:struc-cycle}.

\paragraph{Solving the instance.}
We solve the resulting \pesppassenergy instance using an $\varepsilon$\nobreakdash-constraint method, extending the mixed-integer program \eqref{mip:obj1}-\eqref{mip:pesp5} by lower bound constraints on the total overlap, and running the solver Gurobi~12 \cite{gurobi} on an Intel Xeon E3-1270 v6 CPU with 32 GB RAM. Obtaining all Pareto-optimal solutions takes less than two seconds, which makes this instance suitable for further sensitivity analysis.

\paragraph{Sensitivity analysis.}
To this end, based upon the nominal scenario described above, we vary the parameters $T, l_\transfer, l_\wait, u_\wait, t^\ac, t^\br$ individually and compute the corresponding Pareto fronts. The results are depicted in \cref{fig:mini-pareto}. When the period time $T$ is modified (cf.\ \cref{fig:mini-pareto-period}), we note that the travel time increases, as just missing a transfer -- being an effect of high overlap -- means to wait for almost a full period. For $T = 60$, we can only realize a maximum overlap of 15, as any Hamiltonian cycle $C$ now has $\delta_C = 9$, so that $o(C) = \rho(C) - \min\{t^{\min}, \delta_C\} = 20 - \min\{5, 9\} = 15$. In \cref{fig:mini-pareto-transfer}, we observe that the travel time is almost insensitive with respect to larger minimum transfer times for small overlaps and regarding smaller minimum transfer times for large overlaps. Our explanation is that for small overlaps, transfers can be kept short, so that only the increase in minimum travel time has to be paid for. On the other hand, for large overlaps, we see the opposite effect: Transfer activities will often be at their upper bound, as this provides the best overlaps, so that slightly decreasing this bound will also have only minor effects on the total travel time. For the variation of the waiting times (see \cref{fig:mini-pareto-min-wait,fig:mini-pareto-max-wait}), we see a comparatively large sensitivity for changing the lower bound~$l_\wait$, and much less for the upper bound~$u_\wait$. While $l_\wait$ directly impacts the in-vehicle time of many passengers, wider intervals $[l_\wait, u_\wait]$ offer lower travel times for a given overlap. The last two plots (\cref{fig:mini-pareto-ac,fig:mini-pareto-br}) show little difference compared to the nominal scenario for overlaps at most 16 and confirm the intuition that higher acceleration/braking times allow for larger overlaps and for better travel times for any given overlap.

\begin{figure}[p]
    \centering
    \begin{subfigure}{0.49\linewidth}
        \centering
        \begin{tikzpicture}[x=2.8mm,y=5mm]
	\tikzstyle{n} = [font=\footnotesize]
	\tikzstyle{p1} = [circle, draw, blue, fill=blue, inner sep=1]
	\tikzstyle{p2} = [circle, draw, green, fill=green, inner sep=1]
	\tikzstyle{p3} = [circle, draw, black, fill=black, inner sep=1]
	\tikzstyle{p4} = [circle, draw, red, fill=red, inner sep=1]
	\tikzstyle{p5} = [circle, draw, orange, fill=orange, inner sep=1]
	\draw[->] (-1,4) -- (-1,16);
	\draw[->] (-1,4) -- (21,4);
	\node[n, anchor=north, yshift=-15] at  (10.0,4) {brake-traction overlap};
	\node[n, rotate=90, anchor=south, yshift=35] at (0,10.0) {total passenger travel time};
	\draw[dotted, black!40] (-1,4) -- (21,4);
	\node[n, anchor=east] at (-1,4) {400};
	\draw[dotted, black!40] (-1,5) -- (21,5);
	\node[n, anchor=east] at (-1,5) {500};
	\draw[dotted, black!40] (-1,6) -- (21,6);
	\node[n, anchor=east] at (-1,6) {600};
	\draw[dotted, black!40] (-1,7) -- (21,7);
	\node[n, anchor=east] at (-1,7) {700};
	\draw[dotted, black!40] (-1,8) -- (21,8);
	\node[n, anchor=east] at (-1,8) {800};
	\draw[dotted, black!40] (-1,9) -- (21,9);
	\node[n, anchor=east] at (-1,9) {900};
	\draw[dotted, black!40] (-1,10) -- (21,10);
	\node[n, anchor=east] at (-1,10) {1000};
	\draw[dotted, black!40] (-1,11) -- (21,11);
	\node[n, anchor=east] at (-1,11) {1100};
	\draw[dotted, black!40] (-1,12) -- (21,12);
	\node[n, anchor=east] at (-1,12) {1200};
	\draw[dotted, black!40] (-1,13) -- (21,13);
	\node[n, anchor=east] at (-1,13) {1300};
	\draw[dotted, black!40] (-1,14) -- (21,14);
	\node[n, anchor=east] at (-1,14) {1400};
	\draw[dotted, black!40] (-1,15) -- (21,15);
	\node[n, anchor=east] at (-1,15) {1500};
	\draw[dotted, black!40] (-1,16) -- (21,16);
	\node[n, anchor=east] at (-1,16) {1600};
	\draw[dotted, black!40] (0,4) -- (0,16);
	\node[n, anchor=north] at (0,4) {0};
	\draw[dotted, black!40] (2,4) -- (2,16);
	\node[n, anchor=north] at (2,4) {2};
	\draw[dotted, black!40] (4,4) -- (4,16);
	\node[n, anchor=north] at (4,4) {4};
	\draw[dotted, black!40] (6,4) -- (6,16);
	\node[n, anchor=north] at (6,4) {6};
	\draw[dotted, black!40] (8,4) -- (8,16);
	\node[n, anchor=north] at (8,4) {8};
	\draw[dotted, black!40] (10,4) -- (10,16);
	\node[n, anchor=north] at (10,4) {10};
	\draw[dotted, black!40] (12,4) -- (12,16);
	\node[n, anchor=north] at (12,4) {12};
	\draw[dotted, black!40] (14,4) -- (14,16);
	\node[n, anchor=north] at (14,4) {14};
	\draw[dotted, black!40] (16,4) -- (16,16);
	\node[n, anchor=north] at (16,4) {16};
	\draw[dotted, black!40] (18,4) -- (18,16);
	\node[n, anchor=north] at (18,4) {18};
	\draw[dotted, black!40] (20,4) -- (20,16);
	\node[n, anchor=north] at (20,4) {20};
	\node[label={[n]0:$(T, l_\text{wait}, u_\text{wait}, l_\text{trans}, t^\text{ac}, t^\text{br})$}] (L0) at (0, 16) {};
	\node[p1, label={[n]0:$(10,1,4,3,5,6)$}, below=2.5mm of L0] (L1) {};
	\begin{pgfonlayer}{bg}
	\node[p1] at (0, 4.20) {}; 
	\node[p1] at (1, 4.20) {}; 
	\node[p1] at (2, 4.20) {}; 
	\node[p1] at (3, 4.20) {}; 
	\node[p1] at (4, 4.20) {}; 
	\node[p1] at (5, 4.20) {}; 
	\node[p1] at (6, 4.20) {}; 
	\node[p1] at (7, 4.20) {}; 
	\node[p1] at (8, 4.20) {}; 
	\node[p1] at (9, 4.20) {}; 
	\node[p1] at (10, 4.20) {}; 
	\node[p1] at (11, 4.20) {}; 
	\node[p1] at (12, 4.20) {}; 
	\node[p1] at (13, 4.20) {}; 
	\node[p1] at (14, 4.25) {}; 
	\node[p1] at (15, 4.30) {}; 
	\node[p1] at (16, 4.35) {}; 
	\node[p1] at (17, 4.40) {}; 
	\node[p1] at (18, 4.50) {}; 
	\node[p1] at (19, 4.85) {}; 
	\node[p1] at (20, 5.20) {}; 
	\end{pgfonlayer}
	\node[p2, label={[n]0:$(15,1,4,3,5,6)$}, below=2.5mm of L1] (L2) {};
	\begin{pgfonlayer}{bg}
	\node[p2] at (0, 5.00) {}; 
	\node[p2] at (1, 5.20) {}; 
	\node[p2] at (2, 5.20) {}; 
	\node[p2] at (3, 5.20) {}; 
	\node[p2] at (4, 5.20) {}; 
	\node[p2] at (5, 5.25) {}; 
	\node[p2] at (6, 5.30) {}; 
	\node[p2] at (7, 5.35) {}; 
	\node[p2] at (8, 5.40) {}; 
	\node[p2] at (9, 5.40) {}; 
	\node[p2] at (10, 5.43) {}; 
	\node[p2] at (11, 5.47) {}; 
	\node[p2] at (12, 5.50) {}; 
	\node[p2] at (13, 5.53) {}; 
	\node[p2] at (14, 5.55) {}; 
	\node[p2] at (15, 5.58) {}; 
	\node[p2] at (16, 5.60) {}; 
	\node[p2] at (17, 5.65) {}; 
	\node[p2] at (18, 5.70) {}; 
	\node[p2] at (19, 6.05) {}; 
	\node[p2] at (20, 6.40) {}; 
	\end{pgfonlayer}
	\node[p3, label={[n]0:$(20,1,4,3,5,6)$}, below=2.5mm of L2] (L3) {};
	\node[p3] at (0, 5.00) {}; 
	\node[p3] at (1, 6.20) {}; 
	\node[p3] at (2, 6.20) {}; 
	\node[p3] at (3, 6.20) {}; 
	\node[p3] at (4, 6.30) {}; 
	\node[p3] at (5, 6.40) {}; 
	\node[p3] at (6, 6.40) {}; 
	\node[p3] at (7, 6.45) {}; 
	\node[p3] at (8, 6.50) {}; 
	\node[p3] at (9, 6.55) {}; 
	\node[p3] at (10, 6.60) {}; 
	\node[p3] at (11, 6.70) {}; 
	\node[p3] at (12, 6.80) {}; 
	\node[p3] at (13, 6.90) {}; 
	\node[p3] at (14, 7.00) {}; 
	\node[p3] at (15, 7.10) {}; 
	\node[p3] at (16, 7.20) {}; 
	\node[p3] at (17, 8.45) {}; 
	\node[p3] at (18, 9.00) {}; 
	\node[p3] at (19, 9.55) {}; 
	\node[p3] at (20, 10.10) {}; 
	\node[p4, label={[n]0:$(30,1,4,3,5,6)$}, below=2.5mm of L3] (L4) {};
	\begin{pgfonlayer}{bg}
	\node[p4] at (0, 5.00) {}; 
	\node[p4] at (1, 7.40) {}; 
	\node[p4] at (2, 7.50) {}; 
	\node[p4] at (3, 7.60) {}; 
	\node[p4] at (4, 7.70) {}; 
	\node[p4] at (5, 7.80) {}; 
	\node[p4] at (6, 8.40) {}; 
	\node[p4] at (7, 8.45) {}; 
	\node[p4] at (8, 8.50) {}; 
	\node[p4] at (9, 8.55) {}; 
	\node[p4] at (10, 8.60) {}; 
	\node[p4] at (11, 9.20) {}; 
	\node[p4] at (12, 9.30) {}; 
	\node[p4] at (13, 9.40) {}; 
	\node[p4] at (14, 9.50) {}; 
	\node[p4] at (15, 9.60) {}; 
	\node[p4] at (16, 9.70) {}; 
	\node[p4] at (17, 9.80) {}; 
	\node[p4] at (18, 9.90) {}; 
	\node[p4] at (19, 10.35) {}; 
	\node[p4] at (20, 10.80) {}; 
	\end{pgfonlayer}
	\node[p5, label={[n]0:$(60,1,4,3,5,6)$}, below=2.5mm of L4] (L5) {};
	\begin{pgfonlayer}{bg}
	\node[p5] at (0, 5.00) {}; 
	\node[p5] at (1, 10.40) {}; 
	\node[p5] at (2, 10.50) {}; 
	\node[p5] at (3, 10.60) {}; 
	\node[p5] at (4, 10.70) {}; 
	\node[p5] at (5, 10.80) {}; 
	\node[p5] at (6, 11.60) {}; 
	\node[p5] at (7, 11.65) {}; 
	\node[p5] at (8, 11.70) {}; 
	\node[p5] at (9, 11.75) {}; 
	\node[p5] at (10, 11.80) {}; 
	\node[p5] at (11, 15.20) {}; 
	\node[p5] at (12, 15.30) {}; 
	\node[p5] at (13, 15.40) {}; 
	\node[p5] at (14, 15.50) {}; 
	\node[p5] at (15, 15.60) {}; 
	\end{pgfonlayer}
\end{tikzpicture}
        \caption{Variation of the period time $T$}
        \label{fig:mini-pareto-period}
    \end{subfigure}
    \hfill
    \begin{subfigure}{0.49\linewidth}
        \centering
        \begin{tikzpicture}[x=2.8mm,y=5mm]
	\tikzstyle{n} = [font=\footnotesize]
	\tikzstyle{p1} = [circle, draw, blue, fill=blue, inner sep=1]
	\tikzstyle{p2} = [circle, draw, green, fill=green, inner sep=1]
	\tikzstyle{p3} = [circle, draw, black, fill=black, inner sep=1]
	\tikzstyle{p4} = [circle, draw, red, fill=red, inner sep=1]
	\tikzstyle{p5} = [circle, draw, orange, fill=orange, inner sep=1]
	\draw[->] (-1,1) -- (-1,13);
	\draw[->] (-1,1) -- (21,1);
	\node[n, anchor=north, yshift=-15] at  (10.0,1) {brake-traction overlap};
	\node[n, rotate=90, anchor=south, yshift=35] at (0,7.0) {total passenger travel time};
	\draw[dotted, black!40] (-1,1) -- (21,1);
	\node[n, anchor=east] at (-1,1) {100};
	\draw[dotted, black!40] (-1,2) -- (21,2);
	\node[n, anchor=east] at (-1,2) {200};
	\draw[dotted, black!40] (-1,3) -- (21,3);
	\node[n, anchor=east] at (-1,3) {300};
	\draw[dotted, black!40] (-1,4) -- (21,4);
	\node[n, anchor=east] at (-1,4) {400};
	\draw[dotted, black!40] (-1,5) -- (21,5);
	\node[n, anchor=east] at (-1,5) {500};
	\draw[dotted, black!40] (-1,6) -- (21,6);
	\node[n, anchor=east] at (-1,6) {600};
	\draw[dotted, black!40] (-1,7) -- (21,7);
	\node[n, anchor=east] at (-1,7) {700};
	\draw[dotted, black!40] (-1,8) -- (21,8);
	\node[n, anchor=east] at (-1,8) {800};
	\draw[dotted, black!40] (-1,9) -- (21,9);
	\node[n, anchor=east] at (-1,9) {900};
	\draw[dotted, black!40] (-1,10) -- (21,10);
	\node[n, anchor=east] at (-1,10) {1000};
	\draw[dotted, black!40] (-1,11) -- (21,11);
	\node[n, anchor=east] at (-1,11) {1100};
	\draw[dotted, black!40] (-1,12) -- (21,12);
	\node[n, anchor=east] at (-1,12) {1200};
	\draw[dotted, black!40] (-1,13) -- (21,13);
	\node[n, anchor=east] at (-1,13) {1300};
	\draw[dotted, black!40] (0,1) -- (0,13);
	\node[n, anchor=north] at (0,1) {0};
	\draw[dotted, black!40] (2,1) -- (2,13);
	\node[n, anchor=north] at (2,1) {2};
	\draw[dotted, black!40] (4,1) -- (4,13);
	\node[n, anchor=north] at (4,1) {4};
	\draw[dotted, black!40] (6,1) -- (6,13);
	\node[n, anchor=north] at (6,1) {6};
	\draw[dotted, black!40] (8,1) -- (8,13);
	\node[n, anchor=north] at (8,1) {8};
	\draw[dotted, black!40] (10,1) -- (10,13);
	\node[n, anchor=north] at (10,1) {10};
	\draw[dotted, black!40] (12,1) -- (12,13);
	\node[n, anchor=north] at (12,1) {12};
	\draw[dotted, black!40] (14,1) -- (14,13);
	\node[n, anchor=north] at (14,1) {14};
	\draw[dotted, black!40] (16,1) -- (16,13);
	\node[n, anchor=north] at (16,1) {16};
	\draw[dotted, black!40] (18,1) -- (18,13);
	\node[n, anchor=north] at (18,1) {18};
	\draw[dotted, black!40] (20,1) -- (20,13);
	\node[n, anchor=north] at (20,1) {20};
	\node[label={[n]0:$(T, l_\text{wait}, u_\text{wait}, l_\text{trans}, t^\text{ac}, t^\text{br})$}] (L0) at (0, 13) {};
	\node[p1, label={[n]0:$(20,1,4,1,5,6)$}, below=2.5mm of L0] (L1) {};
	\begin{pgfonlayer}{bg}
	\node[p1] at (0, 1.80) {}; 
	\node[p1] at (1, 3.00) {}; 
	\node[p1] at (2, 3.60) {}; 
	\node[p1] at (3, 4.05) {}; 
	\node[p1] at (4, 4.10) {}; 
	\node[p1] at (5, 4.15) {}; 
	\node[p1] at (6, 4.20) {}; 
	\node[p1] at (7, 4.25) {}; 
	\node[p1] at (8, 4.30) {}; 
	\node[p1] at (9, 4.35) {}; 
	\node[p1] at (10, 4.40) {}; 
	\node[p1] at (11, 5.20) {}; 
	\node[p1] at (12, 5.55) {}; 
	\node[p1] at (13, 5.90) {}; 
	\node[p1] at (14, 6.00) {}; 
	\node[p1] at (15, 6.10) {}; 
	\node[p1] at (16, 7.20) {}; 
	\node[p1] at (17, 8.25) {}; 
	\node[p1] at (18, 8.80) {}; 
	\node[p1] at (19, 9.35) {}; 
	\node[p1] at (20, 9.90) {}; 
	\end{pgfonlayer}
	\node[p2, label={[n]0:$(20,1,4,2,5,6)$}, below=2.5mm of L1] (L2) {};
	\begin{pgfonlayer}{bg}
	\node[p2] at (0, 3.20) {}; 
	\node[p2] at (1, 5.10) {}; 
	\node[p2] at (2, 5.20) {}; 
	\node[p2] at (3, 5.30) {}; 
	\node[p2] at (4, 5.40) {}; 
	\node[p2] at (5, 5.50) {}; 
	\node[p2] at (6, 5.50) {}; 
	\node[p2] at (7, 5.75) {}; 
	\node[p2] at (8, 5.80) {}; 
	\node[p2] at (9, 5.85) {}; 
	\node[p2] at (10, 5.90) {}; 
	\node[p2] at (11, 5.90) {}; 
	\node[p2] at (12, 6.00) {}; 
	\node[p2] at (13, 6.35) {}; 
	\node[p2] at (14, 6.65) {}; 
	\node[p2] at (15, 6.70) {}; 
	\node[p2] at (16, 7.20) {}; 
	\node[p2] at (17, 8.35) {}; 
	\node[p2] at (18, 8.90) {}; 
	\node[p2] at (19, 9.45) {}; 
	\node[p2] at (20, 10.00) {}; 
	\end{pgfonlayer}
	\node[p3, label={[n]0:$(20,1,4,3,5,6)$}, below=2.5mm of L2] (L3) {};
	\node[p3] at (0, 5.00) {}; 
	\node[p3] at (1, 6.20) {}; 
	\node[p3] at (2, 6.20) {}; 
	\node[p3] at (3, 6.20) {}; 
	\node[p3] at (4, 6.30) {}; 
	\node[p3] at (5, 6.40) {}; 
	\node[p3] at (6, 6.40) {}; 
	\node[p3] at (7, 6.45) {}; 
	\node[p3] at (8, 6.50) {}; 
	\node[p3] at (9, 6.55) {}; 
	\node[p3] at (10, 6.60) {}; 
	\node[p3] at (11, 6.70) {}; 
	\node[p3] at (12, 6.80) {}; 
	\node[p3] at (13, 6.90) {}; 
	\node[p3] at (14, 7.00) {}; 
	\node[p3] at (15, 7.10) {}; 
	\node[p3] at (16, 7.20) {}; 
	\node[p3] at (17, 8.45) {}; 
	\node[p3] at (18, 9.00) {}; 
	\node[p3] at (19, 9.55) {}; 
	\node[p3] at (20, 10.10) {}; 
	\node[p4, label={[n]0:$(20,1,4,4,5,6)$}, below=2.5mm of L3] (L4) {};
	\begin{pgfonlayer}{bg}
	\node[p4] at (0, 6.40) {}; 
	\node[p4] at (1, 6.40) {}; 
	\node[p4] at (2, 6.40) {}; 
	\node[p4] at (3, 6.40) {}; 
	\node[p4] at (4, 6.40) {}; 
	\node[p4] at (5, 6.40) {}; 
	\node[p4] at (6, 6.50) {}; 
	\node[p4] at (7, 6.55) {}; 
	\node[p4] at (8, 6.60) {}; 
	\node[p4] at (9, 6.70) {}; 
	\node[p4] at (10, 6.70) {}; 
	\node[p4] at (11, 6.70) {}; 
	\node[p4] at (12, 6.80) {}; 
	\node[p4] at (13, 6.90) {}; 
	\node[p4] at (14, 7.00) {}; 
	\node[p4] at (15, 7.10) {}; 
	\node[p4] at (16, 7.20) {}; 
	\node[p4] at (17, 9.70) {}; 
	\node[p4] at (18, 10.10) {}; 
	\node[p4] at (19, 10.15) {}; 
	\node[p4] at (20, 10.20) {}; 
	\end{pgfonlayer}
	\node[p5, label={[n]0:$(20,1,4,5,5,6)$}, below=2.5mm of L4] (L5) {};
	\begin{pgfonlayer}{bg}
	\node[p5] at (0, 6.60) {}; 
	\node[p5] at (1, 6.60) {}; 
	\node[p5] at (2, 6.60) {}; 
	\node[p5] at (3, 6.60) {}; 
	\node[p5] at (4, 6.60) {}; 
	\node[p5] at (5, 6.60) {}; 
	\node[p5] at (6, 6.60) {}; 
	\node[p5] at (7, 6.70) {}; 
	\node[p5] at (8, 6.70) {}; 
	\node[p5] at (9, 6.70) {}; 
	\node[p5] at (10, 6.73) {}; 
	\node[p5] at (11, 6.77) {}; 
	\node[p5] at (12, 6.80) {}; 
	\node[p5] at (13, 6.90) {}; 
	\node[p5] at (14, 7.00) {}; 
	\node[p5] at (15, 7.10) {}; 
	\node[p5] at (16, 7.20) {}; 
	\node[p5] at (17, 10.65) {}; 
	\node[p5] at (18, 11.20) {}; 
	\node[p5] at (19, 11.75) {}; 
	\node[p5] at (20, 12.30) {}; 
	\end{pgfonlayer}
\end{tikzpicture}
        \caption{Variation of the minimum transfer time $l_\text{trans}$}
        \label{fig:mini-pareto-transfer}
    \end{subfigure}
    \begin{subfigure}{0.49\linewidth}
        \centering
        \begin{tikzpicture}[x=2.8mm,y=5mm]
	\tikzstyle{n} = [font=\footnotesize]
	\tikzstyle{p1} = [circle, draw, blue, fill=blue, inner sep=1]
	\tikzstyle{p2} = [circle, draw, green, fill=green, inner sep=1]
	\tikzstyle{p3} = [circle, draw, black, fill=black, inner sep=1]
	\tikzstyle{p4} = [circle, draw, red, fill=red, inner sep=1]
	\tikzstyle{p5} = [circle, draw, orange, fill=orange, inner sep=1]
	\draw[->] (-1,4) -- (-1,16);
	\draw[->] (-1,4) -- (21,4);
	\node[n, anchor=north, yshift=-15] at  (10.0,4) {brake-traction overlap};
	\node[n, rotate=90, anchor=south, yshift=35] at (0,10.0) {total passenger travel time};
	\draw[dotted, black!40] (-1,4) -- (21,4);
	\node[n, anchor=east] at (-1,4) {400};
	\draw[dotted, black!40] (-1,5) -- (21,5);
	\node[n, anchor=east] at (-1,5) {500};
	\draw[dotted, black!40] (-1,6) -- (21,6);
	\node[n, anchor=east] at (-1,6) {600};
	\draw[dotted, black!40] (-1,7) -- (21,7);
	\node[n, anchor=east] at (-1,7) {700};
	\draw[dotted, black!40] (-1,8) -- (21,8);
	\node[n, anchor=east] at (-1,8) {800};
	\draw[dotted, black!40] (-1,9) -- (21,9);
	\node[n, anchor=east] at (-1,9) {900};
	\draw[dotted, black!40] (-1,10) -- (21,10);
	\node[n, anchor=east] at (-1,10) {1000};
	\draw[dotted, black!40] (-1,11) -- (21,11);
	\node[n, anchor=east] at (-1,11) {1100};
	\draw[dotted, black!40] (-1,12) -- (21,12);
	\node[n, anchor=east] at (-1,12) {1200};
	\draw[dotted, black!40] (-1,13) -- (21,13);
	\node[n, anchor=east] at (-1,13) {1300};
	\draw[dotted, black!40] (-1,14) -- (21,14);
	\node[n, anchor=east] at (-1,14) {1400};
	\draw[dotted, black!40] (-1,15) -- (21,15);
	\node[n, anchor=east] at (-1,15) {1500};
	\draw[dotted, black!40] (-1,16) -- (21,16);
	\node[n, anchor=east] at (-1,16) {1600};
	\draw[dotted, black!40] (0,4) -- (0,16);
	\node[n, anchor=north] at (0,4) {0};
	\draw[dotted, black!40] (2,4) -- (2,16);
	\node[n, anchor=north] at (2,4) {2};
	\draw[dotted, black!40] (4,4) -- (4,16);
	\node[n, anchor=north] at (4,4) {4};
	\draw[dotted, black!40] (6,4) -- (6,16);
	\node[n, anchor=north] at (6,4) {6};
	\draw[dotted, black!40] (8,4) -- (8,16);
	\node[n, anchor=north] at (8,4) {8};
	\draw[dotted, black!40] (10,4) -- (10,16);
	\node[n, anchor=north] at (10,4) {10};
	\draw[dotted, black!40] (12,4) -- (12,16);
	\node[n, anchor=north] at (12,4) {12};
	\draw[dotted, black!40] (14,4) -- (14,16);
	\node[n, anchor=north] at (14,4) {14};
	\draw[dotted, black!40] (16,4) -- (16,16);
	\node[n, anchor=north] at (16,4) {16};
	\draw[dotted, black!40] (18,4) -- (18,16);
	\node[n, anchor=north] at (18,4) {18};
	\draw[dotted, black!40] (20,4) -- (20,16);
	\node[n, anchor=north] at (20,4) {20};
	\node[label={[n]0:$(T, l_\text{wait}, u_\text{wait}, l_\text{trans}, t^\text{ac}, t^\text{br})$}] (L0) at (0, 16) {};
	\node[p1, label={[n]0:$(20,0,0,3,5,6)$}, below=2.5mm of L0] (L1) {};
	\begin{pgfonlayer}{bg}
	\node[p1] at (0, 4.60) {}; 
	\node[p1] at (1, 4.60) {}; 
	\node[p1] at (2, 4.60) {}; 
	\node[p1] at (3, 4.60) {}; 
	\node[p1] at (4, 4.60) {}; 
	\node[p1] at (5, 4.60) {}; 
	\node[p1] at (6, 4.60) {}; 
	\node[p1] at (7, 4.70) {}; 
	\node[p1] at (8, 4.70) {}; 
	\node[p1] at (9, 4.75) {}; 
	\node[p1] at (10, 4.80) {}; 
	\node[p1] at (11, 5.00) {}; 
	\node[p1] at (12, 5.05) {}; 
	\node[p1] at (13, 5.10) {}; 
	\node[p1] at (14, 5.15) {}; 
	\node[p1] at (15, 5.20) {}; 
	\node[p1] at (16, 5.25) {}; 
	\node[p1] at (17, 5.30) {}; 
	\node[p1] at (18, 5.35) {}; 
	\node[p1] at (19, 5.40) {}; 
	\node[p1] at (20, 5.50) {}; 
	\end{pgfonlayer}
	\node[p2, label={[n]0:$(20,0,2,3,5,6)$}, below=2.5mm of L1] (L2) {};
	\begin{pgfonlayer}{bg}
	\node[p2] at (0, 4.60) {}; 
	\node[p2] at (1, 4.60) {}; 
	\node[p2] at (2, 4.60) {}; 
	\node[p2] at (3, 4.60) {}; 
	\node[p2] at (4, 4.60) {}; 
	\node[p2] at (5, 4.60) {}; 
	\node[p2] at (6, 4.60) {}; 
	\node[p2] at (7, 4.70) {}; 
	\node[p2] at (8, 4.70) {}; 
	\node[p2] at (9, 4.75) {}; 
	\node[p2] at (10, 4.80) {}; 
	\node[p2] at (11, 5.00) {}; 
	\node[p2] at (12, 5.05) {}; 
	\node[p2] at (13, 5.10) {}; 
	\node[p2] at (14, 5.15) {}; 
	\node[p2] at (15, 5.20) {}; 
	\node[p2] at (16, 5.25) {}; 
	\node[p2] at (17, 5.30) {}; 
	\node[p2] at (18, 5.35) {}; 
	\node[p2] at (19, 5.40) {}; 
	\node[p2] at (20, 5.50) {}; 
	\end{pgfonlayer}
	\node[p3, label={[n]0:$(20,1,4,3,5,6)$}, below=2.5mm of L2] (L3) {};
	\node[p3] at (0, 5.00) {}; 
	\node[p3] at (1, 6.20) {}; 
	\node[p3] at (2, 6.20) {}; 
	\node[p3] at (3, 6.20) {}; 
	\node[p3] at (4, 6.30) {}; 
	\node[p3] at (5, 6.40) {}; 
	\node[p3] at (6, 6.40) {}; 
	\node[p3] at (7, 6.45) {}; 
	\node[p3] at (8, 6.50) {}; 
	\node[p3] at (9, 6.55) {}; 
	\node[p3] at (10, 6.60) {}; 
	\node[p3] at (11, 6.70) {}; 
	\node[p3] at (12, 6.80) {}; 
	\node[p3] at (13, 6.90) {}; 
	\node[p3] at (14, 7.00) {}; 
	\node[p3] at (15, 7.10) {}; 
	\node[p3] at (16, 7.20) {}; 
	\node[p3] at (17, 8.45) {}; 
	\node[p3] at (18, 9.00) {}; 
	\node[p3] at (19, 9.55) {}; 
	\node[p3] at (20, 10.10) {}; 
	\node[p4, label={[n]0:$(20,3,6,3,5,6)$}, below=2.5mm of L3] (L4) {};
	\begin{pgfonlayer}{bg}
	\node[p4] at (0, 5.40) {}; 
	\node[p4] at (1, 7.80) {}; 
	\node[p4] at (2, 7.80) {}; 
	\node[p4] at (3, 7.85) {}; 
	\node[p4] at (4, 7.90) {}; 
	\node[p4] at (5, 7.95) {}; 
	\node[p4] at (6, 8.00) {}; 
	\node[p4] at (7, 8.05) {}; 
	\node[p4] at (8, 8.10) {}; 
	\node[p4] at (9, 8.10) {}; 
	\node[p4] at (10, 8.10) {}; 
	\node[p4] at (11, 8.12) {}; 
	\node[p4] at (12, 8.15) {}; 
	\node[p4] at (13, 8.18) {}; 
	\node[p4] at (14, 8.20) {}; 
	\node[p4] at (15, 8.25) {}; 
	\node[p4] at (16, 8.30) {}; 
	\node[p4] at (17, 8.65) {}; 
	\node[p4] at (18, 9.00) {}; 
	\node[p4] at (19, 9.35) {}; 
	\node[p4] at (20, 9.70) {}; 
	\end{pgfonlayer}
	\node[p5, label={[n]0:$(20,5,10,3,5,6)$}, below=2.5mm of L4] (L5) {};
	\begin{pgfonlayer}{bg}
	\node[p5] at (0, 8.60) {}; 
	\node[p5] at (1, 9.10) {}; 
	\node[p5] at (2, 9.60) {}; 
	\node[p5] at (3, 9.80) {}; 
	\node[p5] at (4, 10.10) {}; 
	\node[p5] at (5, 10.40) {}; 
	\node[p5] at (6, 10.70) {}; 
	\node[p5] at (7, 11.00) {}; 
	\node[p5] at (8, 11.18) {}; 
	\node[p5] at (9, 11.27) {}; 
	\node[p5] at (10, 11.35) {}; 
	\node[p5] at (11, 11.43) {}; 
	\node[p5] at (12, 11.52) {}; 
	\node[p5] at (13, 11.58) {}; 
	\node[p5] at (14, 11.60) {}; 
	\node[p5] at (15, 11.62) {}; 
	\node[p5] at (16, 11.65) {}; 
	\node[p5] at (17, 11.68) {}; 
	\node[p5] at (18, 11.70) {}; 
	\node[p5] at (19, 11.75) {}; 
	\node[p5] at (20, 11.80) {}; 
	\end{pgfonlayer}
\end{tikzpicture}
        \caption{Variation of the waiting interval $[l_\text{wait}, u_\text{wait}]$}
        \label{fig:mini-pareto-min-wait}
    \end{subfigure}
    \begin{subfigure}{0.49\linewidth}
        \centering
        \begin{tikzpicture}[x=2.8mm,y=5mm]
	\tikzstyle{n} = [font=\footnotesize]
	\tikzstyle{p1} = [circle, draw, blue, fill=blue, inner sep=1]
	\tikzstyle{p2} = [circle, draw, green, fill=green, inner sep=1]
	\tikzstyle{p3} = [circle, draw, black, fill=black, inner sep=1]
	\tikzstyle{p4} = [circle, draw, red, fill=red, inner sep=1]
	\tikzstyle{p5} = [circle, draw, orange, fill=orange, inner sep=1]
	\draw[->] (-1,4) -- (-1,16);
	\draw[->] (-1,4) -- (21,4);
	\node[n, anchor=north, yshift=-15] at  (10.0,4) {brake-traction overlap};
	\node[n, rotate=90, anchor=south, yshift=35] at (0,10.0) {total passenger travel time};
	\draw[dotted, black!40] (-1,4) -- (21,4);
	\node[n, anchor=east] at (-1,4) {400};
	\draw[dotted, black!40] (-1,5) -- (21,5);
	\node[n, anchor=east] at (-1,5) {500};
	\draw[dotted, black!40] (-1,6) -- (21,6);
	\node[n, anchor=east] at (-1,6) {600};
	\draw[dotted, black!40] (-1,7) -- (21,7);
	\node[n, anchor=east] at (-1,7) {700};
	\draw[dotted, black!40] (-1,8) -- (21,8);
	\node[n, anchor=east] at (-1,8) {800};
	\draw[dotted, black!40] (-1,9) -- (21,9);
	\node[n, anchor=east] at (-1,9) {900};
	\draw[dotted, black!40] (-1,10) -- (21,10);
	\node[n, anchor=east] at (-1,10) {1000};
	\draw[dotted, black!40] (-1,11) -- (21,11);
	\node[n, anchor=east] at (-1,11) {1100};
	\draw[dotted, black!40] (-1,12) -- (21,12);
	\node[n, anchor=east] at (-1,12) {1200};
	\draw[dotted, black!40] (-1,13) -- (21,13);
	\node[n, anchor=east] at (-1,13) {1300};
	\draw[dotted, black!40] (-1,14) -- (21,14);
	\node[n, anchor=east] at (-1,14) {1400};
	\draw[dotted, black!40] (-1,15) -- (21,15);
	\node[n, anchor=east] at (-1,15) {1500};
	\draw[dotted, black!40] (-1,16) -- (21,16);
	\node[n, anchor=east] at (-1,16) {1600};
	\draw[dotted, black!40] (0,4) -- (0,16);
	\node[n, anchor=north] at (0,4) {0};
	\draw[dotted, black!40] (2,4) -- (2,16);
	\node[n, anchor=north] at (2,4) {2};
	\draw[dotted, black!40] (4,4) -- (4,16);
	\node[n, anchor=north] at (4,4) {4};
	\draw[dotted, black!40] (6,4) -- (6,16);
	\node[n, anchor=north] at (6,4) {6};
	\draw[dotted, black!40] (8,4) -- (8,16);
	\node[n, anchor=north] at (8,4) {8};
	\draw[dotted, black!40] (10,4) -- (10,16);
	\node[n, anchor=north] at (10,4) {10};
	\draw[dotted, black!40] (12,4) -- (12,16);
	\node[n, anchor=north] at (12,4) {12};
	\draw[dotted, black!40] (14,4) -- (14,16);
	\node[n, anchor=north] at (14,4) {14};
	\draw[dotted, black!40] (16,4) -- (16,16);
	\node[n, anchor=north] at (16,4) {16};
	\draw[dotted, black!40] (18,4) -- (18,16);
	\node[n, anchor=north] at (18,4) {18};
	\draw[dotted, black!40] (20,4) -- (20,16);
	\node[n, anchor=north] at (20,4) {20};
	\node[label={[n]0:$(T, l_\text{wait}, u_\text{wait}, l_\text{trans}, t^\text{ac}, t^\text{br})$}] (L0) at (0, 16) {};
	\node[p1, label={[n]0:$(20,1,1,3,5,6)$}, below=2.5mm of L0] (L1) {};
	\begin{pgfonlayer}{bg}
	\node[p1] at (0, 6.20) {}; 
	\node[p1] at (1, 6.20) {}; 
	\node[p1] at (2, 6.20) {}; 
	\node[p1] at (3, 6.20) {}; 
	\node[p1] at (4, 6.30) {}; 
	\node[p1] at (5, 6.40) {}; 
	\node[p1] at (6, 6.40) {}; 
	\node[p1] at (7, 6.45) {}; 
	\node[p1] at (8, 6.50) {}; 
	\node[p1] at (9, 6.55) {}; 
	\node[p1] at (10, 6.60) {}; 
	\node[p1] at (11, 6.70) {}; 
	\node[p1] at (12, 6.80) {}; 
	\node[p1] at (13, 6.90) {}; 
	\node[p1] at (14, 7.00) {}; 
	\node[p1] at (15, 7.10) {}; 
	\node[p1] at (16, 7.20) {}; 
	\end{pgfonlayer}
	\node[p2, label={[n]0:$(20,1,2,3,5,6)$}, below=2.5mm of L1] (L2) {};
	\begin{pgfonlayer}{bg}
	\node[p2] at (0, 6.20) {}; 
	\node[p2] at (1, 6.20) {}; 
	\node[p2] at (2, 6.20) {}; 
	\node[p2] at (3, 6.20) {}; 
	\node[p2] at (4, 6.30) {}; 
	\node[p2] at (5, 6.40) {}; 
	\node[p2] at (6, 6.40) {}; 
	\node[p2] at (7, 6.45) {}; 
	\node[p2] at (8, 6.50) {}; 
	\node[p2] at (9, 6.55) {}; 
	\node[p2] at (10, 6.60) {}; 
	\node[p2] at (11, 6.70) {}; 
	\node[p2] at (12, 6.80) {}; 
	\node[p2] at (13, 6.90) {}; 
	\node[p2] at (14, 7.00) {}; 
	\node[p2] at (15, 7.10) {}; 
	\node[p2] at (16, 7.20) {}; 
	\end{pgfonlayer}
	\node[p3, label={[n]0:$(20,1,4,3,5,6)$}, below=2.5mm of L2] (L3) {};
	\node[p3] at (0, 5.00) {}; 
	\node[p3] at (1, 6.20) {}; 
	\node[p3] at (2, 6.20) {}; 
	\node[p3] at (3, 6.20) {}; 
	\node[p3] at (4, 6.30) {}; 
	\node[p3] at (5, 6.40) {}; 
	\node[p3] at (6, 6.40) {}; 
	\node[p3] at (7, 6.45) {}; 
	\node[p3] at (8, 6.50) {}; 
	\node[p3] at (9, 6.55) {}; 
	\node[p3] at (10, 6.60) {}; 
	\node[p3] at (11, 6.70) {}; 
	\node[p3] at (12, 6.80) {}; 
	\node[p3] at (13, 6.90) {}; 
	\node[p3] at (14, 7.00) {}; 
	\node[p3] at (15, 7.10) {}; 
	\node[p3] at (16, 7.20) {}; 
	\node[p3] at (17, 8.45) {}; 
	\node[p3] at (18, 9.00) {}; 
	\node[p3] at (19, 9.55) {}; 
	\node[p3] at (20, 10.10) {}; 
	\node[p4, label={[n]0:$(20,1,6,3,5,6)$}, below=2.5mm of L3] (L4) {};
	\begin{pgfonlayer}{bg}
	\node[p4] at (0, 4.60) {}; 
	\node[p4] at (1, 6.20) {}; 
	\node[p4] at (2, 6.20) {}; 
	\node[p4] at (3, 6.20) {}; 
	\node[p4] at (4, 6.20) {}; 
	\node[p4] at (5, 6.20) {}; 
	\node[p4] at (6, 6.20) {}; 
	\node[p4] at (7, 6.20) {}; 
	\node[p4] at (8, 6.20) {}; 
	\node[p4] at (9, 6.45) {}; 
	\node[p4] at (10, 6.60) {}; 
	\node[p4] at (11, 6.70) {}; 
	\node[p4] at (12, 6.80) {}; 
	\node[p4] at (13, 6.90) {}; 
	\node[p4] at (14, 7.00) {}; 
	\node[p4] at (15, 7.10) {}; 
	\node[p4] at (16, 7.20) {}; 
	\node[p4] at (17, 7.85) {}; 
	\node[p4] at (18, 8.20) {}; 
	\node[p4] at (19, 8.75) {}; 
	\node[p4] at (20, 9.30) {}; 
	\end{pgfonlayer}
	\node[p5, label={[n]0:$(20,1,10,3,5,6)$}, below=2.5mm of L4] (L5) {};
	\begin{pgfonlayer}{bg}
	\node[p5] at (0, 4.60) {}; 
	\node[p5] at (1, 5.80) {}; 
	\node[p5] at (2, 6.20) {}; 
	\node[p5] at (3, 6.20) {}; 
	\node[p5] at (4, 6.20) {}; 
	\node[p5] at (5, 6.20) {}; 
	\node[p5] at (6, 6.20) {}; 
	\node[p5] at (7, 6.20) {}; 
	\node[p5] at (8, 6.20) {}; 
	\node[p5] at (9, 6.45) {}; 
	\node[p5] at (10, 6.60) {}; 
	\node[p5] at (11, 6.70) {}; 
	\node[p5] at (12, 6.80) {}; 
	\node[p5] at (13, 6.90) {}; 
	\node[p5] at (14, 7.00) {}; 
	\node[p5] at (15, 7.10) {}; 
	\node[p5] at (16, 7.20) {}; 
	\node[p5] at (17, 7.85) {}; 
	\node[p5] at (18, 8.20) {}; 
	\node[p5] at (19, 8.55) {}; 
	\node[p5] at (20, 8.90) {}; 
	\end{pgfonlayer}
\end{tikzpicture}
        \caption{Variation of the maximum waiting time $u_\text{wait}$}
        \label{fig:mini-pareto-max-wait}
    \end{subfigure}
    \begin{subfigure}{0.49\linewidth}
        \centering
        \begin{tikzpicture}[x=2.3mm,y=5mm]
	\tikzstyle{n} = [font=\footnotesize]
	\tikzstyle{p1} = [circle, draw, blue, fill=blue, inner sep=1]
	\tikzstyle{p2} = [circle, draw, green, fill=green, inner sep=1]
	\tikzstyle{p3} = [circle, draw, black, fill=black, inner sep=1]
	\tikzstyle{p4} = [circle, draw, red, fill=red, inner sep=1]
	\tikzstyle{p5} = [circle, draw, orange, fill=orange, inner sep=1]
	\draw[->] (-1,5) -- (-1,12);
	\draw[->] (-1,5) -- (25,5);
	\node[n, anchor=north, yshift=-15] at  (12.0,5) {brake-traction overlap};
	\node[n, rotate=90, anchor=south, yshift=35] at (0,8.5) {total passenger travel time};
	\draw[dotted, black!40] (-1,5) -- (25,5);
	\node[n, anchor=east] at (-1,5) {500};
	\draw[dotted, black!40] (-1,6) -- (25,6);
	\node[n, anchor=east] at (-1,6) {600};
	\draw[dotted, black!40] (-1,7) -- (25,7);
	\node[n, anchor=east] at (-1,7) {700};
	\draw[dotted, black!40] (-1,8) -- (25,8);
	\node[n, anchor=east] at (-1,8) {800};
	\draw[dotted, black!40] (-1,9) -- (25,9);
	\node[n, anchor=east] at (-1,9) {900};
	\draw[dotted, black!40] (-1,10) -- (25,10);
	\node[n, anchor=east] at (-1,10) {1000};
	\draw[dotted, black!40] (-1,11) -- (25,11);
	\node[n, anchor=east] at (-1,11) {1100};
	\draw[dotted, black!40] (-1,12) -- (25,12);
	\node[n, anchor=east] at (-1,12) {1200};
	\draw[dotted, black!40] (0,5) -- (0,12);
	\node[n, anchor=north] at (0,5) {0};
	\draw[dotted, black!40] (2,5) -- (2,12);
	\node[n, anchor=north] at (2,5) {2};
	\draw[dotted, black!40] (4,5) -- (4,12);
	\node[n, anchor=north] at (4,5) {4};
	\draw[dotted, black!40] (6,5) -- (6,12);
	\node[n, anchor=north] at (6,5) {6};
	\draw[dotted, black!40] (8,5) -- (8,12);
	\node[n, anchor=north] at (8,5) {8};
	\draw[dotted, black!40] (10,5) -- (10,12);
	\node[n, anchor=north] at (10,5) {10};
	\draw[dotted, black!40] (12,5) -- (12,12);
	\node[n, anchor=north] at (12,5) {12};
	\draw[dotted, black!40] (14,5) -- (14,12);
	\node[n, anchor=north] at (14,5) {14};
	\draw[dotted, black!40] (16,5) -- (16,12);
	\node[n, anchor=north] at (16,5) {16};
	\draw[dotted, black!40] (18,5) -- (18,12);
	\node[n, anchor=north] at (18,5) {18};
	\draw[dotted, black!40] (20,5) -- (20,12);
	\node[n, anchor=north] at (20,5) {20};
	\draw[dotted, black!40] (22,5) -- (22,12);
	\node[n, anchor=north] at (22,5) {22};
	\draw[dotted, black!40] (24,5) -- (24,12);
	\node[n, anchor=north] at (24,5) {24};
	\node[label={[n]0:$(T, l_\text{wait}, u_\text{wait}, l_\text{trans}, t^\text{ac}, t^\text{br})$}] (L0) at (0, 12) {};
	\node[p1, label={[n]0:$(20,1,4,3,3,6)$}, below=2.5mm of L0] (L1) {};
	\begin{pgfonlayer}{bg}
	\node[p1] at (0, 5.00) {}; 
	\node[p1] at (1, 6.20) {}; 
	\node[p1] at (2, 6.20) {}; 
	\node[p1] at (3, 6.20) {}; 
	\node[p1] at (4, 6.30) {}; 
	\node[p1] at (5, 6.40) {}; 
	\node[p1] at (6, 6.40) {}; 
	\node[p1] at (7, 6.80) {}; 
	\node[p1] at (8, 6.90) {}; 
	\node[p1] at (9, 6.90) {}; 
	\node[p1] at (10, 6.95) {}; 
	\node[p1] at (11, 7.00) {}; 
	\node[p1] at (12, 7.10) {}; 
	\end{pgfonlayer}
	\node[p2, label={[n]0:$(20,1,4,3,4,6)$}, below=2.5mm of L1] (L2) {};
	\begin{pgfonlayer}{bg}
	\node[p2] at (0, 5.00) {}; 
	\node[p2] at (1, 6.20) {}; 
	\node[p2] at (2, 6.20) {}; 
	\node[p2] at (3, 6.20) {}; 
	\node[p2] at (4, 6.30) {}; 
	\node[p2] at (5, 6.40) {}; 
	\node[p2] at (6, 6.40) {}; 
	\node[p2] at (7, 6.45) {}; 
	\node[p2] at (8, 6.50) {}; 
	\node[p2] at (9, 6.80) {}; 
	\node[p2] at (10, 6.90) {}; 
	\node[p2] at (11, 6.95) {}; 
	\node[p2] at (12, 7.00) {}; 
	\node[p2] at (13, 7.05) {}; 
	\node[p2] at (14, 7.10) {}; 
	\node[p2] at (15, 7.20) {}; 
	\node[p2] at (16, 7.30) {}; 
	\end{pgfonlayer}
	\node[p3, label={[n]0:$(20,1,4,3,5,6)$}, below=2.5mm of L2] (L3) {};
	\node[p3] at (0, 5.00) {}; 
	\node[p3] at (1, 6.20) {}; 
	\node[p3] at (2, 6.20) {}; 
	\node[p3] at (3, 6.20) {}; 
	\node[p3] at (4, 6.30) {}; 
	\node[p3] at (5, 6.40) {}; 
	\node[p3] at (6, 6.40) {}; 
	\node[p3] at (7, 6.45) {}; 
	\node[p3] at (8, 6.50) {}; 
	\node[p3] at (9, 6.55) {}; 
	\node[p3] at (10, 6.60) {}; 
	\node[p3] at (11, 6.70) {}; 
	\node[p3] at (12, 6.80) {}; 
	\node[p3] at (13, 6.90) {}; 
	\node[p3] at (14, 7.00) {}; 
	\node[p3] at (15, 7.10) {}; 
	\node[p3] at (16, 7.20) {}; 
	\node[p3] at (17, 8.45) {}; 
	\node[p3] at (18, 9.00) {}; 
	\node[p3] at (19, 9.55) {}; 
	\node[p3] at (20, 10.10) {}; 
	\node[p4, label={[n]0:$(20,1,4,3,6,6)$}, below=2.5mm of L3] (L4) {};
	\begin{pgfonlayer}{bg}
	\node[p4] at (0, 5.00) {}; 
	\node[p4] at (1, 6.20) {}; 
	\node[p4] at (2, 6.20) {}; 
	\node[p4] at (3, 6.20) {}; 
	\node[p4] at (4, 6.30) {}; 
	\node[p4] at (5, 6.40) {}; 
	\node[p4] at (6, 6.40) {}; 
	\node[p4] at (7, 6.45) {}; 
	\node[p4] at (8, 6.50) {}; 
	\node[p4] at (9, 6.55) {}; 
	\node[p4] at (10, 6.60) {}; 
	\node[p4] at (11, 6.60) {}; 
	\node[p4] at (12, 6.63) {}; 
	\node[p4] at (13, 6.67) {}; 
	\node[p4] at (14, 6.70) {}; 
	\node[p4] at (15, 6.80) {}; 
	\node[p4] at (16, 6.90) {}; 
	\node[p4] at (17, 7.00) {}; 
	\node[p4] at (18, 7.10) {}; 
	\node[p4] at (19, 7.88) {}; 
	\node[p4] at (20, 7.90) {}; 
	\node[p4] at (21, 8.45) {}; 
	\node[p4] at (22, 9.00) {}; 
	\node[p4] at (23, 9.55) {}; 
	\node[p4] at (24, 10.10) {}; 
	\end{pgfonlayer}
	\node[p5, label={[n]0:$(20,1,4,3,7,6)$}, below=2.5mm of L4] (L5) {};
	\begin{pgfonlayer}{bg}
	\node[p5] at (0, 5.00) {}; 
	\node[p5] at (1, 6.20) {}; 
	\node[p5] at (2, 6.20) {}; 
	\node[p5] at (3, 6.20) {}; 
	\node[p5] at (4, 6.30) {}; 
	\node[p5] at (5, 6.40) {}; 
	\node[p5] at (6, 6.40) {}; 
	\node[p5] at (7, 6.45) {}; 
	\node[p5] at (8, 6.50) {}; 
	\node[p5] at (9, 6.50) {}; 
	\node[p5] at (10, 6.53) {}; 
	\node[p5] at (11, 6.57) {}; 
	\node[p5] at (12, 6.60) {}; 
	\node[p5] at (13, 6.63) {}; 
	\node[p5] at (14, 6.67) {}; 
	\node[p5] at (15, 6.70) {}; 
	\node[p5] at (16, 6.80) {}; 
	\node[p5] at (17, 6.90) {}; 
	\node[p5] at (18, 7.00) {}; 
	\node[p5] at (19, 7.35) {}; 
	\node[p5] at (20, 7.45) {}; 
	\node[p5] at (21, 7.80) {}; 
	\node[p5] at (22, 7.90) {}; 
	\node[p5] at (23, 7.95) {}; 
	\node[p5] at (24, 8.00) {}; 
	\end{pgfonlayer}
\end{tikzpicture}
        \caption{Variation of the acceleration time $t^\ac$}
        \label{fig:mini-pareto-ac}
    \end{subfigure}
    \begin{subfigure}{0.49\linewidth}
        \centering
        \begin{tikzpicture}[x=2.8mm,y=5mm]
	\tikzstyle{n} = [font=\footnotesize]
	\tikzstyle{p1} = [circle, draw, blue, fill=blue, inner sep=1]
	\tikzstyle{p2} = [circle, draw, green, fill=green, inner sep=1]
	\tikzstyle{p3} = [circle, draw, black, fill=black, inner sep=1]
	\tikzstyle{p4} = [circle, draw, red, fill=red, inner sep=1]
	\tikzstyle{p5} = [circle, draw, orange, fill=orange, inner sep=1]
	\draw[->] (-1,5) -- (-1,12);
	\draw[->] (-1,5) -- (21,5);
	\node[n, anchor=north, yshift=-15] at  (10.0,5) {brake-traction overlap};
	\node[n, rotate=90, anchor=south, yshift=35] at (0,8.5) {total passenger travel time};
	\draw[dotted, black!40] (-1,5) -- (21,5);
	\node[n, anchor=east] at (-1,5) {500};
	\draw[dotted, black!40] (-1,6) -- (21,6);
	\node[n, anchor=east] at (-1,6) {600};
	\draw[dotted, black!40] (-1,7) -- (21,7);
	\node[n, anchor=east] at (-1,7) {700};
	\draw[dotted, black!40] (-1,8) -- (21,8);
	\node[n, anchor=east] at (-1,8) {800};
	\draw[dotted, black!40] (-1,9) -- (21,9);
	\node[n, anchor=east] at (-1,9) {900};
	\draw[dotted, black!40] (-1,10) -- (21,10);
	\node[n, anchor=east] at (-1,10) {1000};
	\draw[dotted, black!40] (-1,11) -- (21,11);
	\node[n, anchor=east] at (-1,11) {1100};
	\draw[dotted, black!40] (-1,12) -- (21,12);
	\node[n, anchor=east] at (-1,12) {1200};
	\draw[dotted, black!40] (0,5) -- (0,12);
	\node[n, anchor=north] at (0,5) {0};
	\draw[dotted, black!40] (2,5) -- (2,12);
	\node[n, anchor=north] at (2,5) {2};
	\draw[dotted, black!40] (4,5) -- (4,12);
	\node[n, anchor=north] at (4,5) {4};
	\draw[dotted, black!40] (6,5) -- (6,12);
	\node[n, anchor=north] at (6,5) {6};
	\draw[dotted, black!40] (8,5) -- (8,12);
	\node[n, anchor=north] at (8,5) {8};
	\draw[dotted, black!40] (10,5) -- (10,12);
	\node[n, anchor=north] at (10,5) {10};
	\draw[dotted, black!40] (12,5) -- (12,12);
	\node[n, anchor=north] at (12,5) {12};
	\draw[dotted, black!40] (14,5) -- (14,12);
	\node[n, anchor=north] at (14,5) {14};
	\draw[dotted, black!40] (16,5) -- (16,12);
	\node[n, anchor=north] at (16,5) {16};
	\draw[dotted, black!40] (18,5) -- (18,12);
	\node[n, anchor=north] at (18,5) {18};
	\draw[dotted, black!40] (20,5) -- (20,12);
	\node[n, anchor=north] at (20,5) {20};
	\node[label={[n]0:$(T, l_\text{wait}, u_\text{wait}, l_\text{trans}, t^\text{ac}, t^\text{br})$}] (L0) at (0, 12) {};
	\node[p1, label={[n]0:$(20,1,4,3,5,4)$}, below=2.5mm of L0] (L1) {};
	\begin{pgfonlayer}{bg}
	\node[p1] at (0, 5.00) {}; 
	\node[p1] at (1, 6.20) {}; 
	\node[p1] at (2, 6.20) {}; 
	\node[p1] at (3, 6.20) {}; 
	\node[p1] at (4, 6.30) {}; 
	\node[p1] at (5, 6.40) {}; 
	\node[p1] at (6, 6.40) {}; 
	\node[p1] at (7, 6.45) {}; 
	\node[p1] at (8, 6.50) {}; 
	\node[p1] at (9, 6.90) {}; 
	\node[p1] at (10, 6.95) {}; 
	\node[p1] at (11, 7.00) {}; 
	\node[p1] at (12, 7.05) {}; 
	\node[p1] at (13, 7.10) {}; 
	\node[p1] at (14, 7.15) {}; 
	\node[p1] at (15, 7.20) {}; 
	\node[p1] at (16, 7.30) {}; 
	\end{pgfonlayer}
	\node[p2, label={[n]0:$(20,1,4,3,5,5)$}, below=2.5mm of L1] (L2) {};
	\begin{pgfonlayer}{bg}
	\node[p2] at (0, 5.00) {}; 
	\node[p2] at (1, 6.20) {}; 
	\node[p2] at (2, 6.20) {}; 
	\node[p2] at (3, 6.20) {}; 
	\node[p2] at (4, 6.30) {}; 
	\node[p2] at (5, 6.40) {}; 
	\node[p2] at (6, 6.40) {}; 
	\node[p2] at (7, 6.45) {}; 
	\node[p2] at (8, 6.50) {}; 
	\node[p2] at (9, 6.55) {}; 
	\node[p2] at (10, 6.60) {}; 
	\node[p2] at (11, 6.80) {}; 
	\node[p2] at (12, 6.90) {}; 
	\node[p2] at (13, 7.00) {}; 
	\node[p2] at (14, 7.10) {}; 
	\node[p2] at (15, 7.15) {}; 
	\node[p2] at (16, 7.20) {}; 
	\end{pgfonlayer}
	\node[p3, label={[n]0:$(20,1,4,3,5,6)$}, below=2.5mm of L2] (L3) {};
	\node[p3] at (0, 5.00) {}; 
	\node[p3] at (1, 6.20) {}; 
	\node[p3] at (2, 6.20) {}; 
	\node[p3] at (3, 6.20) {}; 
	\node[p3] at (4, 6.30) {}; 
	\node[p3] at (5, 6.40) {}; 
	\node[p3] at (6, 6.40) {}; 
	\node[p3] at (7, 6.45) {}; 
	\node[p3] at (8, 6.50) {}; 
	\node[p3] at (9, 6.55) {}; 
	\node[p3] at (10, 6.60) {}; 
	\node[p3] at (11, 6.70) {}; 
	\node[p3] at (12, 6.80) {}; 
	\node[p3] at (13, 6.90) {}; 
	\node[p3] at (14, 7.00) {}; 
	\node[p3] at (15, 7.10) {}; 
	\node[p3] at (16, 7.20) {}; 
	\node[p3] at (17, 8.45) {}; 
	\node[p3] at (18, 9.00) {}; 
	\node[p3] at (19, 9.55) {}; 
	\node[p3] at (20, 10.10) {}; 
	\node[p4, label={[n]0:$(20,1,4,3,5,7)$}, below=2.5mm of L3] (L4) {};
	\begin{pgfonlayer}{bg}
	\node[p4] at (0, 5.00) {}; 
	\node[p4] at (1, 6.20) {}; 
	\node[p4] at (2, 6.20) {}; 
	\node[p4] at (3, 6.20) {}; 
	\node[p4] at (4, 6.30) {}; 
	\node[p4] at (5, 6.40) {}; 
	\node[p4] at (6, 6.40) {}; 
	\node[p4] at (7, 6.45) {}; 
	\node[p4] at (8, 6.50) {}; 
	\node[p4] at (9, 6.55) {}; 
	\node[p4] at (10, 6.60) {}; 
	\node[p4] at (11, 6.60) {}; 
	\node[p4] at (12, 6.70) {}; 
	\node[p4] at (13, 6.80) {}; 
	\node[p4] at (14, 6.90) {}; 
	\node[p4] at (15, 7.00) {}; 
	\node[p4] at (16, 7.20) {}; 
	\node[p4] at (17, 7.80) {}; 
	\node[p4] at (18, 7.90) {}; 
	\node[p4] at (19, 7.95) {}; 
	\node[p4] at (20, 8.00) {}; 
	\end{pgfonlayer}
	\node[p5, label={[n]0:$(20,1,4,3,5,8)$}, below=2.5mm of L4] (L5) {};
	\begin{pgfonlayer}{bg}
	\node[p5] at (0, 5.00) {}; 
	\node[p5] at (1, 6.20) {}; 
	\node[p5] at (2, 6.20) {}; 
	\node[p5] at (3, 6.20) {}; 
	\node[p5] at (4, 6.30) {}; 
	\node[p5] at (5, 6.40) {}; 
	\node[p5] at (6, 6.40) {}; 
	\node[p5] at (7, 6.45) {}; 
	\node[p5] at (8, 6.50) {}; 
	\node[p5] at (9, 6.50) {}; 
	\node[p5] at (10, 6.53) {}; 
	\node[p5] at (11, 6.57) {}; 
	\node[p5] at (12, 6.60) {}; 
	\node[p5] at (13, 6.70) {}; 
	\node[p5] at (14, 6.80) {}; 
	\node[p5] at (15, 6.85) {}; 
	\node[p5] at (16, 6.90) {}; 
	\node[p5] at (17, 7.25) {}; 
	\node[p5] at (18, 7.40) {}; 
	\node[p5] at (19, 7.45) {}; 
	\node[p5] at (20, 7.80) {}; 
	\end{pgfonlayer}
\end{tikzpicture}
        \caption{Variation of the braking time $t^\br$}
        \label{fig:mini-pareto-br}
    \end{subfigure}
    \caption{Pareto fronts for various modifications of the artificial instance with four trains. The baseline scenario is always in black.}
    \label{fig:mini-pareto}
\end{figure}
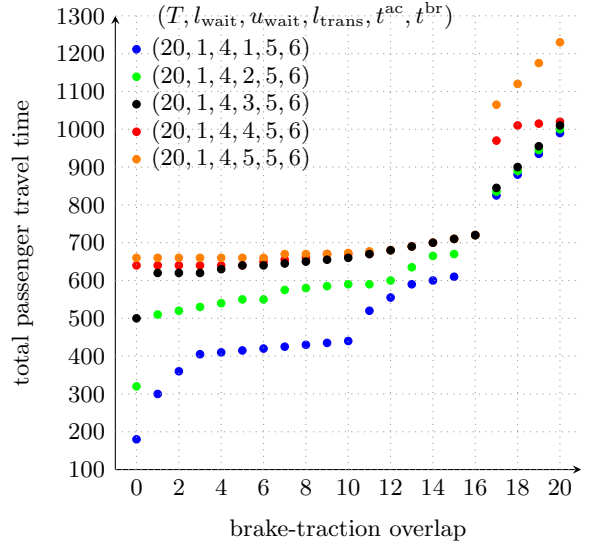
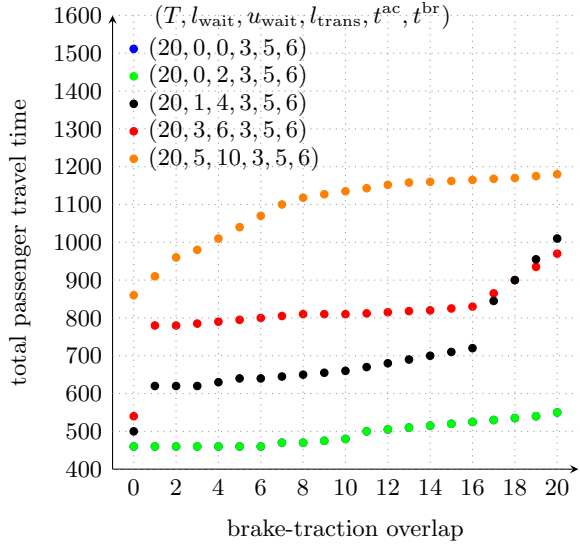
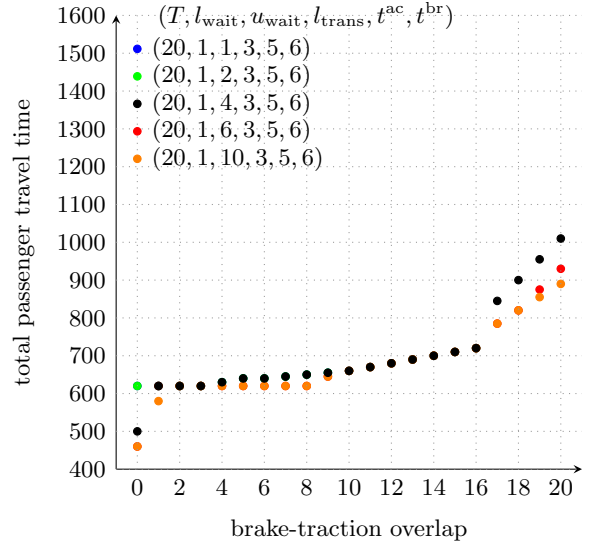
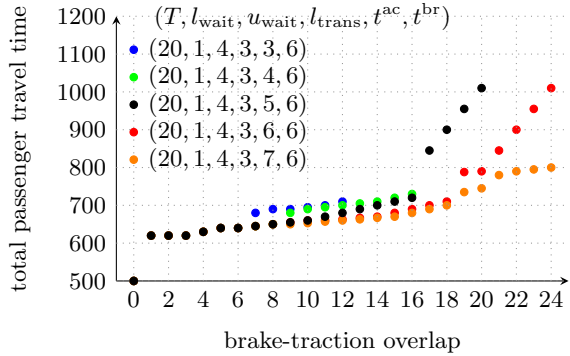
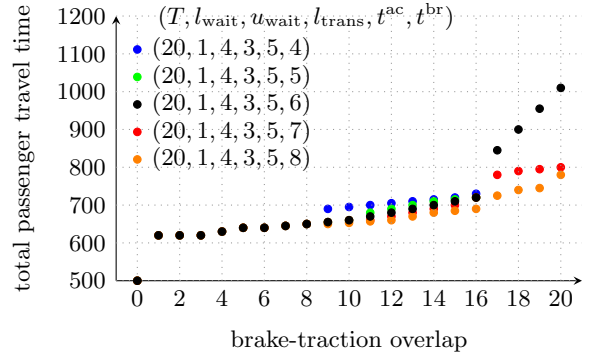

\subsection{Case Study 2: Berlin Ostkreuz with 14 Trains}

We now turn to a realistic larger case study. The Berlin Ostkreuz station is a major hub for local and regional trains, and has been the station with the fourth-highest passenger volume in Germany in 2024 \cite{ostkreuz_anfrage}. During weekend nights, six suburban railway lines of S-Bahn Berlin stop at Ostkreuz, offering convenient passenger transfers between trains using the circle line (S41/S42, S8) and the main west-east corridor (S3, S5, S7, S75), see \cref{fig:ostkreuz-timetable}. These use a 750\,V DC system~\cite{s-bahn}, meaning that synchronizing the braking and acceleration times can lead to significant energy savings. Although the timetable has certain symmetries with several trains meeting every 15 minutes, it is not a pure integrated fixed-interval timetable -- which cannot be realized with the available infrastructure anyway, see \cref{fig:ostkreuz-infrastructure}.

\paragraph{Basic instance parameters.}
Out of the six lines, five are operated with a period time of 30 minutes. Only the circle line S41/S42 runs every 15 minutes. For planning purposes, the S-Bahn Berlin timetable is computed with a resolution of 0.1 minutes, so that we consider a \pesppassenergy instance on a one-station network $\mathcal E_{14}$ with 14 trains and a period time of $T = 300$ for the sake of working with integer timestamps. As bounds for the waiting activities, we set a uniform lower bound of 5 time units (0.5 minutes), and for each train, we round the waiting time of the annual timetable up to the next multiple of 5 for the upper bound.
Around Ostkreuz, the maximum allowed speed is 80 km/h. Rounding up the times found with an interactive calculator tool \cite{golling_bahntechnik}, we set a uniform acceleration time of 0.4 minutes ($t^\ac = 4)$, and a braking time of 0.5 minutes ($t^\br = 5$). As an upper bound $W$ for the possible brake-traction overlap, we obtain $W = 14 \cdot 4 = 56$ time units (cf.\ \cref{prop:upperbound-matching}). As all acceleration times are equal, we expect in view of \cref{prop:specialcase1} a timetable that realizes an overlap of 56.

\paragraph{Railway timetabling requirements.}
However, we need to deal with additional railway timetabling constraints: For the lines S41 and S42, we impose a strict headway of 15.0 minutes between two subsequent arrivals and departures, respectively. This can be dealt with by additional headway activities, see, e.g., Liebchen and Möhring \cite{liebchen_modeling_2007}. Furthermore, we want to ensure that no track is occupied by two trains at the same time, and demand that a time of 1.7 minutes has to pass between a departure and the next arrival for any piece of infrastructure. Finally, while the four-track layout extends further west, the lines S5, S7, S75 need to share a standard double-track section towards the east, and we require every pair of two distinct trains to be at least 2.5 minutes apart here. These values are in accordance with the actual planning parameters at the infrastructure manager, and the track occupation constraints can be integrated into PESP using the methods of Masing et al.~\cite{masing_periodic_2023}. To incorporate the flexibility of the two platforms in the lower level, we use the IPESPC constraints introduced by Bortoletto et al.~\cite{bortoletto_periodic_2024}.

\paragraph{Transfer weights.}
Concerning the passenger flow, we introduce transfer activities with a minimum transfer time of 2.0 minutes only between the upper and the lower level, and between S3 and S5/S7/S75 from and to eastbound services. Other transfer relations can be realized at other stops in the network as well, and we will disregard them here. To obtain meaningful activity weights $w$ in the absence of data, we proceed as follows. Starting from the overall passenger load on a weekday \cite{nvp_2019}, we obtain demands of 147, 187, 184, 163 thousand passengers (kpax) for the directions north, east, south, west, respectively.  We stick to kpax as unit, being well aware that the demand in our after-midnight scenario is much lower, but still believing in similarly distributed passenger flows. As an adjustment, we assume that the demand splits 80:20 in favor of the outbound east to west direction, while we assume a 50:50 split for the circle line in the north and south directions. E.g, we assume that $187 \cdot 80\,\% = 149.6$ kpax want to travel eastbound, and only $163 \cdot 20\% = 32.6$ kpax westbound. Moreover, we suppose that the demand on the lines in the same direction is distributed as on their branching points in \cite{nvp_2019}. E.g., as further south of Ostkreuz, 57 kpax passengers use the direction of Grünau (S8), and 124 kpax use the circle line (S41/S42) we derive that a share of $\frac{57}{57+124} \approx 31.5\,\%$ of the total southbound demand is attributed to line S8. For simplicity, we assume that the transfers from or to a specific direction of a line are distributed in the same way, e.g., from all the passengers that transfer from eastbound line S5, $31.5\%$ will choose southbound S8. As it is unrealistic to assume that all passengers change, we require that 50\% of passengers in each train will stay on board. Lastly, to ensure feasibility, but also noting that the Ostkreuz area is a popular nightlife spot, we allow passengers to enter the system here. All these requirements can be modeled in a straightforward manner as a network-flow-type linear program. To obtain somewhat balanced activity weights, we choose to maximize the minimum number of passengers on a transfer activity. The linear program can be solved instantly with an optimal value of approximately 0.4 kpax. Out of the 347.7 total kpax, only 14.4 kpax enter at Ostkreuz. The highest weight of 7.2 kpax is associated to the transfer from eastbound S5, the strongest west-east line, to each of the two trains of southbound S41.

\paragraph{Computational setup.}
Having set up the details of the \pesppassenergy Ostkreuz instance, we solve the program \eqref{mip:obj1}-\eqref{mip:pesp5} with an $\varepsilon$-constraint method. We use Gurobi 12 \cite{gurobi} as a MIP solver on a high-throughput cluster with Intel Xeon Gold 6342 CPUs and 512 GB RAM, allowing Gurobi to run exclusively and in parallel on 32 threads. To further accelerate the solution process, we invoke Gurobi with best bound emphasis and aggressive cut generation, and additionally separate flip inequalities as user cuts (cf.\ Lindner and Masing \cite{lindner_split_2025}).

\paragraph{Results.}
It turns out that, also with all additional constraints, there is indeed a feasible timetable that realizes the maximum possible overlap of 56 time units. For this overlap, the minimum total travel time of the passengers is 25287.14 (see~\cref{fig:ostkreuz-timetable-overlap-56}), while neglecting energy considerations brings down the travel time to 21475.17 (see~\cref{fig:ostkreuz-timetable-overlap-0}), which is approximately 15\,\% less. For comparison, the timetable in operation (cf.\ \cref{fig:ostkreuz-timetable}) offers a travel time of 29840.55 with an overlap of 3, but is of course not targeted towards our choice of activity weights, and reflects constraints coming from other parts of the network. The full Pareto front is depicted in \cref{fig:ostkreuz-pareto}, along with computation times. The Pareto-efficient points are arranged in groups of four, which can be explained by the relatively short acceleration/braking times: Shifting the overlap on an energy arc by up to 4 time units still seems to allow to maintain the transfers, so that the picture is governed by the number of energy arcs with positive overlap rather than the actual value. For example, moving from overlap 52 to 53 means to have 14 energy arcs with positive overlap instead of 13, and causes a travel time increase of 1164.48, while going from overlap 53 to 54 requires only 30.39 more travel time. On the other hand, following the Pareto principle, significant overlaps can already be obtained with relatively small increases in travel time. E.g., 50\,\% of the maximum possible overlap need less than 4\,\% more travel time. Concerning computation times, the computational effort for the travel time minimization with an overlap of 32 or less is very manageable with wall times less than 1 hour, and even less than 10 minutes for an overlap less than 20. This changes dramatically with higher overlaps, the longest time needed being almost 17 hours to compute an optimal timetable with an overlap of 42.


\begin{center}
    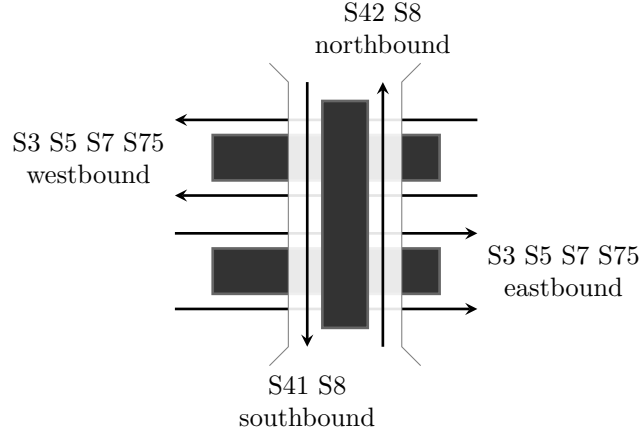
\begin{figure}
        \centering
         \begin{tikzpicture}[scale=0.5]
            \tikzstyle{track} = [draw, line width=1]
            \tikzstyle{lin} = [align=center,font=\small]
            \tikzstyle{bridge} = [draw, black!50]
            \def\platform at (#1,#2)#3[#4]{\draw[black!60, fill=black!80, line width=1] (#1-3,#2-0.6) rectangle ++(6,1.2);\node[white,font=\footnotesize,align=center] at (#1,#2) {#4};}
            \def\platformv at (#1,#2)#3[#4]{\draw[black!60, fill=black!80, line width=1] (#1-0.6,#2-3) rectangle ++(1.2,6);\node[white,font=\footnotesize,align=center] at (#1,#2) {#4};}
            \platform at (4, 4) [];
            \platform at (4, 1) [];
            \coordinate (A1) at (0, 0);
            \coordinate (E1) at (8, 0);
            \coordinate (A2) at (0, 2);
            \coordinate (E2) at (8, 2);
            \coordinate (A3) at (0, 3);
            \coordinate (E3) at (8, 3);
            \coordinate (A4) at (0, 5);
            \coordinate (E4) at (8, 5);
            \coordinate (A5) at (3, 6);
            \coordinate (E5) at (3, -1);
            \coordinate (A6) at (6, 6);
            \coordinate (E6) at (6, -1);
            \draw[track, ->] (A1) -- (E1);
            \draw[track, ->] (A2) -- (E2);
            \draw[track, <-] (A3) -- (E3);
            \draw[track, <-] (A4) -- (E4);
            \fill[fill=white, fill opacity=0.9] (A5) rectangle (E6);
            \platformv at (4.5, 2.5) [];
            \draw[track, ->] ($(A5)+(0.5,0)$) -- ($(E5)+(0.5,0)$);
            \draw[track, <-] ($(A6)-(0.5,0)$) -- ($(E6)-(0.5,0)$);
            \draw[bridge] ($(A5)+(-0.5,0.5)$) -- (A5) -- (E5) -- ($(E5)+(-0.5,-0.5)$);
            \draw[bridge] ($(A6)+(0.5,0.5)$) -- (A6) -- (E6) -- ($(E6)+(0.5,-0.5)$);
            \node[lin, anchor=west] at ($ (E1)!0.5!(E2) $) {S3 S5 S7 S75\\eastbound};
            \node[lin, anchor=east] at ($ (A3)!0.5!(A4) $) {S3 S5 S7 S75\\westbound};
            \node[lin, anchor=north] at ($(E5)+(0.5,-0.5)$) {S41 S8\\southbound};
            \node[lin, anchor=south] at ($(A6)+(-0.5,0.5)$) {S42 S8\\northbound};
        \end{tikzpicture}
        \caption{Infrastructure layout of the S-Bahn Berlin part of the Ostkreuz station with two platforms on the lower level and one platform on the upper level.}
        \label{fig:ostkreuz-infrastructure}
    \end{figure}
\end{center}

\begin{center}
    \begin{figure}
        \centering
         \begin{tikzpicture}[xscale=.3,yscale=.5]
            \def\d{0.25}
            \tikzstyle{n} = [font=\footnotesize]
            \tikzstyle{r} = [anchor=north, font=\scriptsize, rotate=90, align=center]
            \node[n, anchor=north] at (0,0) {0}; 
            \foreach \i in {5,10,...,30} {
                \draw[dotted, black!40] (\i,0) -- (\i,12);
                \node[n, anchor=north] at (\i,0) {\i}; 
            }
            \draw (0,0) -- (0,12.5);
            \draw[->] (0,0) -- (31,0) node[n, anchor=west] {time [min]};
            \draw[dashed] (-16,10.5) -- (30, 10.5);
            \draw[dashed] (-16,8.5) -- (30, 8.5);
            \draw[dashed] (-16,4.5) -- (30, 4.5);
            \node[r] at (-16,11.5) {north-\\bound};
            \node[r] at (-16,9.5) {south-\\bound};
            \node[r] at (-16,6.5) {west-\\bound};
            \node[r] at (-16,2.5) {east-\\bound};
            \okoverlap[0.1]{7}{10}{17.5}
            \okoverlap[0.2]{3}{6}{14.4}
            \oktraine{12}{12.9}{17.2}{S42}{Ring $\circlearrowleft$}
            \oktrainx{12}{27.4}{27.9}{BurntOrange!50}
            \oktrainx{12}{27.9}{30.0}{S42}
            \oktrainx{12}{0.0}{2.2}{S42}
            \oktrainx{12}{2.2}{2.6}{ForestGreen!50}
            \oktraine{11}{9.3}{9.9}{S8}{Pankow}
            \oktraine{10}{18.0}{18.5}{S8}{Grünau}
            \oktraine{9}{13.6}{16.2}{S41}{Ring $\circlearrowright$}
            \oktrainx{9}{28.1}{28.6}{BurntOrange!50}
            \oktrainx{9}{28.6}{30.0}{S41}
            \oktrainx{9}{0.0}{0.4}{S41}
            \oktrainx{9}{0.4}{0.8}{ForestGreen!50}
            \oktrainx{8}{27.4}{27.9}{BurntOrange!50}
            \oktrain{8}{27.9}{30.0}{S75}{Warschauer Str.}
            \oktrainx{8}{0.0}{2.2}{S75}
            \oktrainx{8}{2.2}{2.7}{ForestGreen!50}
            \oktraine{7}{9.9}{17.2}{S7}{Potsdam Hbf}
            \oktraine{6}{12.9}{14.2}{S5}{Warschauer Str.}
            \oktraine{5}{23.1}{25.2}{S3}{Ostbahnhof}
            \oktrainx{4}{29.9}{30.0}{BurntOrange!50}
            \oktrainx{4}{0.0}{0.4}{BurntOrange!50}
            \oktrain{4}{0.4}{2.2}{S3}{Erkner}
            \oktrainx{4}{2.2}{2.6}{ForestGreen!50}
            \oktraine{3}{14.9}{16.2}{S5}{Strausberg Nord}
            \oktraine{2}{12.9}{19.2}{S7}{Ahrensfelde}
            \oktrainx{1}{27.4}{27.9}{BurntOrange!50}
            \oktrain{1}{27.9}{30.0}{S75}{Wartenberg}
            \oktrainx{1}{0.0}{1.2}{S75}{Wartenberg}
            \oktrainx{1}{1.2}{1.6}{ForestGreen!50}
		\end{tikzpicture}
        \caption{Arrivals and departures at Ostkreuz in weekend nights according to the annual timetable for 2025 with 14 trains within 30 minutes. Only line S41/S42 runs more than within this time frame. With an acceleration time of $0.4$ minutes (green) and a braking time of $0.5$ minutes (orange), only two energy arcs can be used to transfer energy with a total overlap of $0.3$ minutes (gray). Source: DB InfraGO AG.}
        \label{fig:ostkreuz-timetable}
    \end{figure}
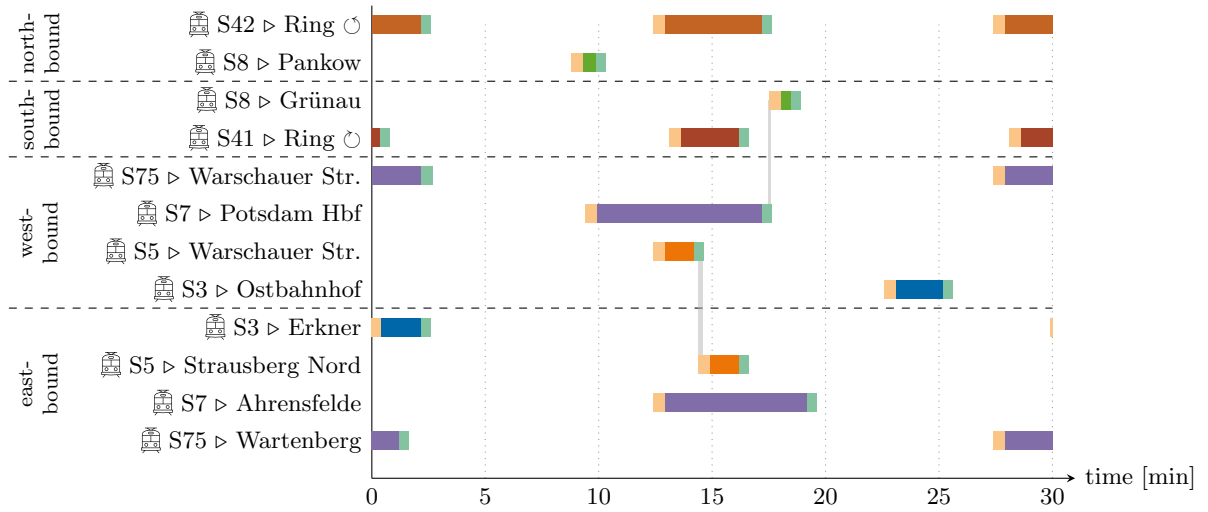
\end{center}

\begin{center}
    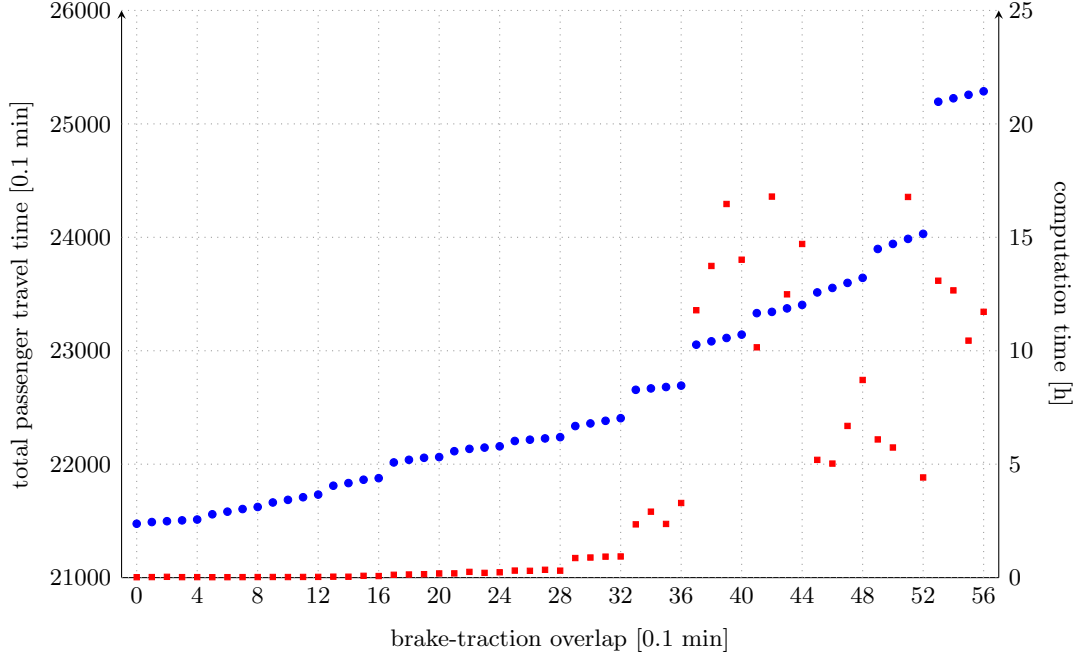
\begin{figure}
        \centering
                \begin{tikzpicture}[x=2mm,y=1.5mm]
            \tikzstyle{n} = [font=\footnotesize]
            \tikzstyle{p} = [circle, draw, blue, fill=blue, inner sep=1]
            \tikzstyle{q} = [rectangle, draw, red, fill=red, inner sep=1]
            \def\xmin{-1}
            \def\xmax{57}
            \def\ymin{210}
            \def\ystep{220}
            \def\ymax{260}
            \draw[->] (\xmin,\ymin) -- (\xmin,\ymax);
            \draw[-] (\xmin,\ymin) -- (\xmax,\ymin);
            \draw[->] (\xmax,\ymin) -- (\xmax,\ymax);
            \node[n, anchor=north, yshift=-15] at  ($({(\xmin+\xmax)/2},\ymin)$) {brake-traction overlap [0.1 min]};
            \node[n, rotate=90, anchor=south, yshift=35] at ($(0,{(\ymin+\ymax)/2})$) {total passenger travel time [0.1 min]};
            \node[n, rotate=-90, anchor=south, yshift=15] at ($(\xmax,{(\ymin+\ymax)/2})$) {computation time [h]};
            \foreach \i in {0,4,...,56} {
                \draw[dotted, black!40] (\i,\ymin) -- (\i,\ymax);
                \node[n, anchor=north] at (\i,\ymin) {\i}; 
            }
            \foreach \i [evaluate=\i as \lab using int(\i)] in {\ymin,\ystep,...,\ymax} {
                \draw[dotted, black!40] (\xmin,\i) -- (\xmax,\i);
                \node[n, anchor=east] at (\xmin,\i) {\i00};
            }
            \foreach \i [evaluate=\i as \lab using int((\i - 210)/2)] in {\ymin,\ystep,...,\ymax} {
                \node[n, anchor=west] at (\xmax,\i) {\lab};
            }
\node[p] at (0, 214.7517) {}; 
\node[p] at (1, 214.8995) {}; 
\node[p] at (2, 214.9734) {}; 
\node[p] at (3, 215.0473) {}; 
\node[p] at (4, 215.1212) {}; 
\node[p] at (5, 215.5840) {}; 
\node[p] at (6, 215.8154) {}; 
\node[p] at (7, 216.0468) {}; 
\node[p] at (8, 216.2297) {}; 
\node[p] at (9, 216.6245) {}; 
\node[p] at (10, 216.8559) {}; 
\node[p] at (11, 217.0873) {}; 
\node[p] at (12, 217.3187) {}; 
\node[p] at (13, 218.1025) {}; 
\node[p] at (14, 218.3339) {}; 
\node[p] at (15, 218.6313) {}; 
\node[p] at (16, 218.7675) {}; 
\node[p] at (17, 220.1545) {}; 
\node[p] at (18, 220.3859) {}; 
\node[p] at (19, 220.5639) {}; 
\node[p] at (20, 220.6277) {}; 
\node[p] at (21, 221.1450) {}; 
\node[p] at (22, 221.3510) {}; 
\node[p] at (23, 221.4635) {}; 
\node[p] at (24, 221.5761) {}; 
\node[p] at (25, 222.0466) {}; 
\node[p] at (26, 222.1592) {}; 
\node[p] at (27, 222.2718) {}; 
\node[p] at (28, 222.3843) {}; 
\node[p] at (29, 223.3621) {}; 
\node[p] at (30, 223.5930) {}; 
\node[p] at (31, 223.8244) {}; 
\node[p] at (32, 224.0548) {}; 
\node[p] at (33, 226.5574) {}; 
\node[p] at (34, 226.6803) {}; 
\node[p] at (35, 226.8038) {}; 
\node[p] at (36, 226.9267) {}; 
\node[p] at (37, 230.5355) {}; 
\node[p] at (38, 230.8311) {}; 
\node[p] at (39, 231.1267) {}; 
\node[p] at (40, 231.4223) {}; 
\node[p] at (41, 233.3161) {}; 
\node[p] at (42, 233.4275) {}; 
\node[p] at (43, 233.7351) {}; 
\node[p] at (44, 234.0427) {}; 
\node[p] at (45, 235.1422) {}; 
\node[p] at (46, 235.5378) {}; 
\node[p] at (47, 235.9824) {}; 
\node[p] at (48, 236.4270) {}; 
\node[p] at (49, 238.9736) {}; 
\node[p] at (50, 239.4182) {}; 
\node[p] at (51, 239.8628) {}; 
\node[p] at (52, 240.3074) {}; 
\node[p] at (53, 251.9522) {}; 
\node[p] at (54, 252.2561) {}; 
\node[p] at (55, 252.5635) {}; 
\node[p] at (56, 252.8714) {}; 
\node[q] at (0, 210.0372) {}; 
\node[q] at (1, 210.0475) {}; 
\node[q] at (2, 210.0702) {}; 
\node[q] at (3, 210.0434) {}; 
\node[q] at (4, 210.0456) {}; 
\node[q] at (5, 210.0391) {}; 
\node[q] at (6, 210.0409) {}; 
\node[q] at (7, 210.0454) {}; 
\node[q] at (8, 210.0508) {}; 
\node[q] at (9, 210.0563) {}; 
\node[q] at (10, 210.0521) {}; 
\node[q] at (11, 210.0594) {}; 
\node[q] at (12, 210.0639) {}; 
\node[q] at (13, 210.0848) {}; 
\node[q] at (14, 210.0804) {}; 
\node[q] at (15, 210.1585) {}; 
\node[q] at (16, 210.1392) {}; 
\node[q] at (17, 210.2497) {}; 
\node[q] at (18, 210.2753) {}; 
\node[q] at (19, 210.3051) {}; 
\node[q] at (20, 210.3673) {}; 
\node[q] at (21, 210.3738) {}; 
\node[q] at (22, 210.5063) {}; 
\node[q] at (23, 210.4189) {}; 
\node[q] at (24, 210.4682) {}; 
\node[q] at (25, 210.6205) {}; 
\node[q] at (26, 210.6007) {}; 
\node[q] at (27, 210.6912) {}; 
\node[q] at (28, 210.6150) {}; 
\node[q] at (29, 211.7269) {}; 
\node[q] at (30, 211.7725) {}; 
\node[q] at (31, 211.8505) {}; 
\node[q] at (32, 211.8655) {}; 
\node[q] at (33, 214.6950) {}; 
\node[q] at (34, 215.8154) {}; 
\node[q] at (35, 214.7333) {}; 
\node[q] at (36, 216.5801) {}; 
\node[q] at (37, 233.5717) {}; 
\node[q] at (38, 237.4751) {}; 
\node[q] at (39, 242.9380) {}; 
\node[q] at (40, 238.0266) {}; 
\node[q] at (41, 230.3009) {}; 
\node[q] at (42, 243.5915) {}; 
\node[q] at (43, 234.9736) {}; 
\node[q] at (44, 239.4139) {}; 
\node[q] at (45, 220.3818) {}; 
\node[q] at (46, 220.0534) {}; 
\node[q] at (47, 223.3728) {}; 
\node[q] at (48, 227.4237) {}; 
\node[q] at (49, 222.1825) {}; 
\node[q] at (50, 221.4710) {}; 
\node[q] at (51, 243.5590) {}; 
\node[q] at (52, 218.8292) {}; 
\node[q] at (53, 236.1749) {}; 
\node[q] at (54, 235.3218) {}; 
\node[q] at (55, 230.8932) {}; 
\node[q] at (56, 233.4337) {}; 
\end{tikzpicture}   
        \caption{Pareto front (blue circles) and computation times (wall times, red squares) for the Ostkreuz instance}
        \label{fig:ostkreuz-pareto}
    \end{figure}
\end{center}

\begin{center}
    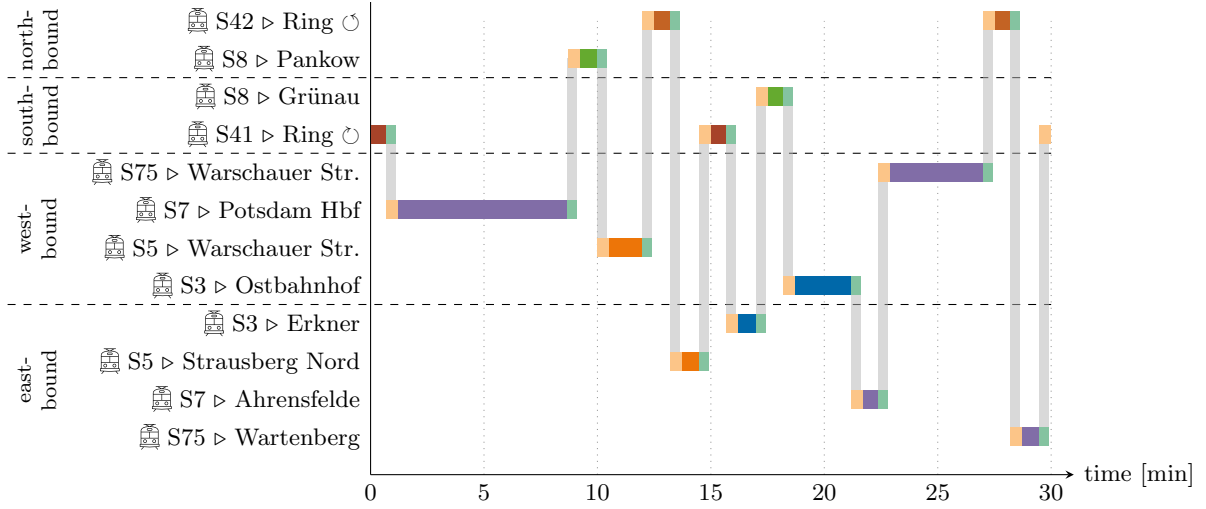
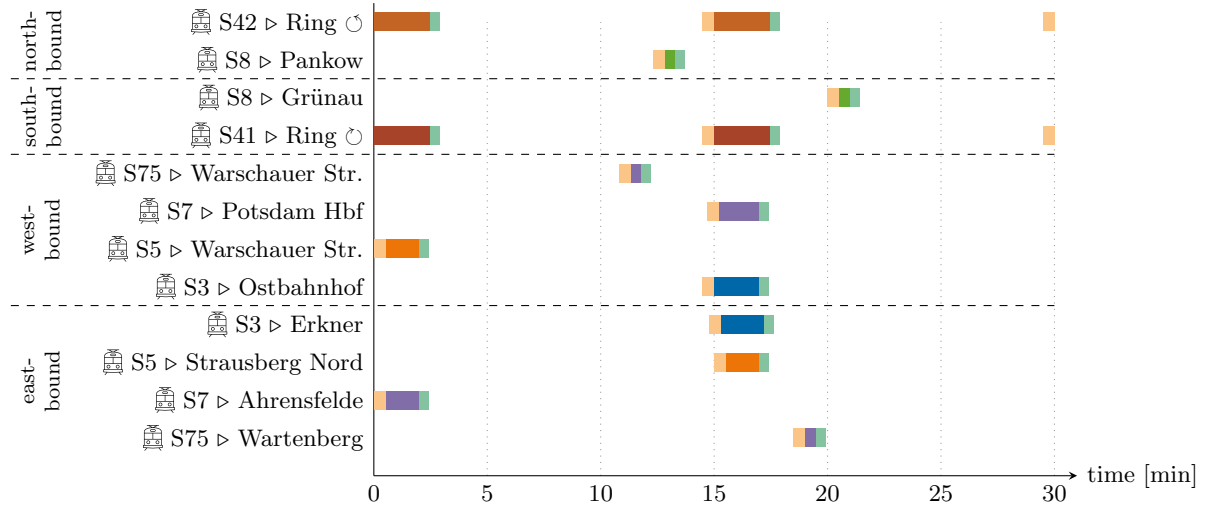
\begin{figure}
        \centering
    \begin{subfigure}{\linewidth}
         \begin{tikzpicture}[xscale=.3,yscale=.5]
            \def\d{0.25}
            \tikzstyle{n} = [font=\footnotesize]
            \tikzstyle{r} = [anchor=north, font=\scriptsize, rotate=90, align=center]
            \node[n, anchor=north] at (0,0) {0}; 
            \foreach \i in {5,10,...,30} {
                \draw[dotted, black!40] (\i,0) -- (\i,12);
                \node[n, anchor=north] at (\i,0) {\i}; 
            }
            \draw (0,0) -- (0,12.5);
            \draw[->] (0,0) -- (31,0) node[n, anchor=west] {time [min]};
            \draw[dashed] (-16,10.5) -- (30, 10.5);
            \draw[dashed] (-16,8.5) -- (30, 8.5);
            \draw[dashed] (-16,4.5) -- (30, 4.5);
            \node[r] at (-16,11.5) {north-\\bound};
            \node[r] at (-16,9.5) {south-\\bound};
            \node[r] at (-16,6.5) {west-\\bound};
            \node[r] at (-16,2.5) {east-\\bound};
            \okoverlap{7}{9}{0.7}
            \okoverlap{4}{9}{15.7}
            \okoverlap{5}{10}{18.2}
            \okoverlap{1}{12}{28.2}
            \okoverlap{3}{12}{13.2}
            \okoverlap{6}{11}{10.0}
            \okoverlap{7}{11}{8.7}
            \okoverlap{6}{12}{12.0}
            \okoverlap{2}{5}{21.2}
            \okoverlap{8}{12}{27.0}
            \okoverlap{1}{9}{29.5}
            \okoverlap{4}{10}{17.0}
            \okoverlap{3}{9}{14.5}
            \okoverlap{2}{8}{22.4}
            \oktraine{12}{12.5}{13.2}{S42}{Ring $\circlearrowleft$}
            \oktrainex{12}{27.5}{28.2}{S42}
            \oktraine{11}{9.2}{10.0}{S8}{Pankow}
            \oktraine{10}{17.5}{18.2}{S8}{Grünau}
            \oktrain{9}{0.0}{0.7}{S41}{Ring $\circlearrowright$}
            \oktrainx{9}{29.5}{30.0}{BurntOrange!50}
            \oktrainx{9}{0.7}{1.1}{ForestGreen!50}
            \oktrainex{9}{15.0}{15.7}{S41}            
            \oktraine{8}{22.9}{27.0}{S75}{Warschauer Str.}
            \oktraine{7}{1.2}{8.7}{S7}{Potsdam Hbf}
            \oktraine{6}{10.5}{12.0}{S5}{Warschauer Str.}
            \oktraine{5}{18.7}{21.2}{S3}{Ostbahnhof}
            \oktraine{4}{16.2}{17.0}{S3}{Erkner}
            \oktraine{3}{13.7}{14.5}{S5}{Strausberg Nord}
            \oktraine{2}{21.7}{22.4}{S7}{Ahrensfelde}
            \oktraine{1}{28.7}{29.5}{S75}{Wartenberg}
		\end{tikzpicture}
        \caption{Minimum travel time timetable with overlap 56. The matching of energy arcs induces a Hamiltonian cycle.}
        \label{fig:ostkreuz-timetable-overlap-56}
    \end{subfigure}
    
    \vspace{1em}
    \begin{subfigure}{\linewidth}
        \centering
        \begin{tikzpicture}[xscale=.3,yscale=.5]
            \def\d{0.25}
            \tikzstyle{n} = [font=\footnotesize]
            \tikzstyle{r} = [anchor=north, font=\scriptsize, rotate=90, align=center]
            \node[n, anchor=north] at (0,0) {0}; 
            \foreach \i in {5,10,...,30} {
                \draw[dotted, black!40] (\i,0) -- (\i,12);
                \node[n, anchor=north] at (\i,0) {\i}; 
            }
            \draw (0,0) -- (0,12.5);
            \draw[->] (0,0) -- (31,0) node[n, anchor=west] {time [min]};
            \draw[dashed] (-16,10.5) -- (30, 10.5);
            \draw[dashed] (-16,8.5) -- (30, 8.5);
            \draw[dashed] (-16,4.5) -- (30, 4.5);
            \node[r] at (-16,11.5) {north-\\bound};
            \node[r] at (-16,9.5) {south-\\bound};
            \node[r] at (-16,6.5) {west-\\bound};
            \node[r] at (-16,2.5) {east-\\bound};
            
            \oktrain{12}{0.0}{2.5}{S42}{Ring $\circlearrowleft$}
            \oktrainx{12}{29.5}{30.0}{BurntOrange!50}
            \oktrainx{12}{2.5}{2.9}{ForestGreen!50}
            \oktrainex{12}{15.0}{17.5}{S42}
            \oktraine{11}{12.8}{13.3}{S8}{Pankow}
            \oktraine{10}{20.5}{21.0}{S8}{Grünau}
            \oktrain{9}{0.0}{2.5}{S41}{Ring $\circlearrowright$}
            \oktrainx{9}{29.5}{30.0}{BurntOrange!50}
            \oktrainx{9}{2.5}{2.9}{ForestGreen!50}
            \oktrainex{9}{15.0}{17.5}{S41}            
            \oktraine{8}{11.3}{11.8}{S75}{Warschauer Str.}
            \oktraine{7}{15.2}{17.0}{S7}{Potsdam Hbf}
            \oktraine{6}{0.5}{2.0}{S5}{Warschauer Str.}
            \oktraine{5}{15.0}{17.0}{S3}{Ostbahnhof}
            \oktraine{4}{15.3}{17.2}{S3}{Erkner}
            \oktraine{3}{15.5}{17.0}{S5}{Strausberg Nord}
            \oktraine{2}{0.5}{2.0}{S7}{Ahrensfelde}
            \oktraine{1}{19.0}{19.5}{S75}{Wartenberg}
		\end{tikzpicture}
        \caption{Minimum travel time timetable with overlap 0. Although not explicitly required, this timetable turns out to be very symmetric and resembles an integrated fixed-interval timetable.}
        \label{fig:ostkreuz-timetable-overlap-0}
    \end{subfigure}
        \caption{Two Pareto-optimal timetables for the Ostkreuz instance, showing braking (orange), accelerating (green), and the matching of energy arcs with the resulting overlap (gray).}
        \label{fig:ostkreuz-timetable-pareto}
    \end{figure}
\end{center}

\section{Conclusion and Outlook} \label{sec:outlook}
We have introduced the \pesppassenergy problem that integrates the maximization of the brake-traction overlap time into the context of periodic railway timetabling by means of the Periodic Event Scheduling Problem. Apart from giving a MIP formulation, we characterized the structure of optimal solutions for both single objective problems on a \osn, and investigated the complexity landscape. On such networks, we obtained a polynomial-time algorithm with an additive performance guarantee for the \pespenergy problem, and exact polynomial-time algorithms in several special cases. Finally, we have presented two case studies. With our model, the computation of the Pareto front has been feasible for a real-world 14 train instance. Besides closing the remaining gap in the complexity status, the next step is to apply our model to network-scale instances. An integration with the task of finding timetables that allow energy-efficient driving is also a topic for future research.

\subsection*{Acknowledgements}

Sarah Roth and Sven Jäger were funded by the German Federal Ministry of Education and Research (BMBF), project number 05M22UKB (SynphOnie).

Niels Lindner: The work for this article has been conducted in the Research Campus MODAL funded by the German Federal Ministry of Research, Technology and Space (BMFTR) (fund number 05M25KEM).

\subsection*{Author Contributions}

\textit{Sarah~Roth:} Conceptualization, Methodology, Investigation, Software, Formal analysis, Writing -- Original Draft, Writing -- Review \& Editing, Visualization;
\textit{Sven~Jäger:} Conceptualization, Methodology, Investigation, Formal analysis, Writing -- Original Draft, Writing -- Review \& Editing, Visualization;
\textit{Niels~Lindner:} Methodology, Investigation, Software, Validation, Formal analysis, Writing -- Original Draft, Writing -- Review \& Editing, Visualization;
\textit{Anita~Schöbel:} Conceptualization, Methodology, Writing -- Review \& Editing, Supervision, Funding acquisition

\hypersetup{hidelinks}
\bibliographystyle{alphasven}
\bibliography{references}

\newpage
\section*{List of Symbols}

\begin{longtable}{p{.11\linewidth}p{.83\linewidth}}%
        \toprule
        Symbol & Description \\ \midrule
        $A$ & activity set \\
        $A_{\drive}$ & set of driving activities \\
        $A_{\wait}$ & set of waiting activities \\
        $A_{\transfer}$ & set of transfer activities \\
        $A_{\energy}$ & set of possible energy activities \\
        $a$ & an activity \\
        $\alpha$ & matching indicator variable in the MIP model \\
        $C$ & a cycle in $M \dotcup A_\wait$ for a matching $M \subseteq A_\energy$ \\
        $\delta_C$ & distance of the full overlap interval of cycle $C$ to nearest multiple of $T$\\
        $\cE$ & event-activity network \\
        $\cE_n$ & \osn with $n$ lines \\
        $\cE_n^{\mathrm{pass}}$ & \osn with $n$ lines for \pesppassenger, without energy arcs \\
        $\cE_n^{\energy}$ & \osn with $n$ lines for \pespenergy, without transfer arcs \\
        $E$ & event set \\
        $E_{\arr}$ & arrival event set \\
        $E_{\dep}$ & departure event set \\
        $G$ & graph with contracted waiting activities \\
        $\Gamma$ & large constant in the MIP model \\
        $K_{n,n}$ & \osn with all possible transfer or energy activities \\
        $\cL$ & set of lines \\
        $L_C$ & sum of waiting lower bounds and $t^{\min}$ on energy activities for a cycle $C$ \\
        $\ell$ & a line \\
        $l$ & lower bound vector \\
        $l^{\max}$ & duration of all transfer and waiting activities in a Basel solution structure \\
        $l_i$ & train with largest brake/accelerate time in cycle~$i$ \\
        $M$ & matching of selected energy activities \\
        $M_{\text{greedy}}$ & greedy matching \\
        $n$ & number of lines/trains \\
        $o$ & brake-traction overlap vector \\
        $P$ & Hamiltonian path \\
        $p$ & vector of modulo parameters in the MIP model \\
        $\pi$ & timetable \\
        $\w$ & shorthand for $t^{\min}$ \\
        $\w_G$ & $\w$ interpreted as arc weights on contracted graph $G$ \\
        $S$ & solution $(M, \pi, x, o)$ \\
        $s_i$ & train with smallest brake/accelerate time in cycle~$i$ \\
        $\sigma$ & permutation of departure events sorting the acceleration times \\
        $T$ & period time \\
        $t^{\ac}$ & acceleration time vector (indexed by departure events) \\
        $t^{\br}$ & braking time vector (indexed by arrival events) \\
        $t^{\max}$ & maxima of acceleration and braking times \\
        $t^{\min}$ & minima of acceleration and braking times \\
        $u$ & upper bound vector \\
        $U_C$ & sum of waiting upper bounds and $t^{\max}$ on energy activities for a cycle $C$ \\
        $\mathcal V$ & set of stations \\
        $v$ & a station \\
        $w$ & vector of passenger numbers/activity weights \\
        $w'(k\ell)$ & $|t_k^{\ac} - t_\ell^{\br}|$ \\
        $\varphi$ & permutation of arriving lines sorting the braking times \\
        $x$ & periodic tension vector \\
        \bottomrule
    \end{longtable}

\end{document}